\documentclass[a4paper, 11pt]{article}
\usepackage{amssymb,amsmath,amsthm,lmodern,cancel}
\usepackage{amsfonts}

\usepackage[activate={true,nocompatibility},final,tracking=true,kerning=true,spacing=true,factor=1100,stretch=15,shrink=15]{microtype}
\microtypecontext{spacing=nonfrench}

\usepackage{graphicx}
\usepackage[english]{babel}
\usepackage[T1]{fontenc}
\usepackage{textcomp}
\usepackage[dvipsnames]{xcolor}
\usepackage{soul}
\usepackage[shortlabels]{enumitem}
\usepackage{thmtools, mathtools,verbatim} \usepackage{upgreek}
\usepackage{thm-restate}
\usepackage{titlesec} 
\usepackage[normalem]{ulem} 
\usepackage{mathrsfs} 

\usepackage[colorlinks=false,hyperindex=true,hypertexnames=false]{hyperref}
\usepackage{cleveref}

\usepackage[font=normal]{caption}

\usepackage{thmtools}

\usepackage[bottom]{footmisc}

\title{Large deviations of Dyson Brownian motion on the circle \\ and multiradial $\SLE_{0+}$}  
\date{}

\author{
Osama Abuzaid\thanks{Department of Mathematics and Systems Analysis, Aalto University, Finland.  \protect\url{osama.abuzaid@aalto.fi}} , \, 
Vivian Olsiewski Healey\thanks{Department of Mathematics, Texas State University, US. \protect\url{healey@txstate.edu} } , \,
and \,
Eveliina Peltola\thanks{Department of Mathematics and Systems Analysis, Aalto University, Finland; and \\ Institute for Applied Mathematics, University of Bonn, Germany.   \protect\url{eveliina.peltola@aalto.fi}}
}

\setlist[enumerate]{topsep = 1ex, leftmargin=1cm, itemsep= -2pt}

\let\OLDthebibliography\thebibliography
\renewcommand\thebibliography[1]{
  \OLDthebibliography{#1}
  \setlength{\parskip}{1pt}
  \setlength{\itemsep}{2pt}
}

\allowdisplaybreaks

\newtheorem{thm}{Theorem}[section]
\newtheorem{cor}[thm]{Corollary}

\newtheorem{lem}[thm]{Lemma}
\newtheorem{prop}[thm]{Proposition}

\newtheorem{theorem}{Theorem}

\newtheorem{lemA}[theorem]{Lemma}
\newtheorem{thmA}[theorem]{Theorem}
\newtheorem{proposition}[theorem]{Proposition}

\newtheorem{corollary}[theorem]{Corollary}

\theoremstyle{definition} 
\newtheorem{df}[thm]{Definition}
\newtheorem{ex}[thm]{Example}
\newtheorem{remark}[thm]{Remark}

\numberwithin{equation}{section}
\numberwithin{figure}{section}

\global\long\def\bR{\mathbb{R}}
\global\long\def\bZ{\mathbb{Z}}
\global\long\def\bN{\mathbb{N}}

\global\long\def\bC{\mathbb{C}}
\global\long\def\bH{\mathbb{H}}
\global\long\def\bD{\mathbb{D}}
\global\long\def\bP{\mathbb{P}}

\global\long\def\cL{\mathcal{L}}

\global\long\def\cC{\mathcal{C}}
\global\long\def\cT{\mathcal{T}}

\global\long\def\ud{\,\mathrm{d}}

\global\long\def\open{O}
\global\long\def\closed{F}

\global\long\def\variable{\, \cdot \,}

\global\long\def\ii{\mathrm{i}}

\renewcommand{\liminf}{\varliminf}
\renewcommand{\limsup}{\varlimsup}

\newcommand{\SLE}{\operatorname{SLE}}

\global\long\def\compsets{\cC}

\global\long\def\one{\scalebox{1.1}{\textnormal{1}} \hspace*{-.75mm} \scalebox{0.7}{\raisebox{.3em}{\bf |}} }

\newcommand{\supnorm}[1]{\lVert #1 \rVert_{[0,\bigT]}}
\newcommand{\metric}[2]{\mathrm{d}_{[0,\bigT]}(#1,#2)}
\newcommand{\metricsymb}{\mathrm{d}_{[0,\bigT]}}

\global\long\def\lenergy{I}

\def\nBessel{J} 
\def\nDenergy{E} 

\def\nin{\not \in}

\DeclareMathOperator*{\argmin}{arg\,min}

\def\btheta{\boldsymbol{\theta}}
\def\bvartheta{\boldsymbol{\vartheta}}
\def\bU{\mathbf{U}}

\def\bgamma{\boldsymbol{\gamma}}
\def\boldeta{\boldsymbol{\eta}}
\def\bomega{\boldsymbol{\omega}}

\def\blambda{\boldsymbol{\lambda}}

\newcommand{\bs}[1]{\boldsymbol{#1}}

\def\bB{\mathbf{B}}

\def\Dyson{U}
\def\bDyson{\mathbf{\Dyson}^\kappa}

\def\Bessel{\Theta}
\def\bBessel{\mathbf{\Bessel}}

\def\charge{\mathrm{c}}

\def\Borel{H}

\global\long\def\chamber{\cT_n}
\newcommand{\chamberCl}[1]{{\overline{\cT}\!{}_{n}\!\!{}^{#1}}}
\global\long\def\chamberSingle{\cT_1}

\def\Pr{\mathsf{P}}
\def\PrPush{\mathbb{P}}
\def\PrDyson{\Pr^\kappa}

\def\bigT{T} 

\def\prob{\mathbb P}
\def\probSLE{\prob^\kappa}

\def\bE{\mathbb E}
\def\spaceX{X}
\def\spaceY{Y}

\newcommand{\BesselSpace} [2]{C_{#2} \big ([0, #1, \chamber \big)}
\newcommand{\BesselSpaceSingle} [2]{C_{#2} \big ([0, #1, \chamberSingle \big)}
\newcommand{\BesselSpaceSmall} [2]{C_{#2} ([0, #1, \chamber)}

\newcommand{\CameronMartinSpace} [2]{C_{#2} \big ([0, #1, \bR \big)}
\newcommand{\nCameronMartinSpace} [2]{C_{#2} \big ([0, #1, \bR^n \big)}

\newcommand{\HoneSpace} [2]{H^1_{#2} \big ([0, #1, \bR \big)}

\newcommand{\nHoneSpace} [2]{H^1_{#2} \big ([0, #1, \bR^n \big)}
\newcommand{\nWonetwoSpace} [2]{W^{1,2}_{#2} \big ([0, #1, \bR^n \big)}

\newcommand{\nradial}[5]{ \partial_{#3} {#2}_{#3}(#4) = #1 \, #2_{#3}(#4) \sum_{j=1}^n \frac{#5_{#3}+ #2_{#3}(#4)} {#5_{#3} - #2_{#3}(#4)} }

\def\bz{\boldsymbol{z}}

\newcommand{\hullspace}[1]{\mathcal K_{#1}} 
\newcommand{\rmultichordspace}{\mathcal A}

\newcommand{\swallowtime}[1]{\tau_{#1}}

\newcommand{\swallowtimehat}[1]{\hat{\tau}_{#1}}
\global\long\def\detblowuptime{\tau_{\mathrm{coll}}}
\global\long\def\blowuptime{\tau_{\mathrm{coll}}}
\global\long\def\blowuptimegen{\tau}

\global\long\def\BLoop{\mu^{\mathrm{loop}}}

\global\long\def\integrand{\varrho}

\newcommand{\Dset}{D}
\newcommand{\meas}[1]{\nu(#1)}

\newcommand{\tildef}{\mathfrak f}
\newcommand{\tildeg}{\mathfrak g}
\newcommand{\tildeh}{\mathfrak h}
\newcommand{\tildep}{\mathfrak p}
\newcommand{\tildeK}{K^{\tildeg}}
\newcommand{\hatg}{g}
\newcommand{\hatw}{z}

\newcommand{\timechange}{ \sigma}

\newcommand{\nbhd}{\mathcal O} 
\newcommand{\bignbhd}{\mathcal V}
\newcommand{\startheta}{\theta^*}

\newcommand{\zprocess}{y}
\newcommand{\Zprocess}{Y}

\newcommand{\spr}{\mu}

\newcommand{\bphi}{\boldsymbol{\phi}}
\newcommand{\bspr}{\boldsymbol{\spr}}

\newcommand{\upbd}{R} 

\newcommand{\newphi}{\varphi}
\newcommand{\bnewphi}{\boldsymbol{\varphi}}
\newcommand{\pairderiv}{\hat \newphi}
\newcommand{\partfn}{\mathcal E}

\newcommand{\PF}{\mathbb P^{\partfn}}

\newcommand{\smoothing}{\varphi_\epsilon}
\newcommand{\bsmoothing}{\bs{\varphi}_\epsilon}
\newcommand{\upper}{R(\epsilon)}

\newcommand{\aux}{\lambda}
\newcommand{\baux}[1]{\blambda^{#1}}

\newcommand{\nradpartfn}{\mathcal Z_{\mathrm{rad}}}

\newcommand{\newphinrad}{\newphi_{\mathrm{rad}}}
\newcommand{\bnewphinrad}{\bnewphi_{\mathrm{rad}}}

\newcommand{\PhiGen}{\Phi^{\kappa,\partfn}}

\newcommand{\DPotential}{\mathcal U}
\newcommand{\Mob}{f}

\newcommand{\constphi}{\mathtt C} 
\newcommand{\constmult}{ \mathtt a} 
\newcommand{\constaddnew}{{\mathtt B}}

\newcommand{\lipconst}{L}
\newcommand{\jmin}{{j_t}}
\newcommand{\jminzero}{{j_0}}
\renewcommand{\j}{m} 
\newcommand{\minindexset}{\mathcal A}

\newcommand{\mgle}{M^\kappa}

\newcommand{\DeltaMin}{\delta}

\newcommand{\Pspiral}{\mathbb P^{\kappa, \spr}}

\newcommand{\two}{2}
\newcommand{\ONE}{1}

\begin{document}
\maketitle

\begin{abstract}
We show a finite-time large deviation principle (LDP) for ``Dyson type'' diffusion processes, including Dyson Brownian motion (DBM) on the circle, for a fixed number of particles as the coupling parameter $\beta=8/\kappa$ tends to $+\infty$. Zero-energy systems correspond to the Calogero-Moser-Sutherland integrable system. We also characterize the large-time behavior of finite-energy systems: zero-energy systems approach exponentially fast a static equally-spaced configuration, while finite-energy systems may have polynomial convergence rates, and the system may never become static.

We use our DBM result to derive a finite-time LDP in the Hausdorff metric for multiradial Schramm-Loewner evolution, $\SLE_\kappa$, as $\kappa$ tends to $0+$, with good rate function being the multiradial Loewner energy. Using a derivative estimate for the radial Loewner map in terms of the energy of its driving function, we show that finite-energy multiradial Loewner hulls are disjoint unions of simple curves, except at their common endpoint. 
\end{abstract}

\newpage

\tableofcontents

\newpage

\section{Introduction}

Our original motivation for the present work was to investigate the asymptotic behavior as $\kappa \to 0+$ of multiradial Schramm-Loewner evolution, $\SLE_\kappa$. 
Here, we show that this process with the common parameterization satisfies a finite-time large deviation principle (LDP) in the Hausdorff metric with good rate function, the multiradial Loewner energy. We also characterize the large-time behavior of curves with finite energy and zero energy, whose driving functions correspond to the Calogero-Moser-Sutherland integrable system.

However, major parts of this article (Sections~\ref{sec:Bessel_LDP}~\&~\ref{sec:zero-energy}) are of independent interest regardless of SLE theory. 
Indeed, \Cref{sec:Bessel_LDP} is devoted to proving a finite-time LDP (\Cref{thm:Bessel_LDP_general}) 
for a general class of diffusion processes that we call ``Dyson type'' (\Cref{def:Dyson_type_potential}),
including Dyson Brownian motion (DBM) on the circle, for a fixed number $n$ of particles as the coupling parameter $\beta=8/\kappa$ tends to $+\infty$. 
To our knowledge, in the literature large deviations of DBM has only been considered for fixed $\beta$ and as $n$ tends to $+\infty$. 
While the non-Lipschitz drift precludes the application of the Freidlin-Wentzell theorem, we show that the rate function (which we call the ``Dyson-Dirichlet energy,'' \Cref{def:multiradial_Dirichlet_energy}) 
has the same form as in Freidlin-Wentzell theory for diffusions with uniformly Lipschitz drift.  
However, our analysis uses substantially different tools than Freidlin-Wentzell theory: 
change of measure plays a prominent role in our arguments. 

Moreover, in \Cref{sec:zero-energy}, we establish the existence and uniqueness of zero-energy particle systems (\Cref{prop:ODE_existence_uniqueness}).
Furthermore, we analyze the large-time behavior under additional repulsion and symmetry assumptions (\Cref{def:symmetric_Dyson_type_potential}).
We show in particular that zero-energy systems approach exponentially fast a static equally-spaced configuration (\Cref{thm:U_equally_spaced}). 
In turn, we show that finite-energy systems converge to an equally-spaced configuration in the long run, but the convergence rate can be slow, and the system may not become static 
(see \Cref{prop:finite_energy} and \Cref{subsec:Finite-energy-results}, where we provide some examples).

In \Cref{sec:SLE_LDP}, we prove the LDP for multiradial $\SLE_{0+}$ processes 
(\Cref{thm:radial_LDP_finite_time}).
Here, the main technical difficulty is that the $\SLE_\kappa$ curves have a common target point, preventing the usual configurational, or global, approach. 
Therefore, we instead will
make careful use of the contraction principle from the LDP for Dyson Brownian motion (\Cref{thm:Bessel_LDP}, a special case of the above LDP for Dyson-type diffusions). 
We combine it with topological results: namely, 
we show that finite-energy multiradial Loewner hulls are always disjoint unions of simple curves, except at their common endpoint (\Cref{thm:finite-energy-gives-simple-curves}).
A key to this is obtained from a derivative estimate for the radial Loewner map in terms of the energy of its driving function (\Cref{thm: FS_estimate}), which may be of independent interest.

We now provide references and discuss the scope and our main results in more detail.

\subsection{Background and scope}

Schramm-Loewner evolution $(\SLE_\kappa)_{\kappa \geq 0}$ is a natural model of a random interface arising from two-dimensional conformal geometry.  
$\SLE_\kappa$ curves have two equivalent characterizations: they can be defined in purely geometric and probabilistic terms (as curves satisfying conformal invariance and the domain Markov property), or
they can be defined in terms of
a one-parameter family of slit domains arising from the solutions to the Loewner equation with driving function $\sqrt \kappa \, B$, where $B$ is a standard Brownian motion~\cite{Schramm:Scaling_limits_of_LERW_and_UST}. These two perspectives are often referred to as the ``configurational'' (or ``global'') and ``dynamical'' (or ``local'') interpretations of $\SLE_\kappa$, respectively. 
Their interplay allows for a rich theory that employs tools from diverse disciplines, including 
conformal geometry~\cite{LSW:Conformal_restriction_the_chordal_case, Wang:Energy_of_deterministic_Loewner_chain}, stochastic analysis~\cite{Rohde-Schramm:Basic_properties_of_SLE, Dubedat:Commutation_relations_for_SLE, Miller-Sheffield:QLE}, interacting particle systems~\cite{Cardy:SLE_and_Dyson_circular_ensembles, ABKM:Pole_dynamics_and_an_integral_of_motion_for_multiple_SLE0, Zhang:Multiple_radial_SLE0_and_classical_Calogero-Sutherland_system}, 
Teichm\"uller theory~\cite{Wang:Equivalent_descriptions_of_Loewner_energy, Bishop:Weil_Petersson_curves_conformal_energies_beta-numbers_and_minimal_surfaces}, 
and algebraic geometry~\cite{Peltola-Wang:LDP_of_multichordal_SLE_real_rational_functions_and_determinants_of_Laplacians}.

The roughness of $\SLE_\kappa$ curves depends on a parameter $\kappa\geq 0$.
In particular, for different values of $\kappa$, variants of $\SLE_\kappa$ curves describe scaling limits of interfaces in a variety of statistical physics models
(e.g.,~\cite{LSW:Conformal_invariance_of_planar_LERW_and_UST, Smirnov:Towards_conformal_invariance_of_2D_lattice_models, Schramm:ICM, Schramm-Sheffield:Contour_lines_of_2D_discrete_GFF}). 
The close relationship with discrete statistical physics models also allows discrete intuition and enumerative analysis to inform conjectures about $\SLE_\kappa$ itself, as applied to the theory of multiple SLEs in~\cite{Kozdron-Lawler:Configurational_measure_on_mutually_avoiding_SLEs, BPW:On_the_uniqueness_of_global_multiple_SLEs, Healey-Lawler:N_sided_radial_SLE}. 
Interestingly, 
$\SLE_\kappa$ curves are also very closely related to conformal field theory~\cite{Bauer-Bernard:Conformal_field_theories_of_SLEs, Bauer-Bernard:CFTs_of_SLE_radial_case, Friedrich-Werner:Conformal_restriction_highest_weight_representations_and_SLE, Friedrich-Kalkkinen:On_CFT_and_SLE,DRC:Identification_of_the_stress-energy_tensor_through_conformal_restriction_in_SLE_and_related_processes, Kontsevich-Suhov:On_Malliavin_measures_SLE_and_CFT, Dubedat:SLE_and_Virasoro_representations_localization, Peltola:Towards_CFT_for_SLEs}, the Gaussian free field~\cite{Dubedat:SLE_and_free_field, Kang-Makarov:Gaussian_free_field_and_conformal_field_theory, Miller-Sheffield:Imaginary_geometry1, Sheffield:Zipper}, 
and random matrix theory~\cite{Cardy:SLE_and_Dyson_circular_ensembles, BLLP:Uniform_spanning_trees_and_random_matrix_statistics, CLM:Rate_of_convergence_in_multiple_SLE_using_random_matrix_theory}.

Natural variants of $\SLE_\kappa$ can be constructed from the so-called chordal $\SLE_\kappa$ by change of measure. For instance, multichordal $\SLE_\kappa$ (where each curve connects two distinct boundary points) has been investigated in many works, including~\cite{BBK:Multiple_SLEs_and_statistical_mechanics_martingales, Dubedat:Commutation_relations_for_SLE, Kozdron-Lawler:Configurational_measure_on_mutually_avoiding_SLEs, Lawler:Partition_functions_loop_measure_and_versions_of_SLE, Kytola-Peltola:Pure_partition_functions_of_multiple_SLEs, Peltola-Wu:Global_and_local_multiple_SLEs_and_connection_probabilities_for_level_lines_of_GFF, BPW:On_the_uniqueness_of_global_multiple_SLEs}. 
When $\kappa \in (0,4]$, it is the measure absolutely continuous with respect to the product measure on $n$ independent $\SLE_\kappa$ curves with Radon-Nikodym derivative
\begin{align}\label{eqn:chordal_RN_deriv}
\one \{\gamma^i \cap \gamma^j =\emptyset \textnormal{ for all } 1 \le i \neq j \le n \} \; \exp\bigg( \frac{\charge}{2} \sum_{p=2}^n \BLoop[L_p(\bgamma)] \bigg),
\end{align}
where $\BLoop[L_p(\bgamma)]$
is the Brownian loop measure of loops that intersect at least $p$ of the curves $\bgamma = (\gamma^1, \ldots, \gamma^n)$, 
and $\charge = \charge(\kappa)$ is a parameter known as the \emph{central charge}, 
\begin{align*}
\charge = \frac{(6-\kappa)(3\kappa -8)}{2\kappa}.
\end{align*}
(See, e.g.,~\cite{Kozdron-Lawler:Configurational_measure_on_mutually_avoiding_SLEs, Lawler:Partition_functions_loop_measure_and_versions_of_SLE, Peltola-Wu:Global_and_local_multiple_SLEs_and_connection_probabilities_for_level_lines_of_GFF} for this definition, 
and~\cite{Lawler-Werner:The_Brownian_loop_soup} for the construction of Brownian loop measure.)
However, these measures are mutually singular when $\kappa=0$.

Recently\footnote{Thereafter, a finite-time parameterized single-curve LDP for chordal $\SLE_{0+}$ appeared in~\cite{Guskov:LPD_for_SLE_in_the_uniform_topology}, 
and is now extended to infinite time in~\cite{Abuzaid-Peltola:Large_deviations_of_radial_SLE0}, where variants of $\SLE_{0+}$ in this stronger topology are systematically considered.
An LDP for $\SLE_{0+}$ with a force point in the Hausdorff metric was proved in the recent~\cite{Krusell:in_prep}. 
None of these works address multiple SLE curves reaching a common target point.}, 
a large deviation principle (LDP) for multichordal $\SLE_\kappa$ as $\kappa \to 0+$ was established in~\cite{Peltola-Wang:LDP_of_multichordal_SLE_real_rational_functions_and_determinants_of_Laplacians}. 
In that result, the convergence takes place in the Hausdorff metric, and the (good) rate function is termed the \emph{multichordal Loewner energy}. 
The results in~\cite{Peltola-Wang:LDP_of_multichordal_SLE_real_rational_functions_and_determinants_of_Laplacians}  
have far-reaching applications. The authors show that there is a unique arrangement of curves that minimize the multichordal Loewner energy for given boundary data, and the union of these curves is the real locus of a real rational function, thus providing an alternate proof of the Shapiro conjecture in real enumerative geometry~\cite{Sottile:Real_Schubert_calculus, Eremenko-Gabrielov:Rational_functions_with_real_critical_points_and_Shapiro_conjecture_in_real_enumerative_geometry}. The authors also show that the Loewner potential (which differs from the Loewner energy by a function of the boundary data) has a simple expression in terms of zeta-regularized determinants of Laplacians (similar to the loop case in~\cite{Wang:Equivalent_descriptions_of_Loewner_energy})
and is the semiclassical $\charge \to -\infty$ limit of certain CFT correlation functions (see~\cite{Dubedat:SLE_and_Virasoro_representations_localization, KKP:Conformal_blocks_q_combinatorics_and_quantum_group_symmetry, Peltola:Towards_CFT_for_SLEs, ABKM:Pole_dynamics_and_an_integral_of_motion_for_multiple_SLE0}).

In the present work, we investigate the asymptotic behavior as $\kappa \to 0+$ of multiradial $\SLE_\kappa$ (a multiple $\SLE$ in the disk where all curves have the origin as their common target point). 
We show that multiradial $\SLE_\kappa$ satisfies a finite-time LDP 
in the Hausdorff metric with good rate function that we call the \emph{multiradial Loewner energy} (see Theorem~\ref{thm:radial_LDP_finite_time}). 
A key to our approach (further developed in~\cite{AHP:Multiradial_SLE_Infinite-time_LDP}) 
is that we work with \emph{parameterized} curves. 
This difference in perspective is a result of the different way that multiple $\SLE_\kappa$ is constructed in the radial (in contrast to the chordal) setting. 
The common target point causes essential difficulties for a ``configurational'' approach to defining multiple $\SLE_\kappa$, since the Brownian loop measure in~\eqref{eqn:chordal_RN_deriv} blows up when curves intersect. 
This difficulty was addressed in the construction of multiradial $\SLE_\kappa$ in~\cite{Healey-Lawler:N_sided_radial_SLE}, whose main result is the construction of multiradial $\SLE_\kappa$ for $\kappa \leq 4$ as the solution to the multiradial Loewner equation for driving functions that evolve according to Dyson Brownian motion on the circle\footnote{We use ``\emph{Dyson Brownian motion on the circle}'' to refer to the evolution of points on the circle, while the ``\emph{radial Bessel process}'' refers to the evolution of the arguments of the same process. See Definition~\ref{def:radial_Bessel_DBM}.} 
with a particular repulsive strength
(\Cref{def:n-radial_SLE_driving_functions}, \Cref{rmk:def_justification}, and Section~\ref{subsec: Multiradial Loewner equation background}).

The connection between multiradial $\SLE_\kappa$ and Dyson Brownian motion was first described by Cardy in the physics literature~\cite{Cardy:SLE_and_Dyson_circular_ensembles}. 
Loewner evolution driven by Dyson Brownian motion has gained recent interest in~\cite{Katori:Bessel_Processes_Schramm_Loewner_Evolution_and_the_Dyson_Model, Katori-Koshida:Three_phases_of_multiple_SLE_driven_by_non-colliding_Dyson_Brownian_motions, Chen-Margarint:Perturbations_of_multiple_SLE_with_two_non-colliding_Dyson_Brownian_motions, FWW:Multiple_Ising_interfaces_in_annulus_and_2N-sided_radial_SLE, BLLP:Uniform_spanning_trees_and_random_matrix_statistics, CLM:Rate_of_convergence_in_multiple_SLE_using_random_matrix_theory}. 
An investigation of chordal Loewner evolution driven by a branching particle system (varying $n$) evolving according to Dyson Brownian motion for $\beta=+\infty$ appears in~\cite{Healey-Menon:Scaling_limits_of_branching_Loewner_evolutions_and_the_Dyson_superprocess}. 
However, asymptotic results linking $\SLE$ and Dyson Brownian motion have thus far focused on the setting where the number of curves tends to infinity
(cf.~\cite{Monaco-Schleissinger:Multiple_SLE_and_the_complex_Burgers_equation, Hotta-Katori:Hydrodynamic_limit_of_multiple_SLE, Hotta-Schleissinger:Limits_of_radial_multiple_SLE_and_Burgers-Loewner_differential_equation}).

The description of multiradial $\SLE_\kappa$ in terms of the corresponding driving functions provides the key tool in deriving the multiradial Loewner energy (Definition~\ref{def:multiradial_LE}).  
Accordingly, much of the present work is devoted to proving an LDP for a class of diffusion processes with
locally Lipschitz drifts (processes obtained from a potential of ``Dyson type,'' described in Definition~\ref{def:Dyson_type_potential}), 
including Dyson Brownian motion on the circle, 
which is of independent interest (see Theorems~\ref{thm:Bessel_LDP}~\&~\ref{thm:Bessel_LDP_general}).
Therefore, we have organized this article in such a way that, after the introduction of the main concepts and results, 
Sections~\ref{sec:Bessel_LDP}~\&~\ref{sec:zero-energy} only address diffusions and Dyson Brownian motion independently of Loewner theory (thus suitable for readers in a general probability audience), 
while Section~\ref{sec:SLE_LDP} contains our main results in Loewner theory (assuming some familiarity with basic techniques in stochastic analysis and complex geometry). 
We will recall concepts from LDP theory along the way.

\subsection{Dyson Brownian motion, Dyson-Dirichlet energy, and an LDP}

Fix an integer $n\geq 1$. 
Let $(\bR/2\pi\bZ)^n$ be the torus with periodic boundary conditions, 
and let $\chamber$ denote the subset of elements admitting representatives 
$\btheta = (\theta^1, \ldots, \theta^n)$ satisfying
\begin{align} \label{eq: torus ordering}
\theta^1<\theta^2< \cdots< \theta^n <\theta^1+2\pi.
\end{align}
Throughout, we use the convention that $\theta^{n+j} = \theta^j+2\pi$ for all $j$. 
Let $\BesselSpace{\infty)}{}$ denote the space of continuous functions $\btheta_t = (\theta^1_t, \ldots, \theta^n_t)$ from $[0,\infty)$ to $\chamber$. 
We will consider the unique strong solution $\bDyson_t := (\Dyson^1_t, \ldots, \Dyson^n_t)$ in $\BesselSpace{\infty)}{}$ to the system of SDEs
\begin{align}\label{eqn:driving_spde}
\ud \bDyson_t = \; & \bnewphi(\bDyson_t) \ud \bs{t} + \sqrt \kappa \ud \bs{W}_{\! t} ,
\qquad \kappa > 0 ,
\end{align}
where $\bs{t} = (t,\ldots,t)$, 
and $\bs{W}_{\! t} = (W^1_t, \ldots, W^n_t)$ are independent standard Brownian motions,
and the drift is given by a ``Dyson-type'' potential $\DPotential = - \log \partfn \geq 0$ (Definition~\ref{def:Dyson_type_potential}), 
\begin{align}\label{eqn:def_of_drift}
\bnewphi = (\newphi^1, \ldots, \newphi^n) = - (\partial_1 \, \DPotential, \ldots, \partial_n \, \DPotential) = - \nabla \DPotential  
\end{align}
up to the collision time
\begin{align} \label{eq: blowup time}
\blowuptime := \; & \inf \Big\{t \geq 0 \; \colon \; \min_{1 \leq i < j \leq n} \, \big| e^{\ii \Dyson^i_t} - e^{\ii \Dyson^j_t} \big| = 0 \Big\} .
\end{align}
The existence of a unique strong solution to~\eqref{eqn:driving_spde} for small enough $\kappa$ (large enough $\beta$ for Dyson Brownian motion)
is proven in \Cref{prop:martingale(new)} and \Cref{prop:Bessel_measure} in \Cref{sec:multiBessel_construction}. 
(The repulsive potential $\DPotential$ serves as a (stochastic) Lyapunov function for the SDEs~\eqref{eqn:driving_spde}.)

\begin{ex}
In particular, with the specific choice\footnote{Readers familiar with the partition function perspective of SLEs may observe that~\eqref{eq:DBM_choices} agrees with the semiclassical limit 
$\smash{ -\kappa \log  ( \nradpartfn (\btheta) )^{1/\kappa} \, \overset{\kappa \to 0+}{\longrightarrow} \, - \log \nradpartfn}$ 
of the multiradial $\SLE_\kappa$ partition function~\eqref{eq:multiradial_partition_function}.} 
$\partfn = (\nradpartfn)^{\two}$, 
\begin{align}\label{eq:DBM_choices}\tag{{\textsc{rad}}}
\begin{split}
\DPotential = \;& - \two \log \nradpartfn , 
\quad \textnormal{and} \quad \newphi^j = \newphinrad^j = \two \, \partial_j \log \nradpartfn, \quad \textnormal{where} \\
\nradpartfn(\btheta) := \; & 
\prod_{1\leq i < j \leq n} \sin^2 \bigg(\frac{\theta^j - \theta^i}{2}\bigg) 
= \prod_{1\leq i \neq j \leq n} \sin \bigg | \frac{\theta^j - \theta^i}{2}\bigg | ,
\qquad \textnormal{so}
\\ \quad 
\newphinrad^j(\btheta) =\;&
\two \phi^j(\btheta) := \two \sum_{\substack{1 \leq i \leq n \\[.1em] i\neq j}} \cot \bigg( \frac{\theta^j - \theta^i}{2} \bigg) ,
\end{split}
\end{align} 
the process $(e^{\ii \Dyson^1_t}, \ldots, e^{\ii \Dyson^n_t})$ 
is \emph{Dyson Brownian motion on the circle} (Section~\ref{subsec: DMB and Bessel background}).
\end{ex}

While using the parameter $\kappa$ in the context of Dyson Brownian motion is non-standard, 
our choice of $\bDyson$ is strongly motivated by its connection to multiradial $\SLE_\kappa$ curves, as discussed in the next section. 
Setting $\kappa=0$ in~\eqref{eqn:driving_spde} motivates the next definition\footnote{This is the usual Dirichlet energy when $n=1$ (see Equation~\eqref{eqn:Dirichlet_energy_def}).}.

\begin{df} \label{def:multiradial_Dirichlet_energy} 
The \emph{Dyson-Dirichlet energy} $\nBessel = \nBessel^{\bnewphi} \colon \BesselSpace{\infty)}{} \to [0,+\infty]$ is the limit 
\begin{align*}
\nBessel(\btheta)
:= \; & \lim_{\bigT \to \infty} \nBessel_\bigT(\btheta) \; \in \; [0,+\infty] ,
\qquad \btheta \in \BesselSpace{\infty)}{} ,
\end{align*}
where for each $\bigT > 0$, the \emph{\textnormal{(}truncated\textnormal{)} Dyson-Dirichlet energy} of $\btheta$ is
\begin{align*}
\nBessel_\bigT(\btheta) 
:= \; &
\begin{dcases} 
\tfrac{1}{2} \int_0^\bigT \sum_{j=1}^n \big| \tfrac{\ud}{\ud s} \theta^j_s -  \newphi^j (\btheta_s) \big|^2 \ud s ,
\quad & \textnormal{if } \btheta 
\textnormal{ is absolutely continuous on $[0,\bigT]$,} \\
+\infty , & \textnormal{otherwise.}
\end{dcases}
\end{align*}
In the case where the drift is given by~\eqref{eq:DBM_choices}, we also call $\nBessel^{\bnewphi} = \nBessel^{\bnewphinrad}$ 
the \emph{multiradial Dirichlet energy} and correspondingly, $\nBessel_\bigT^{\bnewphinrad}$ 
the \emph{\textnormal{(}truncated\textnormal{)} multiradial Dirichlet energy}.
\end{df}

The first main result of this work is \Cref{thm:Bessel_LDP_general}:
an LDP for the solution $\bDyson$ of~\eqref{eqn:driving_spde} as $\kappa \to 0+$, 
with good\footnote{Conventionally, a \emph{good rate function} is a rate function for which all level sets are compact.} 
rate function being the multiradial Dirichlet energy $\nBessel$. 
Our result also implies an LDP for Dyson Brownian motion on the circle (\Cref{thm:Bessel_LDP}). 
To state it, for fixed $\bigT \in (0,\infty)$ and $\btheta_0 \in \chamber$, 
we denote by $\BesselSpace{\bigT]}{\btheta_0}$ the space of continuous functions $\btheta$ from $[0,\bigT]$ to $\chamber$ started at $\btheta_0$.
We endow $\BesselSpace{\bigT]}{\btheta_0}$ with the metric
\begin{align} \label{eqn:sup-metric}
\metric{\btheta}{\bomega} 
:= \sup_{t \in [0,\bigT]} |\btheta_t - \bomega_t| 
= \sup_{t \in [0,\bigT]} \Big( (\theta_t^1-\omega_t^1)^2 + \cdots + (\theta_t^n-\omega_t^n)^2 \Big)^{1/2} .
\end{align}
Then, $\big( \BesselSpace{\bigT]}{\btheta_0}, \metricsymb \big)$ is a Polish space (as a separable complete metric space).

\begin{thm}[LDP for Dyson Brownian motion on the circle]
\label{thm:Bessel_LDP}
Fix $\bigT \in (0,\infty)$. 
Let $\bDyson$ be the unique strong solution to~\eqref{eqn:driving_spde} with drift given by~\eqref{eq:DBM_choices}, started at $\bDyson_0 = \btheta_0 \in \chamber$. 
The family $(\PrDyson)_{\kappa > 0}$ of laws induced by $\bDyson$ satisfies the following LDP in $\BesselSpace{\bigT]}{\btheta_0}$ with good rate function given by the multiradial Dirichlet energy $\nBessel_\bigT$\textnormal{:}

For any closed subset $\closed$ and open subset $\open$ of $\BesselSpace{\bigT]}{\btheta_0}$, we have
\begin{align} 
\label{eq: limsup claim basic}
\limsup_{\kappa \to 0+} \kappa \log \PrDyson \big[ \bDyson \in \closed \big] 
\; & \leq - \inf_{\btheta \in \closed} \nBessel_\bigT(\btheta) ,
\\
\label{eq: liminf claim basic}
\liminf_{\kappa \to 0+} \kappa \log \PrDyson \big[ \bDyson \in \open \big] 
\; & \geq - \inf_{\btheta \in \open} \nBessel_\bigT(\btheta).
\end{align}
\end{thm}

This will follow as a special case of our more general result, \Cref{thm:Bessel_LDP_general}, which we discuss in the next section. 
In essence, the proof of the LDP is a careful application of Varadhan's lemma (Lemma~\ref{lemma:Varadhan}) relying on properties of the Dyson-Dirichlet energy (derived in Section~\ref{subsec: Multiradial Dirichlet Energy}), 
which enable us to transport the well-known LDP of Brownian motion from Schilder's theorem (Theorem~\ref{thm:Schilder}),
which is also the basic case of $n=1$ of \Cref{thm:Bessel_LDP}.

\begin{remark} \label{rem:BM_circle_LDP}
Theorem~\ref{thm:Bessel_LDP} is stated for the radial Bessel process $\bDyson$. 
It is equivalent to an LDP for Dyson Brownian motion on the circle as $\beta = \frac{8}{\kappa}\to +\infty$, 
by considering $\exp(\ii \bDyson_t)$ and applying the contraction principle (\Cref{thm:contraction-principle}) 
to the continuous function\footnote{Throughout, we use the principal branch of the logarithm, so that angles are taken to lie in $[0,2\pi)$.} 
$-\ii \log(\variable)$. 
\end{remark}

Large deviation results for Dyson Brownian motion (for fixed $\beta$) as $n\to \infty$ have been considered, e.g., in~\cite{Guionnet-Zeitouni:Large_deviations_asymptotics_for_spherical_integrals},  and are closely connected to random matrix theory. 
In contrast, our Theorem~\ref{thm:Bessel_LDP} holds for fixed $n$ as $\beta \to +\infty$, thereby filling a gap in the literature.

The original study of Dyson Brownian motion dates back to~\cite{Dyson:Brownian‐motion_model_for_the_eigenvalues_of_a_random_matrix}, one of the founding articles of random matrix theory.
For fixed $n$ and $\beta\geq 1$, Dyson Brownian motion may be defined as the unique strong solution $(X^1_t, \ldots, X^n_t)$ in the Weyl chamber of type $A_{n-1}$,
\begin{align*}
\big\{ (x^1, \ldots, x^n) \in \bR^n \; | \; x^1 < x^2 < \cdots < x^n \big\},
\end{align*} 
to the system of SDEs
\begin{align}\label{eqn:DBM_real}
\ud X^j_t = \sum_{\substack{1 \leq i \leq n \\[.1em] i\neq j}} \frac{ \ud t }{X^j_t - X^i_t} + \sqrt{\frac{2}{\beta}} \ud W^j_t ,
\qquad \textnormal{for all } j \in \{1,\ldots,n\} .
\end{align}
In the present work, we consider the analogue of this process on the circle,  
where the radial Bessel-type equation~\eqref{eqn:driving_spde} plays the role of Equation~\eqref{eqn:DBM_real}. 
In particular, after a suitable time change (see Section~\ref{subsec: DMB and Bessel background}), 
we see that the relationship between $\beta$ and $\kappa$ is 
\begin{align}\label{eqn:beta_kappa_dependence}
\beta = \frac{8}{\kappa} ,
\end{align}
which matches the prediction of John Cardy from the physics literature~\cite{Cardy:SLE_and_Dyson_circular_ensembles}.

For particular values of $\beta$, Dyson Brownian motion describes the evolution of the ordered eigenvalues of symmetric, Hermitian, and symplectic matrix Brownian motions 
(corresponding to the self-dual Gaussian ensembles GOE, GUE, and GSE, for $\beta=1,2,4$, respectively --- see~\cite[Chapter~4]{AGZ:An_introduction_to_random_matrices}). 
For $\beta=2$, Dyson Brownian motion has the same law as $n$ independent Brownian motions conditioned on nonintersection~\cite{Katori-Tanemura:Noncolliding_Brownian_motions_and_Harish-Chandra_formula}. 
Furthermore, for general $\beta \in (0, +\infty]$ one can construct ensembles of Jacobi matrices whose eigenvalues correspond to~\eqref{eqn:DBM_real}, see~\cite{Dumitriu-Edelman:Matrix_models_for_beta_ensembles, Gorin-Kleptsyn:Universal_objects_of_the_infinite_beta_random_matrix_theory}.  
It would be particularly interesting to investigate the fluctuations near the large deviation limit of the Dyson Brownian motion~\eqref{eqn:DBM_real} in the sense of \Cref{thm:Bessel_LDP} and its relation with the $\beta=+\infty$ process considered in~\cite{Gorin-Kleptsyn:Universal_objects_of_the_infinite_beta_random_matrix_theory}.
Lastly, let us remark that a
new geometric construction of Dyson Brownian motion for general $\beta\in (0,+\infty]$ has recently appeared 
in~\cite{HIM:Motion_by_mean_curvature_and_Dyson_Brownian_motion} --- interestingly, this construction relies on tools from Riemannian geometry and mean curvature flow.

\subsection{General Dyson-type diffusions and their large deviations}

Next, we formulate a more general version of \Cref{thm:Bessel_LDP}: \Cref{thm:Bessel_LDP_general} stated below.
This key result is of independent interest, and will be useful, e.g., in applications to $\SLE$ variants. 
To state it, we need the following definitions (used throughout for the SDEs~\eqref{eqn:driving_spde}). 

\begin{df} \label{def:Dyson_type_potential} 
We say that a function $\DPotential \in C^2(\chamber, [0,\infty))$ is a \emph{Dyson-type potential} if 
\begin{enumerate}
\item it satisfies the asymptotic repulsive behavior
\begin{align}\label{eqn:growth_conditions}\tag{{\textsc{rep}}}
\lim_{\epsilon \to 0} \inf_{\btheta \in \partial \chamber^\epsilon} \DPotential(\btheta) = + \infty , 
\end{align} 
where $\chamber^\epsilon := \{\btheta\in \chamber \; | \; \DeltaMin_{\btheta} > \epsilon \}$, for $\epsilon > 0$, and $\smash{\DeltaMin_{\btheta} 
:= \underset{1\leq j\leq n}{\min} \, \big| \theta^{j+1} - \theta^j \big| \in \big[0,  \tfrac{2\pi}{n} \big]}$; 

\item 
and there exist constants $\constmult>0$ and $\constaddnew \geq 0$ such that 
\begin{align} \label{eq:differential_inequalities}\tag{{\textsc{de}}}
- \constaddnew \le \constmult\, \Delta \DPotential(\btheta) \le | \nabla \DPotential (\btheta) |^2 + \constaddnew , \qquad \textnormal{for all }  \; \btheta \in \chamber.
\end{align}
\end{enumerate}
We define the \emph{Dyson-type partition function} as $\partfn := \exp(-\DPotential) \in C^2(\chamber, (0,1])$. 
\end{df}

We have not seen \Cref{def:Dyson_type_potential} in the literature. 
Our motivation to refer to the potentials $\DPotential$ in it as ``Dyson-type'' stems from our application to Dyson Brownian motion. 
The recent work~\cite{GLV:Dyson_Brownian_motion_on_a_Jordan_curve} considers a particular Dyson-type potential and uses 
our main results to derive an LDP for the parametrization process of Dyson Brownian motion on a Jordan curve.
Other related examples include ``Log-gas type'' or ``electrostatic'' potentials, appearing in the Coulomb gas literature 
(see~\cite{Serfaty:Lectures_on_Coulomb_and_Riesz_gases} and references therein).

Roughly speaking, the first condition~\eqref{eqn:growth_conditions} guarantees strong local repulsion of particles. 
As the second condition~\eqref{eq:differential_inequalities} may not appear natural at first, let us briefly comment on its significance. 
Roughly speaking, away from singularities the derivatives $\newphi^j = -\partial_j \, \DPotential$ of the potential $\DPotential$ --- yielding the drift~\eqref{eqn:def_of_drift} in the SDE~\eqref{eqn:driving_spde} --- are locally Lipschitz.
The upper bound in~\eqref{eq:differential_inequalities} together with the asymptotics in~\eqref{eqn:growth_conditions} guarantees 
a (non-Lipschitz) repulsion of strength at least one over the distance \Cref{ex:differential_inequalities_Phi}). 
This condition is also important for the existence of the strong solution to the SDEs~\eqref{eqn:driving_spde} for all time. 
In turn, the lower bound in~\eqref{eq:differential_inequalities} prevents non-Lipschitz attraction of the particles. 
One may also think of the Laplacian $\Delta \DPotential$ as the mean curvature of the graph of $\DPotential$. 
Then, the lower bound in~\eqref{eq:differential_inequalities} gives a lower bound for the mean curvature.
(Note also that requiring both a constant upper and lower bound in~\eqref{eq:differential_inequalities} would yield a Lipschitz constraint that would make the potential fall into the scope of Freidlin-Wentzell theory.)

Before stating the result, let us discuss a couple of examples that concretely illustrate its scope.
\Cref{ex: differential_inequalities} concerns the particular choice used in our application to SLE theory in Equation~\eqref{eq:DBM_choices},
and \Cref{ex:differential_inequalities_Phi} illustrates the upper bound in~\eqref{eq:differential_inequalities}.

\begin{ex}\label{ex: differential_inequalities}
If $\partfn=\nradpartfn$, condition~\eqref{eq:differential_inequalities} holds for $\constmult=2$ and $\constaddnew=\frac{n(n^2-1)}{3}$,
with equality in the lower bound, as in~\cite[Lemma~5.1]{Healey-Lawler:N_sided_radial_SLE}:
\begin{align}\label{eqn:trig_identity}
- 2 \sum_{j=1}^n \partial_j  \phi^j(\btheta) 
= \; &
\sum_{j=1}^n \sum_{\substack{1 \leq i \leq n \\[.1em] i\neq j}} 
\csc^2 \bigg( \frac{ \theta^j - \theta^i }{2} \bigg) \\
\; = \; & 
\sum_{j=1}^n \bigg(\sum_{\substack{1 \leq i \leq n \\[.1em] i\neq j}} 
\cot \bigg( \frac{ \theta^j - \theta^i }{2} \bigg)\bigg)^2
+ \frac{n(n^2-1)}{3} 
\; = \; \sum_{j=1}^n \big(\phi^j(\btheta)\big)^2 + \frac{n(n^2-1)}{3}.
\nonumber
\end{align}
Similarly, if $\partfn = (\nradpartfn)^2$, condition~\eqref{eq:differential_inequalities} holds for $\constmult=4$ and $\constaddnew=\frac{4n(n^2-1)}{3}$. 
\end{ex}

\begin{ex}[$(1/\theta)$-singularity for $n=2$ particles]\label{ex:differential_inequalities_Phi}
To better understand the behavior of functions that satisfy the differential inequalities~\eqref{eq:differential_inequalities}, 
let us consider the case of two particles with a potential $\DPotential$ 
that depends only on $\theta := \theta_2-\theta_1$.  
Letting $\varphi(\theta) := -\partial_1 \DPotential(\theta_1, \theta_2)$, equation~\eqref{eq:differential_inequalities} gives the bounds
\begin{align} \label{eq:differential_inequalities_Phi}
-\constaddnew \le -2 \constmult \, \varphi'(\theta) \le 2 \big(\varphi(\theta)\big)^2 + \constaddnew , \qquad \theta \in (0, 2\pi) ,
\end{align}
for some $\constmult>0$, $\constaddnew\ge 0$. 
We are primarily interested in the behavior of the system near the singularity $\theta=0$.
The condition~\eqref{eqn:growth_conditions} implies that 
\begin{align*}
\lim_{\theta\to 0+} \DPotential(\theta)=+\infty 
\qquad \Longrightarrow \qquad 
\liminf_{\theta \to 0+}\varphi(\theta) = -\limsup_{\theta \to 0+}\DPotential'(\theta) = +\infty .
\end{align*}
If $\constaddnew=0$, then the upper bound in~\eqref{eq:differential_inequalities_Phi} directly yields $F(\theta) \geq \frac{\constmult}{\theta}+C$. 
If $\constaddnew>0$, then the upper bound is satisfied by the solution 
$F(\theta) = \sqrt{\tfrac{ \constaddnew }{2}} \cot \Big(\sqrt {\frac{ \constaddnew }{2}}\,\frac{ \theta}{\constmult}\Big)$ 
to\footnote{The asymptotic boundary condition can be justified by taking a sequence of initial conditions converging to this one. However, we have omitted this step since it obscures the focus of the example.}
\begin{align*}
- 2\constmult  \, F'(\theta) = 2 (F(\theta))^2 + \constaddnew , \qquad 
\lim_{\theta \to 0 +} F(\theta)=+\infty . 
\end{align*}
The upper bound in~\eqref{eq:differential_inequalities_Phi} implies that $\varphi'(\theta)\geq F'(\theta)$, so by general theory of differential inequalities (see, e.g.,~\cite[Chapter~1]{Lakshmikantham-Leela:Differential_and_integral_inequalities}), we have $\varphi(\theta)\geq F(\theta)$ near $0$. 
As $\theta \mapsto \cot(\theta)$ has a $(1/\theta)$-singularity at $0$, we conclude that $\varphi$ has  a singularity at $0$ of order at least $\constmult/\theta$.
(Note also that $F(\theta)$ trivially satisfies the lower bound in~\eqref{eq:differential_inequalities_Phi}.) 
\end{ex}

We can now state the general LDP for Dyson-type potentials ---
the LDP for the radial Bessel process (\Cref{thm:Bessel_LDP}) is the special case~\eqref{eq:DBM_choices} 
with $\partfn=(\nradpartfn)^\two$ and $\bnewphi=\bnewphinrad$. 

\begin{restatable}{thm}{BesselLDPgeneral}\textnormal{(LDP for Dyson-type diffusions)\textbf{.}}
\label{thm:Bessel_LDP_general}
Fix $\bigT \in (0,\infty)$. 
Let $\DPotential=-\log \partfn$ be a Dyson-type potential
\textnormal{(}\Cref{def:Dyson_type_potential}\textnormal{)}, 
and let $\bDyson$ be the associated unique strong solution to~\eqref{eqn:driving_spde}, started at $\bDyson_0 = \btheta_0 \in \chamber$. 
The family $(\PrDyson)_{\kappa > 0}$ of laws induced by $\bDyson$ satisfies the LDP~\textnormal{(\ref{eq: limsup claim basic},~\ref{eq: liminf claim basic})} in $\BesselSpace{\bigT]}{\btheta_0}$ with good rate function~$\nBessel_\bigT$ \textnormal{(}\Cref{def:multiradial_Dirichlet_energy}\textnormal{)}.
\end{restatable}

Large deviation theory for stochastic differential equations is a field of independent interest. Theorem~\ref{thm:Bessel_LDP_general} fits very naturally in this context. 
The Dyson-Dirichlet energy $\nBessel_\bigT$  
is exactly the rate function that would be predicted by applying Freidlin-Wentzell theory to the diffusion~\eqref{eqn:driving_spde} (see~\cite{Freidlin-Wentzell:Random_Perturbations_of_Dynamical_Systems}, originally published in Russian in 1979, and~\cite{Dembo-Zeitouni:Large_deviations_techniques_and_applications} for a survey). 
However, since the drift in~\eqref{eqn:driving_spde} is not uniformly Lipschitz continuous, the Freidlin-Wentzell theorem does not apply directly. 
Our Theorem~\ref{thm:Bessel_LDP_general} can thus be considered as an extension of the Freidlin-Wentzell theorem to a class of diffusions with non-Lipschitz drift.
However, our analysis uses substantially different tools than Freidlin-Wentzell theory: motivated by the applications to $\SLE$ theory, change of measure plays a prominent role in our arguments. 
We implement this via the interaction potentials described in \Cref{def:Dyson_type_potential}.

Our results also immediately yield an LDP for diffusions including a linear drift $2\spr \ud t$, which are used 
in~\cite{Miller-Sheffield:Imaginary_geometry4, KWW:Commutation_relations_for_two-sided_radial_SLE, HPW:Multiradial_SLE_with_spiral} to study
$\SLE_\kappa^\spr$ curves with ``spiraling rate'' $\spr \in \bR$. 
Note that the potential associated to these processes is not bounded from below.

\begin{restatable}{cor}{BesselLDPspiral}\textnormal{(LDP for Dyson Brownian motion on the circle with spiral)\textbf{.}}
\label{cor:Bessel_LDP_spiral}
Fix $\bigT \in (0,\infty)$ and $\spr \in \bR$. 
Let $\bDyson$ be the unique strong solution to
\begin{align}\label{eqn:diffusion_in_new_measure_spiral}
\ud \bDyson_t{}^{,\spr} = 2 \big( \bphi(\bDyson_t{}^{,\spr})  + \bspr \big)\ud \bs{t}  + \sqrt{\kappa} \ud \bs{W}_{\! t}  ,
\end{align}
with drift involving $\bphi$ as in~\eqref{eq:DBM_choices}, 
started at $\bDyson_0{}^{,\spr} = \btheta_0 \in \chamber$. 
The family $(\Pspiral)_{\kappa > 0}$ of laws induced by $\bDyson{}^{,\spr}$ satisfies the LDP~\textnormal{(\ref{eq: limsup claim basic},~\ref{eq: liminf claim basic})} in $\BesselSpace{\bigT]}{\btheta_0}$ with good rate function
\begin{align*}
\nBessel^{\spr}_\bigT(\btheta) 
:= \; &
\begin{dcases} 
\tfrac{1}{2} \int_0^\bigT \sum_{j=1}^n \big| \tfrac{\ud}{\ud s} \theta^j_s - 2 \big( \phi^j(\btheta_s) + \spr \big)\big|^2 \ud s ,
\quad & \textnormal{if } \btheta 
\textnormal{ is absolutely continuous on $[0,\bigT]$,} \\
+\infty , & \textnormal{otherwise,}
\end{dcases}
\end{align*}
associated to\footnote{Readers familiar with the partition function perspective of SLEs may observe that this agrees with the semiclassical limit of the multiradial $\SLE_\kappa^\spr$ partition function~\cite[Equations~(1.10,~1.12)]{HPW:Multiradial_SLE_with_spiral},
$\smash{ - \frac12 \log\partfn_{\mathrm{rad}}^{\spr} \, = \, - \underset{\kappa \to 0+}{\lim} \, \kappa \log \big( ( \nradpartfn (\btheta) )^{1/\kappa} \exp\big(\frac{\spr}{\kappa} \sum_{j} \theta^j \big)  \big)}$.}  
the potential $\DPotential_{\mathrm{rad}}^{\spr} = - \log \partfn_{\mathrm{rad}}^{\spr}$ defined by the partition function
\begin{align*} 
\partfn_{\mathrm{rad}}^{\spr} 
:= (\nradpartfn(\btheta))^\two \exp\bigg( 2 \spr \sum_{j=1}^n \theta^j \bigg) 
\, = \,\underset{1\leq i < j \leq n}{\prod} \, \sin^{4} \bigg(\frac{\theta^j - \theta^i}{2}\Big) \, \exp \bigg( \spr\sum_{j=1}^n \theta^j \bigg) .
\end{align*}
\end{restatable}

\begin{proof}
This follows from \Cref{thm:Bessel_LDP}, e.g., by applying the contraction principle (\Cref{thm:contraction-principle}) to the continuous map sending the function $t \mapsto \btheta_t$ to the function $t \mapsto \btheta_t + 2 \spr \, t$.
\end{proof}

An analogue of \Cref{cor:Bessel_LDP_spiral} of course also holds with more general drifts $(\newphi^j + 2\spr) \ud t$.

\subsection{Large deviations of multiradial $\SLE_{0+}$}

Next, we describe the implications of our DBM results to SLE theory: with some topological work (see \Cref{thm:finite-energy-gives-simple-curves}), 
we obtain a finite-time LDP for multiple radial $\SLE_{0+}$ curves (\Cref{thm:radial_LDP_finite_time}; and including spiral in \Cref{rem:spiral_intro}).
Readers only interested in the results concerning Dyson-type diffusions may proceed directly to \Cref{subsec:fin_en}.

We will mainly work on Loewner flows in the following setup. 
The \emph{multiradial Loewner equation} with the common parameterization is 
the boundary value problem
\begin{align}\label{eqn:multiradial_Loewner_1common}
\nradial{}{g}{t}{z}{z^j}, 
\qquad g_0(z) = z ,
\qquad z \in \overline{\bD} , \quad t \geq 0 ,
\end{align}
where $z^1_t, \ldots, z^n_t \in \partial \bD$ are cyclically ordered points on the unit circle, called the \emph{driving functions} (drivers).
It is most common to study~\eqref{eqn:multiradial_Loewner_1common} for drivers which are non-intersecting and continuous in time, in which case the maps $g_t$ 
that satisfy~\eqref{eqn:multiradial_Loewner_1common} generate a locally growing family of compact subsets $K_t$ of $\overline{\bD}$ 
(hulls\footnote{A hull is a compact set $K \subset \overline{\bD}$ such that $\bD\smallsetminus K$ is simply connected, $0 \in \bD\smallsetminus K$, and $\overline{K \cap \bD} = K$.}) satisfying $g_t(\bD\smallsetminus K_t)=\bD$. 
In fact, each $g_t \colon \bD\smallsetminus K_t \to \bD$ 
is the unique conformal mapping that satisfies $g_t(0)=0$ and $g_t'(0)>0$. 
Throughout, we refer to this map as the \emph{uniformizing map normalized at the origin}.

The parameterization in~\eqref{eqn:multiradial_Loewner_1common} guarantees that $g_t'(0) = e^{nt}$ for all $t$. 
If each hull $K_t$ is a union of $n$ disjoint connected components,  the ``common parameterization'' implies that, roughly, 
each component is locally growing at the same rate (see~\cite{Healey-Lawler:N_sided_radial_SLE} for more details). 
See also Equation~\eqref{eqn:multiradial_Loewner_general} for a more general case involving weights.

It will be convenient to use the angle coordinates $z^j_t = \exp(\ii \theta^j_t)$, where the driving function becomes $\btheta = (\theta^1, \ldots, \theta^n) \in \BesselSpace{\infty)}{}$. 
For each $t \geq 0$, 
the \emph{Loewner transform} 
$\cL_t \colon \BesselSpace{t]}{} \to \compsets$ sends driving functions to hulls, 
\begin{align} \label{eq: Loewner transform}
K_t = \cL_t(\btheta) 
:= \; & \{z \in \overline{\bD} \; | \; \swallowtime{z} \le t\} \; \subset \; \compsets ,
\end{align}
where $\compsets$ is the set of non-empty  compact subsets of $\overline{\bD}$, and
$\swallowtime{z}$ is the swallowing time of $z$,
\begin{align*}
\swallowtime{z} := \; & \sup \Big\{t \geq 0 \; \colon \;  \inf_{s \in [0,t]} \, \min_{1 \leq j \leq n} \, \big| g_s(z) - e^{\ii \theta_s^j} \big| > 0 \Big\} \; \in \; [0,+\infty] .
\end{align*}
We endow $\compsets$ with the Hausdorff metric $d_{\compsets} \colon \compsets \times \compsets \to [0,\infty)$ defined by
\begin{align} \label{eq: Hausdorff metric}
d_{\compsets}(K, \tilde K) 
: = \inf \big\{r > 0 \; \colon \; K \subset \mathcal{B}_{\tilde K}(r) \textnormal{ and } \tilde K \subset \mathcal{B}_K(r)\big\},
\end{align}
where $\mathcal{B}_K(r) := \displaystyle{\bigcup_{x \in K}} \mathcal{B}_x(r)$. 
Then, $(\compsets,d_{\compsets})$ is a compact metric space.

\begin{df} \label{def:radial_multichord}
Fix distinct points $x^1, \ldots, x^n \in \partial \bD$. 
We call an $n$-tuple $\bgamma = (\gamma^1, \ldots, \gamma^n)$ such that $\gamma^1, \ldots, \gamma^n$ are curves\footnote{Here, we think of $\gamma^j$ as a continuous map from $(0,\infty)$ to $\overline{\bD}$, started at $\gamma^j_0$ and ending at $\smash{\underset{t\to \infty}{\lim} \, \gamma^j_t}$. 
Note that the definition of a radial multichord allows the curves to intersect.} in $\overline{\bD}$ with  
$\gamma^j_0 = x^j$, and $\lim_{t\to \infty} \gamma^j_t=0$ for each $j$, 
a~\emph{radial multichord} in $(\bD;x^1, \ldots, x^n)$. 
We naturally identify $\bgamma$ with the union $\cup_j \gamma^j \in \compsets$. 
\end{df}

\begin{df} \label{def:multiradial_LE}
For each $\bigT \in (0,\infty)$, 
we define the \emph{\textnormal{(}truncated\textnormal{)} multiradial Loewner energy} of a radial multichord to be the multiradial Dirichlet energy $\nBessel_\bigT$ 
of its driving function in the common parameterization (as in Definition~\ref{def:multiradial_Dirichlet_energy} with drift given by~\eqref{eq:DBM_choices}). 
More generally, we define the energy functional 
$\lenergy_\bigT \colon \compsets \to [0,+\infty]$ 
on the metric space $(\compsets, d_{\compsets})$ by 
\begin{align}\label{eqn:nradial-energy}
\lenergy_\bigT(K) := \inf_{\btheta \in \cL_\bigT^{-1}(K)} \nBessel_\bigT(\btheta) , \qquad \bigT \in (0,\infty) , \quad K \in \compsets .
\end{align}
\end{df}

The infimimum in~\eqref{eqn:nradial-energy} is finite only if $K$ is an image of a finite-energy hull, which necessarily is a union of simple curves. 
Thus, either $\cL_\bigT^{-1}(K)$ consists of only one element, or every $\btheta \in \cL_\bigT^{-1}(K)$ has infinite energy. 
In particular, we have $\lenergy_\bigT(K)=+\infty$ if there is no driving function that generates $K$ in time $\bigT$ in the common parameterization.

We can also define the energy in a general domain $D \subsetneq \bC$ by conformal equivalence using a conformal mapping $\Mob \colon D \to \bD$ with $\Mob(0)=0$ and $\Mob'(0)>0$:
\begin{align*}
\lenergy_\bigT(\tilde{K}; D) := \lenergy_\bigT(K) ,
\qquad \textnormal{where $\tilde{K} \subset \overline{D}$ and $\Mob(\tilde{K})=K \subset \overline{\bD}$.}
\end{align*}

We next give the definition of $n$-radial $\SLE_\kappa$ that we will use for the remainder of this work. 
\Cref{rmk:def_justification2} and the discussion following it in Section~\ref{subsec: Multiradial Loewner equation background} offer additional justification for this definition and comparison to~\cite{Healey-Lawler:N_sided_radial_SLE}. 
(Here and in the sequel, we use a loose notion of an $n$-radial multichord: we frequently consider them in finite time-windows\footnote{However, given an $n$-tuple of radial curves on a time-window $[0,\bigT]$, we obtain an $n$-radial multichord in the strict sense of \Cref{def:radial_multichord} by attaching to the tip $\gamma^j_\bigT$ of each curve, say, a hyperbolic line segment from $\gamma^j_\bigT$ to $0$ and parameterizing $\gamma^j$ such that $\smash{\underset{t\to \infty}{\lim} \, \gamma^j_t=0}$ for each $j$.}.)

\begin{df}\label{def:n-radial_SLE_driving_functions}
Fix $\btheta_0 \in \chamber$ and $\bz_0 = (e^{\ii \theta^1_0}, \ldots, e^{\ii \theta^n_0})$. 
For each parameter $\kappa \in (0,4]$, 
\emph{$n$-radial $\SLE_\kappa$ with the common parameterization} started from $\bz_0$ 
is the random radial multichord $\bgamma$ for which the uniformizing conformal maps $\smash{g_t \colon \bD \smallsetminus \bgamma_{[0,t]} \to \bD}$ 
satisfy Equation~\eqref{eqn:multiradial_Loewner_1common}
with drivers $z^j_t = e^{\ii \Dyson^j_t}$ for $1 \leq j \leq n$, 
where $\bDyson_t = (\Dyson^1_t, \ldots, \Dyson^n_t)$ is the unique strong solution in $\BesselSpace{\infty)}{\btheta_0}$ to the SDEs~\eqref{eqn:driving_spde} 
with drift given by~\eqref{eq:DBM_choices}, started at $\bDyson_0 = \btheta_0$.
Note that in this case, the collision time~\eqref{eq: blowup time} is almost surely infinite, i.e., $\blowuptime = \infty$
(see \Cref{prop:Bessel_measure}).
(Compare with~\eqref{eqn:DBM_real} for $\beta\geq 1$, cf.~\cite{AGZ:An_introduction_to_random_matrices}.) 
\end{df}

While not addressed in the original~\cite{Healey-Lawler:N_sided_radial_SLE},  
$n$-radial $\SLE_\kappa$ with $\kappa \in (0,4]$ comprises transient curves that eventually end up at the origin,
and thus calling them $n$-radial multichords is justified~\cite{AHP:Multiradial_SLE_Infinite-time_LDP,HPW:Multiradial_SLE_with_spiral}.
(The present article does not need this fact.)

A key aspect of Definition \ref{def:n-radial_SLE_driving_functions} is the specific weight ``$2$'' in front of the drift term $\newphinrad^j = \two \phi^j $ in~\eqref{eq:DBM_choices}. 
This is the drift strength that appears when considering large-time-$\bigT$ truncations of the chordal Radon-Nikodym derivative~\eqref{eqn:chordal_RN_deriv} 
when all curves are growing simultaneously (in the common parameterization) and then taking $\bigT \to \infty$, as in~\cite{Healey-Lawler:N_sided_radial_SLE}.  
Other values of this weight give rise to other variants of $\SLE$, including so-called ``locally independent'' $\SLE_\kappa$ when the drift is instead multiplied by ``$1,$'' 
see~\cite{Healey-Lawler:N_sided_radial_SLE}. 
In Section~\ref{subsec: Multiradial Loewner equation background}, we discuss Loewner evolutions with various weight functions.

\begin{remark} \label{rmk:def_justification}
The drift strength $\newphinrad^j =\two \phi^j$ for simultaneously growing multiradial $\SLE_\kappa$ curves in~\eqref{eq:DBM_choices}
can also be derived via the superposition of the individually growing (marginal) processes 
in the following manner (see also~\cite{Graham:Multiple_SLEs, Healey-Lawler:N_sided_radial_SLE}). 
The multiradial $\SLE_\kappa$ partition function~\cite{{Cardy:SLE_and_Dyson_circular_ensembles}} is defined as
\begin{align}\label{eq:multiradial_partition_function}
\big ( \nradpartfn (\btheta) \big)^{1/\kappa} =
\underset{1\leq i < j \leq n}{\prod} \, \sin^{2/\kappa} \bigg(\frac{\theta^j - \theta^i}{2}\bigg).
\end{align} 
Growing one curve at a time (in this case $\gamma^j$)  yields the marginal dynamics
\begin{align}\label{eq:indiv_dynamics}
\begin{aligned}
\ud \theta^j_t &= \kappa \, \partial_j \log \big ( \nradpartfn (\btheta_t) \big)^{1/\kappa} \ud t + \sqrt{\kappa} \ud W^j_t
=  \phi ^j (\btheta_t) \ud t + \sqrt{\kappa} \ud W^j_t,\\
\ud \theta^i_t&= \cot \bigg( \frac{\theta^i_t-\theta^j_t}{2}\bigg) \ud t, \quad i\neq j , 
\end{aligned}
\end{align}
--- see~\cite{Dubedat:Commutation_relations_for_SLE, Lawler:Partition_functions_loop_measure_and_versions_of_SLE, Peltola-Wu:Global_and_local_multiple_SLEs_and_connection_probabilities_for_level_lines_of_GFF} for this point of view. 
Simultaneous growth in the common parameterization is obtained by the superposition of these dynamics (as in, e.g.,~\cite{BBK:Multiple_SLEs_and_statistical_mechanics_martingales, ABKM:Pole_dynamics_and_an_integral_of_motion_for_multiple_SLE0}). 
In particular, to determine the flow of $\theta^\ell$ in the common parameterization, 
we can sum the systems~\eqref{eq:indiv_dynamics} over $j=1, \ldots, n$ and collect all of the $\! \ud\theta^\ell_t$ terms: 
if $j=\ell$, the first line contributes a drift of $\phi^\ell (\btheta)$, while for each $j\neq \ell$ the second line contributes a single term $\cot \big( (\theta^\ell_t-\theta^j_t)/2\big)$. 
Adding these up, we obtain an SDE of the same form as~\eqref{eqn:driving_spde}: 
\begin{align}
\begin{split}
\ud \theta^\ell_t &=\phi^\ell (\btheta_t) \ud t + \sqrt{\kappa} \ud W^\ell_t+ \sum_{\substack{1 \leq j\leq n \\[.1em] j\neq \ell}}  \cot \bigg( \frac{\theta^\ell_t-\theta^j_t}{2}\bigg) \ud t \\
\label{eqn:driving_spde_radial}
&= 2\phi^\ell (\btheta_t) \ud t + \sqrt{\kappa} \ud W^\ell_t , \qquad  \textnormal{for all } \ell \in \{1,\ldots,n\} ,
\end{split}
\end{align}
with the weight ``$2$'' appearing. 
\end{remark}

Our second main result is the following finite-time LDP for multiradial $\SLE_{0+}$.

\begin{thm}[LDP for multiradial $\SLE$]\label{thm:radial_LDP_finite_time}
Fix $\bigT \in (0,\infty)$. 
The initial segments $\bgamma^\kappa_{[0,\bigT]} \in \compsets $ of multiradial $\SLE_\kappa$ curves 
satisfy the following LDP in $\compsets$ with good rate function $\lenergy_\bigT$\textnormal{:}

For any Hausdorff-closed subset $\closed$ and Hausdorff-open subset $\open$ of $\compsets$, we have
\begin{align} 
\label{eq: limsup claim curves}
\limsup_{\kappa \to 0+} \kappa \log \probSLE \big[ \bgamma^\kappa_{[0,\bigT]} \in \closed \big] 
\; & \leq - \inf_{K \in \closed} \lenergy_\bigT (K) ,
\\
\label{eq: liminf claim curves}
\liminf_{\kappa \to 0+} \kappa \log \probSLE \big[ \bgamma^\kappa_{[0,\bigT]} \in \open \big] 
\; & \geq - \inf_{K \in \open} \lenergy_\bigT (K) .
\end{align}
\end{thm}

We prove \Cref{thm:radial_LDP_finite_time} in Section~\ref{subsec: LDP for SLE}. 
The idea is to make careful use of the contraction principle and derive Theorem~\ref{thm:radial_LDP_finite_time} 
from the LDP for Dyson Brownian motion 
(Theorem~\ref{thm:Bessel_LDP}).
The usage of the contraction principle will be justified by topological results concerning Loewner theory and finite-energy hulls.
These results enable us to essentially disregard the discontinuities of the Loewner transform 
(as in~\cite{Peltola-Wang:LDP_of_multichordal_SLE_real_rational_functions_and_determinants_of_Laplacians}). 
We show that finite-energy multiradial Loewner hulls are always disjoint unions of simple curves  
(see Theorem~\ref{thm:finite-energy-gives-simple-curves} below). 
For this, our strategy is to first derive a derivative estimate for the single-chord radial Loewner map in terms of the energy of its driving function (see \Cref{thm: FS_estimate}), 
and then to use complex analysis techniques to pull this result to the case of several curves.

\begin{thm}\label{thm:finite-energy-gives-simple-curves}
Consider a multiradial Loewner chain $(K_t)_{t \geq 0} = (\cL_t(\btheta))_{t \geq 0}$ 
with the common parameterization for which the uniformizing conformal maps $g_t \colon \bD \smallsetminus K_t \to \bD$ solve~\eqref{eqn:multiradial_Loewner_1common}
with drivers $z^j_t = e^{\ii \theta^j_t}$ for $1 \leq j \leq n$, 
where $\btheta = (\theta^1, \ldots, \theta^n) 
\in C_{\btheta_0}([0,\bigT], \chamber)$.
If $\lenergy_\bigT(K_\bigT) < \infty$, then the hull $K_\bigT$ consists of $n$ pairwise disjoint simple curves.
\end{thm}

We prove Theorem~\ref{thm:finite-energy-gives-simple-curves} in Section~\ref{subsec:finite-energy-gives-simple-curves}. 
The key inputs are the derivative estimate in the case of $n=1$ (\Cref{thm: FS_estimate} in \Cref{subsec:Derivative estimate}), 
which is a weighted, radial generalization of a result appearing in~\cite{Friz-Shekhar:Finite_energy_drivers}, 
and a sort of analogue of the conformal restriction property 
(see Proposition~\ref{prop:removal-of-hulls-loewner}), which we will utilize to pull the $n=1$ result to general $n \geq 2$. 

\begin{remark}\label{rem:spiral_intro}
Multiradial $\SLE_\kappa^\spr$ with spiraling rate $\spr \in \bR$ 
is the random radial multichord $\bgamma$ for which the uniformizing conformal maps $\smash{g_t \colon \bD \smallsetminus \bgamma_{[0,t]} \to \bD}$ satisfy~\eqref{eqn:multiradial_Loewner_1common}
with driving functions $z^j_t = e^{\ii \Dyson^j_t}$ for $1 \leq j \leq n$, 
where $\bDyson_t = (\Dyson^1_t, \ldots, \Dyson^n_t)$ is the strong solution in $\BesselSpace{\infty)}{\btheta_0}$ to the SDEs~\eqref{eqn:diffusion_in_new_measure_spiral} 
\cite{Miller-Sheffield:Imaginary_geometry4, KWW:Commutation_relations_for_two-sided_radial_SLE, HPW:Multiradial_SLE_with_spiral}. 
Our results apply directly to derive a finite-time LDP for this process as well (i.e., a version of \Cref{thm:radial_LDP_finite_time}), 
with good rate function obtained from \Cref{cor:Bessel_LDP_spiral} 
similarly as in Equation~\eqref{eqn:nradial-energy}. 
\end{remark}

\subsection{Finite-energy and zero-energy systems}
\label{subsec:fin_en}

In the final Section~\ref{sec:zero-energy}, we shall analyze the interacting particle system corresponding to finite-energy drivers of Dyson type governed by Definitions~\ref{def:multiradial_Dirichlet_energy}~\&~\ref{def:Dyson_type_potential}, 
under the additional assumption that the potential is ``symmetric'' and ``separately convex'' (\Cref{def:symmetric_Dyson_type_potential}). 
The symmetry assumption~\eqref{eqn:sym_def_intro} implies that the potential comprises identical pair potentials governing the interaction of each pair of neighboring particles,
and the separate convexity~\eqref{eqn:deriv_bound_def_intro} is a sort of partial convexity (or monotonicity) constraint. 
These properties enable us to identify the zero-energy equilibrium as the static equally-spaced configuration (Theorem~\ref{thm:U_equally_spaced}), 
and to analyze the large-time convergence rates of finite-energy and zero-energy systems (Theorem~\ref{thm:U_equally_spaced}~\&~\Cref{prop:finite_time_finite_energy_energy_slowconv}). 
Interestingly enough, the zero-energy case can also be viewed in terms of the dynamics of a Calogero-Moser-Sutherland integrable system, as we briefly discuss at the end of this section.

\begin{df}\label{def:symmetric_Dyson_type_potential} 
We say that a Dyson-type potential $\DPotential = -\log\partfn$ is \emph{symmetric} if
\begin{align} \label{eqn:sym_def_intro}\tag{{\textsc{sym}}}
- \partial_j \, \DPotential (\btheta) = \newphi^j (\btheta) = \sum_{\substack{1 \leq i \leq n \\[.1em] i\neq j}} \pairderiv(\theta^j - \theta^i), \qquad \textnormal{for all } j \in \{1,\ldots,n\} , 
\end{align}
for an odd function (pair interaction) $\pairderiv \in C^1 (\chamberSingle\setminus\{0\}, \bR)$ satisfying
\begin{align*}
\pairderiv (\pi) =0 \qquad \textnormal{and} \qquad \lim_{\theta\to 0^+}\pairderiv(\theta)\in (0, +\infty] .
\end{align*}
Furthermore, we say that such $\DPotential$ is \emph{separately (strictly) convex} if moreover 
\begin{align} \label{eqn:deriv_bound_def_intro}\tag{{\textsc{cvx}}}
\constphi := - \sup_{\theta\in (0, 2\pi)}  \pairderiv' (\theta) > 0.
\end{align}
\end{df}

In particular, the Dyson-type potential~\eqref{eq:DBM_choices} is symmetric: $\pairderiv (\theta) = \two \cot(\frac{\theta}{2})$, and separately convex: $\constphi=\ONE$. 
Our analysis of the large-time behavior of finite-energy systems relies on the assumption that $\constphi$ is strictly greater than zero --- 
indeed, this constant appears in the rate of convergence for zero-energy systems in \Cref{thm:U_equally_spaced}. 
Moreover, by~\eqref{eqn:def_of_drift}, the condition $\constphi > 0$ implies that $\Delta \DPotential (\btheta) > 0$ for all $\btheta \in \chamber$ --- that is, $\DPotential$ is strictly subharmonic
(the mean curvature of the graph of $\DPotential$ is strictly positive). Note that this is a weaker condition than 
convexity\footnote{More precisely, by~\eqref{eqn:def_of_drift}, the condition $\constphi > 0$ implies that the Hessian matrix of $\DPotential$ has strictly positive diagonal entries and strictly negative off-diagonal entries, not guaranteeing convexity as such.}.
The strict subharmonicity also implies that $\DPotential$ satisfies a Poisson equation with a strictly positive source term, and in the spirit of Coulomb gas or electrostatics, 
$\DPotential$ could thus be regarded of as a potential associated to a strictly positive density.
Combined with the differential inequalities~\eqref{eq:differential_inequalities}, the condition $\constphi > 0$ results in
\begin{align*}
0 < \constmult \, \constphi \, n (n-1)< \constmult\, \Delta \DPotential(\btheta) \le | \nabla \DPotential (\btheta) |^2 + \constaddnew , \qquad \textnormal{for all }  \; \btheta \in \chamber.
\end{align*}

\begin{thm}[Asymptotic configuration of finite-energy systems]
\label{thm:U_equally_spaced}
Fix an integer $n \geq 2$. Let $\DPotential$ be a symmetric and separately convex Dyson-type potential. 
If $\nBessel(\btheta) < \infty$, then 
\begin{align}\label{eq:U_equally_spaced}
\lim_{t \to \infty}(\theta^{j+1}_t-\theta^j_t) = \frac{2\pi}{n} , \qquad \textnormal{for all } j \in \{1, \ldots, n\}.
\end{align}
Furthermore, if $\nBessel(\btheta) = 0$, then there exists $\zeta \in [0, 2\pi)$ such that 
\begin{align} \label{eq:equally-spaced-static}
\lim_{t\to \infty} \btheta_t = \big( \zeta , \, \zeta + \tfrac{2\pi}{n} , \, \ldots , \, \zeta+ \tfrac{(n-1)2\pi}{n} \big) ,
\end{align}
and the convergence is exponentially fast with exponential rate $\constphi n$, 
for $\constphi$ as in~\eqref{eqn:deriv_bound_def_intro}.
\end{thm}

We prove \Cref{thm:U_equally_spaced}
in Section~\ref{subsec:Finite-energy-results}, 
where we also discuss the rate of convergence for finite-energy systems (see \Cref{rem:conv_optimal}~\&~\Cref{prop:finite_time_finite_energy_energy_slowconv}). 

To understand what this result means for the zero-energy curves, let the angle $\zeta\in [0, 2\pi)$ be fixed, and let $\btheta^{\zeta}$ denote the constant configuration 
\begin{align*}
\btheta^{\zeta}_t 
\equiv \btheta^{\zeta}
:= \big( \zeta, \zeta + \tfrac{2\pi}{n}, \ldots, \zeta+ \tfrac{(n-1)2\pi}{n} \big) , \qquad \textnormal{for all } t \geq 0 . 
\end{align*}
By symmetry, we see that the constant driving functions $\exp(\ii \btheta^\zeta_t)$ generate the ``pizza-pie'' configuration of curves: the union of straight lines in $\overline{\bD}$ from the points 
$\exp(\ii(\zeta + 2\pi j/n))$, $1 \leq j \leq n$, to the origin. 
\Cref{thm:U_equally_spaced} implies that for large enough times, the zero-energy driving functions approach this configuration. 
Thus, we expect that the union of curves $g_\bigT(\bgamma_{[\bigT,\infty)})$ approaches the pizza-pie configuration,
though we do not attempt to prove this in the present article. 

In the special case where the potential is given by~\eqref{eq:DBM_choices}, 
when considering zero-energy systems in the context of Hamiltonian dynamics, the associated particle system is called the
\emph{Calogero-Moser-Sutherland} system (or sometimes the trigonometric Calogero-Moser system). 
Its study dates back to the original articles~\cite{Calogero:Solution_of_one-dimensional_n-body_problems, Sutherland:Exact_results_for_a_quantum_many-body_problem_in_one-dimension_II, Moser:Three_integrable_Hamiltonian_systems_connected_with_isospectral_deformations}.
Interestingly, it is known in particular that the equilibrium states of systems with the Dyson-type and Calogero-type potentials coincide (see~\cite{Calogero:Equilibrium_configuration} for the case of particles on the real line). 
To see this, let us consider the (quantum) Calogero-Sutherland  Hamiltonian
\begin{align*}
H^{\kappa}(\btheta) 
\; := \;  \frac{1}{2} \sum_{j=1}^n \partial_j^2 \; + \; \frac{4-\kappa}{\kappa^2} \sum_{j=1}^n \partial_j \sum_{\substack{1 \leq i \leq n \\[.1em] i \neq j}} \cot \bigg(\frac{\theta^j - \theta^i}{2} \bigg) ,
\end{align*}
or rather, its classical limit (Calogero-Moser-Sutherland Hamiltonian)
\begin{align*}
H(\btheta) = \lim_{\kappa \to 0+} \kappa^2 H^{\kappa}(\btheta) 
= \; \; & 4 \sum_{j=1}^n \partial_j \sum_{\substack{1 \leq i \leq n \\[.1em] i \neq j}} \cot \bigg(\frac{\theta^j - \theta^i}{2} \bigg) 
\; = \; 4 \sum_{j=1}^n \partial_j^2 \log \nradpartfn (\btheta) ,
\end{align*}
using the choice~\eqref{eq:DBM_choices} for the potential $- 2 \log \nradpartfn$ with $\nradpartfn(\btheta) = \smash{\underset{i < j}{\prod} \, \sin^2 \big(\frac{\theta^j - \theta^i}{2}\big)}$.
The associated Hamiltonian equation for the momenta $\tfrac{\ud}{\ud t} \theta_t^j$ reads
\begin{align*}
\tfrac{\ud^2 }{\ud t^2} \theta_t^j
= - \partial_j H(\btheta_t)
= - 4 \sum_{i=1}^n \partial_j \partial_i \phi^i(\btheta)
= 4 \sum_{\substack{1 \leq j \leq n \\[.1em] j \neq i}} \cot \bigg(\frac{\theta^i - \theta^j}{2} \bigg) \csc^2 \bigg(\frac{\theta^i - \theta^j}{2} \bigg) , 
\qquad t \geq 0 ,
\end{align*}
with $\phi^i = \smash{\underset{j\neq i}{\sum} \, \cot \big( \frac{\theta^i - \theta^j}{2} \big) }$. 
Interestingly enough, our Proposition~\ref{prop:ODE_existence_uniqueness} also yields
\begin{align*}
\tfrac{\ud^2 }{\ud t^2} \theta_t^j 
= 2 \tfrac{\ud}{\ud t} \phi^j (\btheta_t) 
= 4 \sum_{\substack{1 \leq i \leq n \\[.1em] i \neq j}} \big( \partial_i \phi^j (\btheta_t) \big) \, \phi^i (\btheta_t) 
= 4 \sum_{\substack{1 \leq j \leq n \\[.1em] j \neq i}} \cot \bigg(\frac{\theta^i - \theta^j}{2} \bigg) \csc^2 \bigg(\frac{\theta^i - \theta^j}{2} \bigg) , 
\end{align*}
since $\newphinrad^j=2\phi^j$, so we recover the Calogero-Moser-Sutherland equations of motion.

Connections between chordal $\SLE_{0}$ and the (rational) Calogero-Moser systems have appeared recently in~\cite{ABKM:Pole_dynamics_and_an_integral_of_motion_for_multiple_SLE0}, 
and the preprint~\cite{Zhang:Multiple_radial_SLE0_and_classical_Calogero-Sutherland_system} also discusses the Calogero-Moser-Sutherland case. 
The existence and uniqueness (up to rotation) of a stable equilibrium for the latter system has been considered, e.g., in~\cite{Muller:2D_locus_configurations_and_the_trigonometric_Calogero-Moser_system}, though our proofs were developed independently. 
Instead of leveraging the connection to Hamiltonian dynamics, our approach depends on explicit analysis of the deterministic PDE obtained by setting $\kappa = 0$ in~\eqref{eqn:driving_spde},
and it applies to a more general setting, often relevant in SLE theory. 
The existence and uniqueness of the zero-energy flow for each starting point $\btheta_0$ is stated in Proposition~\ref{prop:ODE_existence_uniqueness}.

\subsection*{Acknowledgments}

We would like to thank several anonymous referees for their careful reading of our work and for providing insightful comments, which have helped us to improve this article and clarify many arguments and concepts. We also thank Fredrik Viklund for useful comments.

This project was initiated while~V.O.H.~and~E.P.~were participating in a program hosted by the Mathematical Sciences Research Institute (MSRI) in Berkeley, California, in Spring 2022, US,  
supported by the National Science Foundation (Grant No.~DMS-1928930).

Part of this project was performed while E.P.~was participating in a program hosted by the Institute for Pure and Applied Mathematics (IPAM) in Los Angeles, California, US, in Spring 2024,  
supported by the National Science Foundation (Grant No.~DMS-1925919).

Part of this project was performed while the authors were participating in a program hosted by the Hausdorff Research Institute for Mathematics (HIM) in Bonn, Germany, in Summer 2025,  
supported by the Deutsche Forschungsgemeinschaft (DFG, German Research Foundation) under Germany's Excellence Strategy EXC-2047/1-390685813.

O.A.~is supported by the Academy of Finland grant number 340461 ``Conformal invariance in planar random geometry.''

V.O.H.~is partially supported by an AMS Simons Research Enhancement Grant for PUI Faculty.
V.O.H.~would like to thank Nestor Guillen for useful conversations about PDEs.

This material is part of a project that has received funding from the  European Research Council (ERC) under the European Union's Horizon 2020 research and innovation programme (101042460): 
ERC Starting grant ``Interplay of structures in conformal and universal random geometry'' (ISCoURaGe) 
and from the Academy of Finland grant number 340461 ``Conformal invariance in planar random geometry.''
E.P.~is also supported by 
the Academy of Finland Centre of Excellence Programme grant number 346315 ``Finnish centre of excellence in Randomness and STructures (FiRST)'' 
and by the Deutsche Forschungsgemeinschaft (DFG, German Research Foundation) under Germany's Excellence Strategy EXC-2047/1-390685813, 
as well as the DFG collaborative research centre ``The mathematics of emerging effects'' CRC-1060/211504053.

\section{LDP for Dyson-type diffusions on the circle}\label{sec:Bessel_LDP}

In this section, we prove our first main result, Theorem~\ref{thm:Bessel_LDP_general}, which in particular yields  
a finite-time LDP for the $n$-radial Bessel process (equivalently, for Dyson Brownian motion on the circle), Theorem~\ref{thm:Bessel_LDP}. 
We will justify the upper and lower bounds of type~(\ref{eq: limsup claim basic},~\ref{eq: liminf claim basic}) separately, 
by applying Varadhan's lemma (Lemma~\ref{lemma:Varadhan}) and 
relying on Schilder's theorem for Brownian motion (Theorem~\ref{thm:Schilder}) as key input.
We however first need to control the difference of the rate function $\nBessel_\bigT$ to the usual
Dirichlet energy $\nDenergy_{\bigT}$ appearing in Schilder's theorem 
--- see in particular Definition~\ref{def:Phi_kappa} and Lemmas~\ref{lem: limit of Phi kappa}~\&~\ref{lem: J as sum of independent and interaction terms}. 
Moreover, because the change of measure from independent Brownian motions 
contains a factor that is not uniformly bounded, we need a specific tail estimate (\Cref{lem:bad event bound}).

Before addressing the proof of the main result, in  Section~\ref{subsec: DMB and Bessel background} 
we gather definitions of the various diffusions (Dyson Brownian motion and $n$-radial Bessel processes), 
and in Section~\ref{sec:multiBessel_construction} we explicitly describe the setup in the context of changes of measures --- see in particular \Cref{prop:martingale(new)} and \Cref{prop:Bessel_measure}.
We then address salient properties of the Dyson-Dirichlet energy $\nBessel_\bigT$ (Section~\ref{subsec: Multiradial Dirichlet Energy}), the rate function in the LDP.
We finally prove the main \Cref{thm:Bessel_LDP_general} at the end of Section~\ref{subsec: proof of Bessel LDP}.

\subsection{Dyson Brownian motion and $n$-radial Bessel process}
\label{subsec: DMB and Bessel background}

\begin{df}[\cite{Healey-Lawler:N_sided_radial_SLE}] \label{def:radial_Bessel_DBM}
The \emph{$n$-radial Bessel process} on $\tfrac{1}{2}\chamber := \{\btheta \; | \;  2\btheta\in \chamber \}$ with parameter $\alpha \in \bR$ is the process 
$\bBessel_t = \bBessel^\alpha_t = (\Bessel^1_t, \ldots, \Bessel^n_t)$
satisfying
\begin{align}\label{eqn:n-radial_bessel} 
\ud \Bessel^j_t 
= \; & \alpha \sum_{\substack{1 \leq i \leq n \\[.1em] i\neq j}} \cot \big( \Bessel^j_t - \Bessel^i_t \big) \ud t + \ud W^j_t, \qquad \textnormal{for all } j \in \{1,\ldots,n\} , 
\end{align}
where $W^1_t, \ldots, W^n_t$ are independent Brownian motions.
\emph{Dyson Brownian motion on the circle} is the process 
$e^{2 \ii \bBessel_t} = (e^{2 \ii \Bessel^1_t}, \ldots, e^{2 \ii \Bessel^n_t})$. 
Note that $\bU_t := 2\bBessel_{\kappa t/4}^\alpha$ on $\chamber$ satisfies
\begin{align}\label{eqn:time_change_bessel}
\ud U^j_t = 
\frac{\alpha}{2} \kappa \sum_{\substack{1 \leq i \leq n \\[.1em] i\neq j}} \cot \bigg( \frac{U^j_t - U^i_t}{2}\bigg) \ud t + \sqrt \kappa \ud W^j_t ,
\qquad \textnormal{for all } j \in \{1,\ldots,n\} .
\end{align}
In particular, the SDE~\eqref{eqn:driving_spde} appearing in Theorem~\ref{thm:Bessel_LDP}  is~\eqref{eqn:time_change_bessel} with $\alpha=4/\kappa$. 
\end{df}

The half-angle convention in the definition of the $n$-radial Bessel process 
is also convenient for direct comparison with Dyson Brownian motion $(X^1_t, \ldots, X^n_t)$ on the real line, which satisfies the SDEs~\eqref{eqn:DBM_real}.
For random matrix theory applications, this process is more commonly written using the time change $\tilde{X}^j_t = X^j_{t/n}$ so that~\eqref{eqn:DBM_real} is equivalent to
\begin{align*}
\ud \tilde{X}^j_t = 
\frac{1}{n}\sum_{\substack{1 \leq i \leq n \\[.1em] i\neq j}} \frac{1}{\tilde{X}^j_t - \tilde{X}^i_t} \ud t + \sqrt{\frac{2}{n\beta}} \ud W^j_t ,
\qquad \textnormal{for all } j \in \{1,\ldots,n\} .
\end{align*}
Using $\alpha=4/\kappa$ in Definition~\ref{def:radial_Bessel_DBM}, 
we find the relationship $\beta=8/\kappa$ between $\beta$ and $\kappa$~\cite{Cardy:SLE_and_Dyson_circular_ensembles}.

Dyson Brownian motion on the circle is also referred to in the literature as the \emph{Dyson circular ensemble}~\cite{FWW:Multiple_Ising_interfaces_in_annulus_and_2N-sided_radial_SLE}.
Although Definition~\ref{def:radial_Bessel_DBM}
holds for any $\alpha$, we will restrict our attention to $\alpha \geq 1$, which corresponds to $\kappa\leq 4$ in Equation~\eqref{eqn:beta_kappa_dependence}. 
Comparison to the usual Bessel process shows that $\alpha=1$ corresponds to the phase transition for recurrence and transience. The existence of a unique strong solution to~\eqref{eqn:n-radial_bessel} (for any $\alpha\geq\frac{1}{2}$) follows from the analogous result for Dyson Brownian motion~\cite{AGZ:An_introduction_to_random_matrices}(see also~\cite{Lawler:SLE_book_draft}).

\subsection{Dyson-type diffusions via change of measure} 
\label{sec:multiBessel_construction}

We now describe the general theory of how to obtain a diffusion of the form~\eqref{eqn:driving_spde} via a change of measure. 
In particular, the construction of $n$-radial Bessel process from~\cite{Healey-Lawler:N_sided_radial_SLE} uses this method, as described below in \Cref{rmk:HL_Bessel}. This perspective will be necessary for the application of Varadhan's Lemma in the proof of \Cref{thm:Bessel_LDP_general}.

Suppose that $\bB_t = (B^1_t, \ldots, B^n_t)$ is an $n$-dimensional standard Brownian motion in $\bR^n$ 
defined on the filtered probability space $(\Omega, \mathcal F_t, \bP)$, 
where $\mathcal F_\bullet$ is its natural right-continuous completed filtration. 
Fix $\btheta_0 \in \chamber$ and define $\bDyson_t = (\Dyson^1_t,\ldots, \Dyson^n_t)$ on $\chamber$ by
\begin{align} \label{eqn:def_U_kappa}
\bDyson_t = \btheta_0 + \sqrt \kappa \, \bB_t , 
\qquad
\bDyson_0 = \btheta_0 ,
\end{align}
stopped at the collision time
\begin{align*}
\blowuptime := \; & \inf \Big\{t \geq 0 \; \colon \; \min_{1 \leq i < j \leq n} \, \big| e^{\ii \Dyson^j_t} - e^{\ii \Dyson^i_t} \big| = 0 \Big\} 
= \inf \big\{ t\geq 0 \; \colon \; \bDyson_t \nin \chamber \big\} .
\end{align*}

\begin{proposition}[{See,~e.g.,~\cite[Equation~(42)]{Healey-Lawler:N_sided_radial_SLE} and 
\cite[Theorem~5.3.2]{Lawler:Stochastic_calculus}}]
\label{prop:generalmartingale}
Let $\kappa >0$ and $\bDyson$ as in~\eqref{eqn:def_U_kappa}. 
Let $\partfn \in C^2(\chamber,[0,\infty))$ with $\partfn(\btheta_0)>0$.
There exists a stopping time $0<\blowuptimegen\leq \blowuptime$ such that the process on $(\Omega, \mathcal F_t, \bP)$ defined as\footnote{Here, $\Delta$ is the Laplacian operator.} 
\begin{align}\label{eqn:gen_local_mart}
M^\partfn_t := \partfn(\bDyson_t) \, 
\exp \bigg(\! -\frac{1}{2} \int_0^t \frac{\Delta \partfn (\bDyson_s)}{\partfn(\bDyson_s)} \ud s \bigg),
\qquad t < \blowuptimegen,
\end{align}
is a continuous nonnegative local martingale satisfying
\begin{align}\label{eqn:Mart_alpha}
\frac{\ud M^\partfn_t}{M^\partfn_t}
= \sqrt{\kappa} \, \sum_{j=1}^n \newphi_\partfn^j(\bDyson_t) \ud B^j_t , 
\qquad t < \blowuptimegen,
\end{align}
where $\newphi_\partfn^j=\partial_j \log \partfn$ as in~\eqref{eqn:def_of_drift}.
Moreover, if $\PF_t$ is the probability measure absolutely continuous with respect to $\prob_t = \bP|_{\mathcal F_t}$ with Radon-Nikodym derivative $M^\partfn_t/M^\partfn_0$, then
\begin{align}\label{eqn:diffusion_in_new_measure_generalized}
\ud \bDyson_t = \kappa \, \bnewphi_\partfn(\bDyson_t) \ud \bs{t}  + \sqrt{\kappa} \ud \bs{W}_{\! t} , \qquad t < \blowuptimegen , 
\end{align}
where $\bs{W}_{\! t} = (W^1_t, \ldots, W^n_t)$ 
are independent standard Brownian motions with respect to $\bP^{\partfn}_t$.
\end{proposition}

\begin{proof}
As $\partfn$ is a nonnegative $C^2$ function, $M^\partfn_t$ defined in~\eqref{eqn:gen_local_mart} is continuous and nonnegative. 
It\^o's formula implies that $M^\partfn_t$ is a local martingale satisfying~\eqref{eqn:Mart_alpha}, and  the
Cameron-Martin-Girsanov theorem (in a local martingale form) implies that this change of measure yields Equation~\eqref{eqn:diffusion_in_new_measure_generalized} 
--- see, e.g.~\cite[Chapter~VIII]{Revuz-Yor:Continuous_martingales_and_Brownian_motion} or~\cite[Section~5]{Lawler:Stochastic_calculus} for details. 
In particular (as in~\cite[Theorem~5.3.2]{Lawler:Stochastic_calculus}, for example), we may take the stopping time
\begin{align*}
\blowuptimegen = \lim_{m\to \infty} \inf \Big \{t  \ge 0 \;\colon\; M^\partfn_t + \sqrt{\kappa} \, \sum_{j=1}^n \big| \newphi_\partfn^j(\bDyson_t) \big| =m \Big \} 
\end{align*}
which clearly satisfies  $0<\blowuptimegen\leq \blowuptime$ and for which the tilting is valid.
\end{proof}

Observe that the appearance of $\kappa$ in Equations~(\ref{eqn:Mart_alpha},~\ref{eqn:diffusion_in_new_measure_generalized}) comes from the definition of $\bDyson_t$ in~\eqref{eqn:def_U_kappa}.
A more standard application of the Cameron-Martin-Girsanov theorem is obtained by setting $\kappa=1$ in \Cref{prop:generalmartingale} (or performing a linear time change). However, we state the result with general $\kappa>0$  since we are interested in large deviations as $\kappa \to 0$.

We will use \Cref{prop:generalmartingale} for powers of partition functions --- including powers that depend on $\kappa$. 
To avoid confusion caused by the associated exponent, we now state\footnote{In \Cref{cor:power_of_partfn} and the subsequent remarks, ``$(\cdot)^a$'' denotes the $a$:th power, for $a>0$, while elsewhere this article, superscripts have been reserved for coordinates (e.g.,~$i,j,k$) and parameters (e.g.,~$\kappa, \alpha, \spr$).} 
the specific version of this result that we will use.

\begin{corollary}\label{cor:power_of_partfn}
Let $\kappa >0$ and $\bDyson$ as in~\eqref{eqn:def_U_kappa}.
Let $\partfn \in C^2(\chamber,[0,\infty))$ with $\partfn(\btheta_0)>0$, and let $a>0$. 
Let $M^{\partfn, a}$ be the continuous nonnegative local martingale on $(\Omega, \mathcal F_t, \bP)$ defined as
\begin{align}\label{eqn:gen_local_mart_power}
M^{\partfn, a}_t := (\partfn(\bDyson_t))^{a} \, 
\exp \bigg(\! -\frac{1}{2} \int_0^t \frac{\Delta (\partfn (\bDyson_s))^{a}}{(\partfn(\bDyson_s))^{a}} \ud s \bigg) , \qquad t < \blowuptimegen,
\end{align}
up to a stopping time $0 < \blowuptimegen\leq \blowuptime$. 
Moreover, if $\prob^{\partfn,a}_t$ is the probability measure absolutely continuous with respect to $\prob_t = \bP|_{\mathcal F_t}$ with Radon-Nikodym derivative $M^{\partfn,a}_t/M^{\partfn,a}_0$
and $\newphi_\partfn^j=\partial_j \log \partfn$, then 
\begin{align}\label{eqn: Dyson_general}
\ud \bDyson_t = \; & a\kappa \, \bnewphi_\partfn(\bDyson_t) \ud \bs{t} 
+ \sqrt \kappa \ud \bs{W}_{\! t} , \qquad t < \blowuptimegen , 
\end{align}
where $\bs{W}_{\! t} = (W^1_t, \ldots, W^n_t)$
are independent standard Brownian motions with respect to $\prob^{\partfn,a}_t$.
In~particular, if $a=1/\kappa$, then $\bDyson_t$ satisfies Equation~\eqref{eqn:driving_spde} up to the stopping time $\blowuptimegen$. 
\end{corollary}

\begin{remark}
Let $\DPotential=-\log \partfn$ be a Dyson-type potential. If $a>0$, then $a\, \DPotential = -\log (\partfn^a) $ is also a Dyson-type potential.  
Furthermore, if~\eqref{eq:differential_inequalities} holds for $\partfn$ with constants $\constmult, \constaddnew$, then~\eqref{eq:differential_inequalities} holds for 
the $a$:th power $\partfn^a$ with scaled constants $a\constmult, a^2\constaddnew$.
In particular, the potential obtained from $\partfn$ as in~\eqref{eq:DBM_choices}
is a Dyson-type potential with $\constmult=4$ and $\constaddnew=\frac{4n(n^2-1)}{3}$.
\end{remark}

\begin{remark}\label{rmk:HL_Bessel}
Taking $\partfn=(\nradpartfn)^\two$ as in~\eqref{eq:DBM_choices} and $a=\ONE/\kappa$ in \Cref{cor:power_of_partfn}, so that
\begin{align}\label{eq:nradpartfn_2/kappa}
(\nradpartfn (\btheta))^{2/\kappa} = \prod_{1\leq i<j\leq n} \sin^{4/\kappa}\bigg(\frac{\theta^j-\theta^i}{2}\bigg) , 
\end{align} 
yields Equation~\eqref{eqn:driving_spde_radial}. 
This result is contained in~\cite{Healey-Lawler:N_sided_radial_SLE}, though the authors use a different parameterization convention\footnote{Notice that~\eqref{eq:nradpartfn_2/kappa} is the square of the multiradial $\SLE_\kappa$ partition function; 
see also \Cref{rmk:def_justification}.}. 
\end{remark}

For Dyson-type potentials, the system~\eqref{eqn: Dyson_general} is valid up until the collision time $ \blowuptime$.

\begin{prop}\label{prop:martingale(new)}
If $\DPotential$ is a potential of Dyson type \textnormal{(}as in \Cref{def:Dyson_type_potential}\textnormal{)}, then \Cref{cor:power_of_partfn} holds with
$\blowuptimegen = \blowuptime$. 
Moreover, the following properties hold.
\begin{enumerate}[leftmargin=*]
\item \label{item:martingale(new)1}
For each $T<\blowuptime$,
the stopped process $(M^{\partfn, a}_{t \land T})_{t \geq 0}$ is a positive $\bP$-martingale. 

\item \label{item:martingale(new)2}
If $a\geq \frac{1}{2\constmult}$, then for all $t \geq 0$, we have $\prob^{\partfn,a}_t$-almost surely $\blowuptime >t$.
As a consequence, the unique strong solution $(\bDyson_t)_{t \geq 0}$ to~\eqref{eqn: Dyson_general} exists for all time.
\end{enumerate}
\end{prop}

The proof of \Cref{prop:martingale(new)} follows \Cref{prop:Bessel_measure}.
(These results might be known, or might follow from general theory --- but we have not been able to identify a reference for them.
Therefore, we shall give detailed proofs, also for the sake of potential future usage.)

Going forward, let $\Pr^{\kappa}_t := \prob^{\partfn,1/\kappa}_t$ 
be the measure absolutely continuous with respect to $\bP_t$ with Radon-Nikodym derivative 
obtained from the martingale~\eqref{eqn:gen_local_mart_power} with $a=\ONE/\kappa$:  
\begin{align}\label{eq:the_good_measure}
\frac{\ud \Pr^{\kappa}_t}{\ud \bP_t} 
= \frac{\mgle_t}{\mgle_0} , \qquad t < \blowuptime ,
\qquad \textnormal{where} \quad \mgle_t := M^{\partfn, \ONE/\kappa}_t, 
\end{align}
where we keep the dependence on the Dyson-type potential $\DPotential = - \log \partfn$ implicit throughout.

\begin{cor}\label{prop:Bessel_measure} 
Suppose that $\DPotential = -\log \partfn$ is a potential of Dyson type \textnormal{(}as in \Cref{def:Dyson_type_potential}\textnormal{)}. 
If $\kappa \in (0, \two\constmult]$, 
then $\bDyson$ as in~\eqref{eqn:def_U_kappa} is the unique strong solution to the system of SDEs~\eqref{eqn:driving_spde} in the measure $\Pr^{\kappa}$ and $\blowuptime=\infty$ almost surely.

\noindent 
In particular, if $\partfn$ is chosen as in~\eqref{eq:DBM_choices}  
and $\kappa \in (0,8]$, then $\bDyson$ satisfy~\eqref{eqn:time_change_bessel} with $\alpha=4/\kappa$.
\end{cor}

The particular case of~\eqref{eq:DBM_choices} is covered by~\cite{Healey-Lawler:N_sided_radial_SLE} (see also~\cite{AGZ:An_introduction_to_random_matrices}), 
though the argument there is different (namely, \cite{Healey-Lawler:N_sided_radial_SLE} uses a comparison to the usual Bessel process). 

\begin{proof}
\Cref{prop:martingale(new)} implies that under $\Pr^{\kappa}$, the collision time~\eqref{eq: blowup time} is infinite when 
$\kappa \in (0, \two\constmult]$. 
In particular, setting $a=\ONE/\kappa$ 
in~\eqref{eqn: Dyson_general}, the process $\bDyson$ indeed satisfies the system of SDEs~\eqref{eqn:driving_spde}, 
where $\bs{W}_{\! t} = (W^1_t, \ldots, W^n_t)$ are independent $\Pr^{\kappa}$-Brownian motions.
The SDEs~\eqref{eqn:time_change_bessel} for $\partfn$ as in~\eqref{eq:DBM_choices} follow because $\constmult=4$ in this case.
\end{proof}

\begin{proof}[Proof of \Cref{prop:martingale(new)}]
Applying a direct computation and the local martingale form of the Cameron-Martin-Girsanov theorem (e.g.,~\cite[Theorem~5.3.2]{Lawler:Stochastic_calculus}), 
we see that \Cref{cor:power_of_partfn} holds for the stopping time
\begin{align*}
\blowuptimegen = \lim_{m\to \infty} T_m , 
\qquad \textnormal{where} \qquad
T_m := \inf \Big \{t \ge 0 \;\colon\;  M^{\partfn, a}_{t} + a\sqrt{\kappa} \, \sum_{j=1}^n  \big|\newphi_\partfn^j(\bDyson_t) \big| = m \Big \}.
\end{align*}
Let us first argue that $\blowuptimegen = \blowuptime$.  
As $\partfn$ is a $C^2$ function on $\chamber$ and $a>0$, a blowup of $ M^{\partfn, a}_{t}$ or $\newphi_\partfn^j(\bDyson_t)$ can only happen when $\partfn(\bDyson_t)$ approaches $0$.  
In particular, the sequence  $(\DPotential(\bDyson_{T_m}))_{m\in\bN}$ is unbounded as $m\to \infty$.
However, for any $\epsilon>0$, since $\DPotential$ is $C^2$ and $\chamberCl{\epsilon} = \{\btheta\in \chamber \; | \; \DeltaMin_{\btheta} \ge \epsilon \}$ is compact\footnote{Recall that $\smash{\DeltaMin_{\btheta} 
:= \underset{1\leq j\leq n}{\min} \, \big| \theta^{j+1} - \theta^j \big| \in \big[0,  \tfrac{2\pi}{n} \big]}$.}, 
the potential $\DPotential$ is bounded on $\chamberCl{\epsilon}$, so $\bDyson_{T_m}\to \partial \chamber$ as $m\to \infty$.
This implies that $\blowuptimegen = \blowuptime$, so the conclusion of \Cref{cor:power_of_partfn} holds with
$\blowuptimegen= \blowuptime$. 

Item \ref{item:martingale(new)1} now follows by a standard argument 
(e.g.,~\cite[Theorem~5.3.2]{Lawler:Stochastic_calculus}), which we include here (as it is very short). 
Assume that $T<\blowuptime$.  
Then, for any $0\leq t \leq T$, we have
\begin{align*}
1 = \prob^{\partfn,a}_t [\blowuptime >  t]
= \lim_{\epsilon \to 0} \prob^{\partfn,a}_t  [\blowuptimegen_\epsilon >  t]
= \lim_{\epsilon \to 0}\bE_t \bigg[\frac{M^{\partfn, a}_{t }}{M^{\partfn, a}_{0}} \, \one {\{\blowuptimegen_\epsilon > t \}}\bigg]
\leq \bE_t \bigg[\frac{M^{\partfn, a}_{t}}{M^{\partfn, a}_{0}}  \bigg].
\end{align*}
On the other hand, $\bE_t \big [M^{\partfn, a}_{t \land T}  \big] \leq M^{\partfn, a}_{0}$, since $M^{\partfn, a}_{t \land T}$ is a supermartingale (as it is a non-negative local martingale), 
so we conclude that $M^{\partfn, a}_{t \land T}$ is a martingale. 
This shows Item~\ref{item:martingale(new)1}.

It remains to prove Item~\ref{item:martingale(new)2}. 
The idea of the proof is similar to the proof in~\cite[Lemma~4.3.3]{AGZ:An_introduction_to_random_matrices} that Dyson Brownian motion is noncolliding for $\beta\geq 1$, 
but our use of the potential $\DPotential$ generalizes the argument. 
(In fact, the potential $\DPotential$ plays a similar role to the particular (stochastic) Lyapunov function used in \cite[Lemma~4.3.3]{AGZ:An_introduction_to_random_matrices}.) 
First, we introduce a cutoff to obtain a system with uniformly Lipschitz drift. For each $\epsilon>0$, we define the auxiliary system $\baux{}_t = \baux{\epsilon}_t = (\aux^1_t(\epsilon), \ldots, \aux^n_t(\epsilon))$ by
\begin{align}\label{eqn:aux_Lip_system}
\ud \baux{\epsilon}_t = a \kappa \, \bsmoothing (\baux{\epsilon}_t)  \ud \bs{t} 
+ \sqrt \kappa \ud \bs{W}_{\! t} , 
\end{align}
where $\bs{W}_{\! t} = (W^1_t, \ldots, W^n_t)$ are independent standard Brownian motions, 
and the drift functions $\bsmoothing = (\smoothing^1, \ldots, \smoothing^n)$ are defined by
\begin{align}
\smoothing^j (\btheta) :=
\begin{cases}
\newphi^j(\btheta) = \partial_j  \log \partfn (\mathbf \btheta), & \btheta \in \chamberCl{\epsilon} , \\[.5em]
\displaystyle \min \Big\{ \newphi^j(\btheta) , \, \max_{\bvartheta \in \chamberCl{\epsilon}} \newphi^j(\bvartheta) \Big\} , & \btheta \notin \chamberCl{\epsilon} .
\end{cases}
\end{align}
Notice that for each $\epsilon > 0$, the processes $\baux{\epsilon}_t$ and $\bDyson_t$ agree until the exit time
\begin{align*} 
\tau_\epsilon :=\;& \inf \big\{t \ge 0 \; \colon \; \baux{\epsilon}_t \notin \chamberCl{\epsilon}  \big\} \; \leq \; \blowuptime .
\end{align*}
In order to prove Item~\ref{item:martingale(new)2}, it suffices to show that for every $t\geq 0$, we have 
$\prob^{\partfn,a}_t \big[ \tau_\epsilon \leq t \big]\to 0$ as $\epsilon \to 0$. 
Indeed, the family $\{\blowuptimegen_\epsilon\}_{\epsilon>0}$ of random variables  is increasing as $\epsilon \to 0$, 
so the convergence $\blowuptimegen_\epsilon \to \infty$ in probability implies almost sure convergence (by applying the Borel-Cantelli lemma to a subsequence, like in~\cite{AGZ:An_introduction_to_random_matrices}). 
To this end, we will instead show that for each $t\geq 0$, we have 
$\prob^{\partfn,a}_t \big[ T_{\upper} \leq t \big]\to 0$ as $\epsilon \to 0$, 
where
\begin{align*}
T_{\upper} := \inf \big\{t\geq 0 \; \colon \; \DPotential(\baux{\epsilon}_t)\geq \upper \big\} \; \leq \; \tau_\epsilon , \qquad
\upper := \underset{\bvartheta \in \chamberCl{\epsilon}}{\max} \, \DPotential(\bvartheta) .
\end{align*}
By the first property in~\eqref{eqn:growth_conditions}, we know that $\smash{\underset{\epsilon \to 0}{\lim} \,\upper = \infty}$. 
Also,  for each $j$ we have 
\begin{align*}
\partial_j\, \DPotential(\blambda_t) = - \newphi^j(\blambda_t) 
\qquad \textnormal{and} \qquad
\partial_j^2 \, \DPotential (\blambda_t) =  - \partial_j \newphi^j(\blambda_t),
\qquad \textnormal{for all $t \leq T_{\upper} \leq \tau_\epsilon$} .
\end{align*}
Therefore, It\^o's formula gives 
\begin{align} \label{eqn:Lyapunov_system}
\ud \DPotential(\blambda_t) 
= \; & - \sum_{j=1}^n  \newphi^j(\blambda_t) \Big( a \kappa \newphi^j (\blambda_t) \ud t + \sqrt \kappa \ud W^j_t\Big)
- \frac{ \kappa}{2} \, \sum_{j=1}^n  \partial_j \newphi^j(\blambda_t)  \ud t \nonumber \\
= \; &  -a \kappa \sum_{j=1}^n \big( \newphi^j(\blambda_t) \big)^2 \ud t
-  \frac{\kappa}{2} \, \sum_{j=1}^n  \partial_j \newphi^j(\blambda_t)  \ud t 
+ \underbrace{ \sqrt \kappa \, \sum_{j=1}^n  \newphi^j(\blambda_t)  \ud W^j_t}_{=: \; \ud N_t},
\end{align}
where $(N_{t \land T_{\upper}})_{t \geq 0}$ is a $\prob^{\partfn,a}_t$-martingale with zero expectation. 
Applying the upper bound in~\eqref{eq:differential_inequalities}, the drift term in~\eqref{eqn:Lyapunov_system} simplifies to 
\begin{align*}
\frac{\kappa}{2} \sum_{j=1}^n \Big( \! -2 a \big(\newphi^j(\blambda_t)\big)^2 -  \, \partial_j \newphi^j(\blambda_t) \Big) 
\;\leq \; & \frac{\kappa \constaddnew}{2 \constmult} +\frac{\kappa}{2}\sum_{j=1}^n \Big( \!-2 a \big(\newphi^j(\blambda_t)\big)^2 +\frac{1}{\constmult} \, \big( \newphi^j(\blambda_t) \big)^2  \Big) \\
\;= \;\; & \frac{\kappa \constaddnew}{2\constmult} 
\; + \; \frac{\kappa}{2}  \sum_{j=1}^n \Big( \frac{1}{\constmult} - 2 a  \Big)\, \big( \newphi^j(\blambda_t) \big)^2 .
\end{align*}
If $a\geq \frac{1}{2 \constmult}$, then the last term is nonpositive, so the drift is upper bounded and
\begin{align*}
\bE^{\partfn,a}_t \big[ \DPotential (\blambda_{t \land T_{\upper}}) \big] 
\; \leq \; \frac{\kappa \constaddnew}{2 \constmult} \,  
\bE^{\partfn,a}_t [t \land T_{\upper} ] + \DPotential(\blambda_0) 
\; \leq \; \frac{\kappa \constaddnew}{2 \constmult} \,  t + \DPotential(\blambda_0),
\end{align*}
where $\bE^{\partfn,a}_t$ denotes the expectation with respect to $\prob^{\partfn,a}_t$. 
Now, if $t\geq T_{\upper}$, then we have $\DPotential(\blambda_{t\land T_{\upper}}) = \DPotential(\blambda_{T_{\upper}}) \geq \upper$, so we find that
\begin{align*}
\prob^{\partfn,a}_t  \big[ T_{\upper} \leq t \big]
\; \leq \;\; & \frac{ \bE^{\partfn,a}_t \big[ \, \DPotential(\blambda_{t\land T_{\upper}}) \one \{T_{\upper} \leq t\} \, \big]}{\upper}
\\ 
\; \leq \; \; & \frac{ \bE^{\partfn,a}_t  \big[ \, \DPotential(\blambda_{t\land T_{\upper}}) \, \big]}{\upper}
\; \leq \; \frac{1}{\upper} \bigg( \frac{\kappa \constaddnew}{2 \constmult}  \,  t + \DPotential(\blambda_0) \bigg) 
\quad \overset{\epsilon \to 0}{\longrightarrow} \quad 0 .
\end{align*}
This concludes the proof. 
\end{proof}

\subsection{Dyson-Dirichlet energy and basic properties}
\label{subsec: Multiradial Dirichlet Energy}

From Proposition~\ref{prop:Bessel_measure}, we learn that the Dyson-type process $\bDyson_t$ solving~\eqref{eqn:driving_spde}, 
equivalent to the $n$-radial Bessel process by~\eqref{eqn:time_change_bessel}, 
and to the Dyson Brownian motion on the circle via $\exp(\ii \bDyson_t) = \exp(2 \ii \, \bBessel_{\kappa t/4}^\alpha )$ with $\alpha=4/\kappa$, 
is a Girsanov transform of $n$-dimensional standard Brownian motion $\bB$. 
From Schilder's classical theorem, one readily obtains an LDP for $n$-dimensional Brownian motion $\bB$, whose components are independent (Theorem~\ref{thm:Schilder}).

Denote by $\nCameronMartinSpace{\bigT]}{\boldsymbol{0}}$ the space of continuous functions $\btheta \colon [0,\bigT] \to \bR^n$ started at $\btheta_0 = \boldsymbol{0}$, 
equipped with the supremum norm $\smash{\supnorm{\btheta} := \underset{t \in [0,\bigT]}{\sup} \, |\btheta_t|}$. The rate function in Schilder's theorem is the $n$-dimensional Dirichlet energy 
\begin{align}\label{eqn:nDirichlet_energy_def}
\nDenergy_{\bigT}(\btheta) 
:= \sum_{j=1}^n \nDenergy_{\bigT}(\theta^j) , 
\qquad \btheta = (\theta^1, \ldots, \theta^n) \in \nCameronMartinSpace{\bigT]}{\boldsymbol{0}},
\end{align}
where $\nDenergy_{\bigT}(\theta)$ is the Dirichlet energy of $\theta \in \CameronMartinSpace{\bigT]}{0}$\textnormal{:}
\begin{align}\label{eqn:Dirichlet_energy_def}
\nDenergy_{\bigT}(\theta)
:= \; &
\begin{dcases} 
\tfrac{1}{2} \int_0^\bigT \big| \tfrac{\ud}{\ud t} \theta_t \big|^2 \ud t ,
\quad & \textnormal{if } \theta  
\textnormal{ is absolutely continuous on $[0,\bigT]$,} \\
+\infty , & \textnormal{otherwise.}
\end{dcases}
\end{align}

\begin{remark}
The (Cameron-Martin) space of absolutely continuous functions on $[0,\bigT]$ with square-integrable derivative
coincides with the Sobolev space $\nWonetwoSpace{\bigT]}{\boldsymbol{0}}$ that has the norm $|| \cdot ||_{1,2; [0,\bigT]}$ given by
\begin{align*}
|| \btheta ||_{1,2; [0,\bigT]} 
:= 
\bigg( \sum_{j=1}^n \int_0^\bigT |\theta^j_t|^2 \ud t 
+ \sum_{j=1}^n \int_0^\bigT \big| \tfrac{\ud}{\ud t} \theta^j_t \big|^2 \ud t
\bigg)^{1/2} , \qquad \btheta \in \nWonetwoSpace{\bigT]}{\boldsymbol{0}} ,
\end{align*}
thanks to the ACL characterization of Sobolev spaces~\cite[Lemma~A.5.2]{AIM:Elliptic_partial_differential_equations_and_quasiconformal_mappings_in_the_plane}
(note that as such, this fails for $\bigT = \infty$). 
We will thus identify all these spaces:
\begin{align*} 
\nHoneSpace{\bigT]}{\boldsymbol{0}} 
= \nWonetwoSpace{\bigT]}{\boldsymbol{0}}  
= \big\{ \btheta \in \nCameronMartinSpace{\bigT]}{\boldsymbol{0}} \; | \; \nDenergy_{\bigT} (\btheta) < \infty \big\}.
\end{align*}
Let us also note that if $\btheta \in \nHoneSpace{\bigT]}{\boldsymbol{0}}$, then $\btheta$ is $\tfrac{1}{2}$-H\"older continuous. 
\end{remark}

\begin{theorem}[{Direct consequence of Schilder's theorem; see, e.g.,~\cite[Chapter~5.2]{Dembo-Zeitouni:Large_deviations_techniques_and_applications}}]\label{thm:Schilder}

Fix $\bigT \in (0,\infty)$. 
The process $\big(\sqrt \kappa \, \bB_t \big)_{t\in [0,\bigT]}$ satisfies the following LDP 
in $\nCameronMartinSpace{\bigT]}{\boldsymbol{0}}$, 
with good rate function $\nDenergy_{\bigT}$\textnormal{:}

For any closed subset $\closed$ and open subset $\open$ of $\nCameronMartinSpace{\bigT]}{\boldsymbol{0}}$, we have
\begin{align*}
\limsup_{\kappa \to 0+} \kappa \log \bP \big[ \sqrt \kappa \, \bB_{[0, \bigT]} \in \closed \big] \leq - \inf_{\btheta \in \closed} \nDenergy_{\bigT}(\btheta) , \\
\liminf_{\kappa \to 0+} \kappa \log \bP \big[ \sqrt \kappa \, \bB_{[0, \bigT]} \in \open \big] \geq - \inf_{\btheta \in \open} \nDenergy_{\bigT}(\btheta).
\end{align*}
\end{theorem}

A convenient tool for proving an LDP when a family of probability measures is absolutely continuous with respect to another family for which an LDP is already known is provided by the classical Varadhan's lemma.
We will use it in combination with Theorem~\ref{thm:Schilder}.

\begin{lemA}[{Varadhan's lemma; see, e.g.,~\cite[Lemmas~4.3.4 and 4.3.6]{Dembo-Zeitouni:Large_deviations_techniques_and_applications}}]
\label{lemma:Varadhan}

Suppose that the probability measures $(\bP^\kappa)_{\kappa>0}$ satisfy an LDP in a topological space $\spaceX$
with good rate function $E$. 
Let $\Phi \colon \spaceX \to \bR$
be a function bounded from above. 
Then, the following hold.
\begin{enumerate}
\item \label{item1_Var} 
If $\Phi$ is upper semicontinuous, then for any closed subset $\closed$ of $\spaceX$, 
\begin{align*}
\limsup_{\kappa\to 0+} \kappa \log \bE^\kappa \Big[\exp\Big(\frac{1}{\kappa} \Phi(X) \Big) \one \{X \in \closed \} \Big]
\leq - \inf_{x \in \closed} \big(E(x) - \Phi(x)\big) .
\end{align*}
\item \label{item2_Var} 
If $\Phi$ is lower semicontinuous, then for any open subset $\open$ of $\spaceX$,
\begin{align*}
\liminf_{\kappa\to 0+} \kappa \log \bE^\kappa \Big[\exp\Big(\frac{1}{\kappa} \Phi(X) \Big) \one \{X \in \open \} \Big]
\geq - \inf_{x \in \open} \big( E(x) - \Phi(x)\big).
\end{align*}
\end{enumerate}
\end{lemA}

In order to apply Varadhan's lemma in the measure $\Pr^{\kappa}$ defined in~\eqref{eq:the_good_measure} and appearing in \Cref{prop:Bessel_measure}, we first find a suitable function $\Phi^{\kappa}_t$ 
so that the Radon-Nikodym derivative from the 
martingale $\mgle := M^{\partfn, \ONE/\kappa}_t$ defined in~\eqref{eqn:gen_local_mart_power} with $a=\ONE/\kappa$ takes the form 
\begin{align} \label{eq: desired RN}
\frac{\mgle_t}{\mgle_0} =
\exp \bigg( \frac{1}{\kappa} \Phi^{\kappa}_t (\bB) \bigg) , \qquad t \geq 0 . 
\end{align}
We see that, because 
\begin{align*}
\frac{\Delta (\partfn (\bDyson_t))^{\ONE/\kappa}}{(\partfn(\bDyson_t))^{\ONE/\kappa}}
= \; & \sum_{j=1}^n \Big( ( \partial_j \newphi^j)(\bDyson_t)
+ \frac{\ONE}{\kappa} \big( \newphi^j(\bDyson_t) \big)^2 \Big) , \qquad t \geq 0 ,
\end{align*}
\eqref{eq: desired RN} holds with the following definition of the interaction functional $\Phi^{\kappa}_t (\bB) =: \Phi^{\kappa}_t (\bDyson)$. 
\begin{df} \label{def:Phi_kappa}
Suppose that $\DPotential = -\log \partfn \geq 0$ is a potential of Dyson type \textnormal{(}as in \Cref{def:Dyson_type_potential}\textnormal{)}\footnote{The nonnegativity follows from the assumption that any potential $\DPotential = - \log \partfn$ of Dyson type is obtained from a function $\partfn \colon \chamber \to (0,1]$. (Note that this does not hold if we include a nonzero spiraling rate.)}. 
Let $\kappa \geq 0$ and $\bigT \in (0,\infty)$. 
The \emph{Dyson-type interaction functional} is
\begin{align}\label{eqn:Phi_kappa}
\begin{split}
\Phi^{\kappa}_\bigT = \PhiGen_\bigT 
\; \colon \; & \BesselSpace{\bigT]}{\btheta_0} \to \bR , \\
\Phi^{\kappa}_\bigT(\btheta) 
:= \; & \DPotential(\btheta_0) -  \DPotential(\btheta_\bigT)
\; - \; \frac{1}{2} \sum_{j=1}^n \int_0^\bigT \Big( 
\kappa \, ( \partial_j \newphi^j)(\btheta_s)
+ \big( \newphi^j(\btheta_s) \big)^2 \Big) \ud s .
\end{split}
\end{align} 
\end{df}

\begin{remark}
When the potential $\partfn$ is chosen as in~\eqref{eq:DBM_choices}, we have 
\begin{align}\label{eqn:Phi_identity}
\Phi^{\kappa}_t(\btheta) 
= \;& 2 \log \frac{\nradpartfn(\btheta_t)}{\nradpartfn(\btheta_0)} + \kappa \frac{n (n^2 - 1)}{6} \, t + \frac{(\kappa-4)}{2} \int_0^t \sum_{j=1}^n (\phi^j(\btheta_s))^2\ud s ,
\end{align}
with $\phi^j$ as in~\eqref{eq:DBM_choices}. 
This follows from~\eqref{eqn:trig_identity}, or equivalently,~\cite[Lemma~5.1]{Healey-Lawler:N_sided_radial_SLE}.
\end{remark}

In the next lemma, we gather useful properties of the functional $\Phi^{\kappa}_\bigT$.

\begin{lem} \label{lem: limit of Phi kappa}
Fix $\bigT \in (0,\infty)$ and $\btheta_0 \in \chamber$. 
Equation~\eqref{eqn:Phi_kappa} defines a continuous functional with respect to the metric~\eqref{eqn:sup-metric}, and for each $\btheta \in \BesselSpace{\bigT]}{\btheta_0}$, we have 
\begin{align} \label{eq: Phi_limit} 
\Phi^{\kappa}_\bigT (\btheta) + \kappa  \, \frac{\constaddnew T}{2\constmult} \, 
= \,  \Phi^{0}_\bigT(\btheta) - \frac{\kappa}{\two} \bigg( \sum_{j=1}^n \int_0^\bigT (\partial_j \newphi^j)(\btheta_s) \ud s 
- \frac{\constaddnew T}{\constmult} \bigg)
\quad \overset{\kappa \to 0+}{\longrightarrow} \quad 
\Phi^{0}_\bigT(\btheta) ,
\end{align} 
and this limit is monotonically decreasing. 
Furthermore, 
 $\Phi^{0}_\bigT$ is bounded from above as
\begin{align} \label{eqn: Phi0_upper_bound}
\Phi^{0}_\bigT(\btheta) \, \leq \,    \DPotential(\btheta_0) 
 = -  \log \partfn (\btheta_0) 
 , \qquad \btheta \in \BesselSpace{\bigT]}{\btheta_0} .
\end{align}
\end{lem}

\begin{proof}
By the lower bound in~\eqref{eq:differential_inequalities}, we have $(\sum_j\partial_j\newphi^j )- \frac{\constaddnew}{\constmult} \le 0$, 
so the limit~\eqref{eq: Phi_limit} is monotonically decreasing. 
The bound~\eqref{eqn: Phi0_upper_bound} follows from~\eqref{eqn:Phi_kappa} and the non-negativity of the Dyson-type potential $\DPotential$. 
The continuity is clear. 
\end{proof}

The following technical tail estimate is needed in the proof of \Cref{thm:Bessel_LDP}. 
To state it, we use the notation $\chamber^\epsilon := \{\btheta\in \chamber \; | \; \DeltaMin_{\btheta} > \epsilon \}$ and $\smash{\DeltaMin_{\btheta} 
:= \underset{1\leq j\leq n}{\min} \, \big| \theta^{j+1} - \theta^j \big| \in \big[0,  \tfrac{2\pi}{n} \big]}$.

\begin{lem} \label{lem:bad event bound}
Fix $\kappa \in (0, \constmult]$. 
Fix $\epsilon > 0$ and consider the stopping time 
\begin{align} \label{eq:stopping time}
\tau_\epsilon :=\;& \inf \big\{t \ge 0 \; \colon \; \bDyson_t \notin \chamberCl{\epsilon} \big\} \; \leq \; \blowuptime , 
\end{align}
for the process $\bDyson_t = (\Dyson^{1}_t,\ldots, \Dyson^n_t)$ satisfying the SDEs~\eqref{eqn:driving_spde} 
under the measure $\Pr^{\kappa}$ in~\eqref{eq:the_good_measure}. 
For each initial configuration $\bDyson_0 = \btheta_0 \in \chamber$, 
there exist constants $\upbd = \upbd(\epsilon, \btheta_0, \partfn) \in (0,\infty)$
and $C = C(n, \bigT, \partfn) \in (0,\infty)$ 
independent of $\kappa$ such that $\underset{\epsilon \to 0}{\lim} \, \upbd(\epsilon, \btheta_0, \partfn) = \infty$ and 
\begin{align} \label{eqn:bad event bound}
\Pr^{\kappa}\big[ \tau_\epsilon \le \bigT\big]
\le C \, e^{-\upbd/\kappa} ,
\end{align}
\end{lem}

\begin{proof}
Recall (cf.~\Cref{prop:Bessel_measure}) 
that $\Pr^{\kappa}$ is the probability measure absolutely continuous with respect to $\bP$ with Radon-Nikodym derivative~\eqref{eq: desired RN}, given by the martingale (rather than simply a local martingale)
$\mgle_t = M^{\partfn, \ONE/\kappa}_t$ from \Cref{prop:martingale(new)}. 
Therefore, since $\tau_{\epsilon}\wedge \bigT$ is a stopping time bounded by $\bigT$, by  the optional stopping theorem (OST), we have
\begin{align} \label{eqn:OST}
\begin{split}
\bE_\bigT
\bigg[ \frac{\mgle_\bigT}{\mgle_0} \;
\one \{\tau_\epsilon \le \bigT\} \bigg]  
= \; & \bE_\bigT\bigg[ \bE_\bigT \bigg[ \frac{\mgle_{\bigT} }{\mgle_0} \; \one\{\tau_\epsilon \le \bigT\} \; \Big \vert \; \mathcal F_{\tau_\epsilon \wedge \bigT}\; \bigg]\bigg] \\
= \; & 
\bE_\bigT\bigg[ \frac{\mgle_{\tau_{\epsilon} \wedge \bigT} }{\mgle_0} \; \one\{\tau_\epsilon \le \bigT\}\bigg]
= \; 
\bE_\bigT\bigg[ \frac{\mgle_{\tau_{\epsilon} } }{\mgle_0} \; \one\{\tau_\epsilon \le \bigT\}\bigg].
\end{split}
\end{align}
Thus, we obtain 
\begin{align*}
\Pr^{\kappa}\big[ \tau_\epsilon \le \bigT\big]
= \; &  \bE_\bigT \bigg[
\exp \bigg( \frac{1}{\kappa} \Phi^{\kappa}_\bigT (\bDyson) \bigg) \; \one \{\tau_\epsilon \le \bigT\} \bigg]
\\ 
= \; &  \bE_\bigT \bigg[
\exp \bigg( \frac{1}{\kappa} \Phi^{\kappa}_{\tau_\epsilon} (\bDyson) \bigg) \; \one \{\tau_\epsilon \le \bigT\} \bigg] 
&& \textnormal{[by OST,~\eqref{eqn:OST}]}
\\
\leq \;& \bE_\bigT \bigg[
\exp \bigg( 
\frac{\ONE}{\kappa} \log \frac{\partfn(\bDyson_{\tau_\epsilon})}{\partfn(\btheta_0)} 
\, + \,\frac{ \constaddnew}{\two \constmult} \, \tau_\epsilon 
\bigg) \; \one \{\tau_\epsilon \le \bigT\} \bigg] 
&& \textnormal{[by~\eqref{eq:differential_inequalities}; see below]}
\\
\leq \;& \exp\Big(\frac{ \constaddnew}{\two \constmult} \, \bigT \Big) \, \Bigg(\frac{\underset{\btheta \in \partial \chamber^\epsilon}{\max} \, \partfn(\btheta)}{\partfn(\btheta_0)}\Bigg)^{\ONE/\kappa} ,
\end{align*}
where we used the upper bound in~\eqref{eq:differential_inequalities} to control 
\begin{align}\label{eqn:control_constants}
\begin{split}
\; & - \; \frac{1}{2}  \sum_{j=1}^n\int_0^{\tau_\epsilon} \bigg( 
( \partial_j \newphi^j)(\btheta_s)
+ \frac{1}{\kappa} \, \big( \newphi^j(\btheta_s) \big)^2 \bigg) \ud s
\\
\leq \; & \tau_\epsilon\bigg( \frac{\constaddnew}{2 \constmult} + \Big(\frac{1}{2\constmult} - \frac{1}{2\kappa}\Big)\sum_{j=1}^n \big( \newphi^j(\btheta_s) \big)^2 \bigg)
\leq \frac{  \constaddnew}{\two\constmult} \, \tau_\epsilon , \qquad \kappa \leq \constmult. 
\end{split}
\end{align}
Hence, the desired bound~\eqref{eqn:bad event bound} holds with 
$C(n, \bigT, \partfn) = \exp\big(\frac{\constaddnew}{\two \constmult} \, \bigT \big)$ and 
\begin{align*}
\upbd(\epsilon, \btheta_0, \partfn) 
= - \log \Bigg(\frac{\underset{\btheta \in \partial \chamber^\epsilon}{\max} \, \partfn(\btheta)}{\partfn(\btheta_0)}\Bigg)  
\qquad \xrightarrow[]{\epsilon \to 0} 
\qquad + \infty ,
\end{align*}
thanks to the first limit in~\eqref{eqn:growth_conditions}. This concludes the proof. 
\end{proof}

We next express the Dyson-Dirichlet energy $\nBessel_\bigT$ of Definition~\ref{def:multiradial_Dirichlet_energy}  
in terms of the functional $\Phi^{0}_\bigT$ of \Cref{def:Phi_kappa} 
and the sum of independent Dirichlet energies~\eqref{eqn:nDirichlet_energy_def} appearing in Schilder's theorem, denoted $\nDenergy_{\bigT}$. 

\begin{lem} \label{lem: J as sum of independent and interaction terms}
Fix $\bigT \in (0,\infty)$ and $\btheta_0 \in \chamber$. 
For any $\btheta \in \BesselSpace{\bigT]}{\btheta_0}$, we have
\begin{align} \label{eq: J sum of independent and interaction}
\nBessel_\bigT(\btheta) 
= \nDenergy_{\bigT}(\btheta) - \Phi^{0}_\bigT(\btheta) ,
\qquad \textnormal{where} \qquad 
\nDenergy_{\bigT}(\btheta) 
:= \sum_{j=1}^n \nDenergy_{\bigT}(\theta^j) .
\end{align}
\end{lem}

\begin{proof}
On the one hand, if $\btheta$ is not absolutely continuous,  
then $\nBessel_\bigT(\btheta)=\infty$, 
since the last term on the righthand side of~\eqref{eqn:Phi_kappa} with $\kappa=0$ is negative.
On the other hand, if $\btheta$ is 
absolutely continuous, then using the identity $\newphi^j = \partial_j \log \partfn$, we have 
\begin{align*}
\DPotential(\btheta_0) - \DPotential(\btheta_\bigT)
= \log \frac{\partfn(\btheta_\bigT)}{\partfn(\btheta_0)}
\; = \; \int_0^\bigT \sum_{j=1}^n \big( \tfrac{\ud}{\ud s} \theta^j_s \big) \, 
\newphi^j(\btheta_s)  \ud s ,
\end{align*}
which together with \Cref{def:Phi_kappa} implies that
\begin{align*}
\nDenergy_{\bigT}(\btheta) - \Phi^{0}_\bigT (\btheta)
= \; &  \frac{1}{2} \, \int_0^\bigT \sum_{j=1}^n \big| \tfrac{\ud}{\ud s} \theta^j_s \big|^2 \ud s
\; - \; \frac{1}{2} \int_0^\bigT  \, 
\sum_{j=1}^n \Big( 2 \newphi^j(\btheta_s) \big( \tfrac{\ud}{\ud s} \theta^j_s \big)
- \big(  \newphi^j(\btheta_s) \big)^2 \Big) \ud s  \\
= \; & \frac{1}{2} \int_0^\bigT \sum_{j=1}^n \big| \tfrac{\ud}{\ud s} \theta^j_s -  \newphi^j(\btheta_s) \big|^2 \ud s 
\; = \; 
\nBessel_\bigT(\btheta) .
\end{align*}
\end{proof}

As a corollary, we may characterize finite-energy drivers in finite time as non-colliding drivers having finite Dirichlet energy. 

\begin{cor}\label{cor:finite-energy-characterization}
We have $\nBessel_\bigT(\btheta) < \infty$ 
if and only if $\nDenergy_{\bigT}(\theta^j) < \infty$ for all $j \in \{1,\ldots,n\}$~and 
\begin{align} \label{eqn: deterministic collision time}
\detblowuptime = \detblowuptime(\btheta) := \; & \inf \Big\{t \geq 0 \; \colon \; \min_{1 \leq i < j \leq n} \, \big| e^{\ii \theta^j_t} - e^{\ii \theta^i_t} \big| = 0 \Big\} 
\;  > \; \bigT .
\end{align}
\end{cor}

\begin{proof}
If $\detblowuptime > \bigT$, then we see from \Cref{def:Phi_kappa} that $\Phi^{0}_\bigT(\btheta) > -\infty$. 
If furthermore $\nDenergy_{\bigT}(\theta^j) < \infty$ for all $j$, then we see from~\eqref{eq: J sum of independent and interaction} in \Cref{lem: J as sum of independent and interaction terms} that $\nBessel_\bigT(\btheta) < \infty$. 
This proves the converse implication.
To prove (the contraposition of) the forward implication, note that $\Phi^{0}_\bigT$ is bounded from above by~\eqref{eqn: Phi0_upper_bound}, 
so $\nDenergy_{\bigT}(\theta^j) = +\infty$ readily implies that $\nBessel_\bigT(\btheta) = +\infty$, 
while if $\detblowuptime \leq \bigT$, then monotonicity of the energy in time gives
\begin{align*}
\nBessel_\bigT(\btheta) \geq \nBessel_{\detblowuptime}(\btheta) 
= \; & \limsup_{t \to \detblowuptime-} \nBessel_t(\btheta) \\
= \; & \limsup_{t\to\detblowuptime-} 
\big( \nDenergy_t(\btheta) - \Phi^{0}_t (\btheta) \big) 
&& \textnormal{[by~\eqref{eq: J sum of independent and interaction} in \Cref{lem: J as sum of independent and interaction terms}]}
\\
\geq \; & \limsup_{t\to\detblowuptime-} 
\big( \DPotential(\btheta_t) - \DPotential(\btheta_0) \big) 
&& \textnormal{[by~\eqref{eqn:Phi_kappa}, as $\nDenergy_t(\btheta) \ge 0$]} \\
= \; & +\infty .
&& \textnormal{{[by~\eqref{eqn:growth_conditions}, as $\DPotential = - \log \partfn$]}}
\end{align*}
This concludes the proof.
\end{proof}

We also have a similar (unidirectional) result for infinite time, to be used in \Cref{sec:zero-energy}.

\begin{cor}\label{cor:infinite-energy-characterization}
If $\nBessel(\btheta) := \underset{\bigT \to \infty}{\lim} \, \nBessel_\bigT(\btheta)<\infty$, then $\nDenergy(\btheta) := \underset{\bigT \to \infty}{\lim} \, \nDenergy_\bigT(\btheta)<\infty$.
\end{cor}

\begin{proof} 
Using Lemma~\ref{lem: J as sum of independent and interaction terms} and the definition~\eqref{eqn:Phi_kappa} of $\Phi^{0}_\bigT$, 
we see that
\begin{align*}
\nDenergy (\btheta) 
= \; &  \lim_{\bigT \to \infty } \big(\nBessel_\bigT (\btheta) + \Phi^{0}_T(\btheta) \big) 
&& \textnormal{[by~\eqref{eq: J sum of independent and interaction} in \Cref{lem: J as sum of independent and interaction terms}]}
\\
\leq \; & \lim_{\bigT \to \infty } \nBessel_\bigT (\btheta) +  \DPotential(\btheta_0) 
\, = \, \nBessel (\btheta) +  \DPotential(\btheta_0) 
\, < \, \infty .
&& \textnormal{[by~\eqref{eqn: Phi0_upper_bound}, as $\nBessel_T (\btheta) \ge 0$ ]}
\end{align*}
\end{proof}

\begin{remark}\label{rem:finite-energy-characterization_spiral}
The above proofs of~\Cref{cor:finite-energy-characterization}~\&~\Cref{cor:infinite-energy-characterization} 
require that $\DPotential$ is bounded from below; in fact, without loss of generality, Dyson-type potentials are assumed to be non-negative (\Cref{def:Dyson_type_potential}), 
which corresponds to $\partfn$ taking values in $(0,1]$. (Note that this does not hold if we include a nonzero spiraling rate as in \Cref{cor:Bessel_LDP_spiral}.)
Let us note, however, that \Cref{cor:finite-energy-characterization} also holds more generally --- for example, it follows from \Cref{rem:spiral_SLE} that it does hold for multiradial Dirichlet energy with spiral. 
\end{remark}

\subsection{LDP for Dyson-type diffusions on the circle}
\label{subsec: proof of Bessel LDP}

Recall (e.g., from~\cite{Dembo-Zeitouni:Large_deviations_techniques_and_applications}) that for a topological space $\spaceX$, a~\emph{rate function} $\lenergy$ is a lower semicontinuous mapping $\lenergy \colon \spaceX \to [0, +\infty]$ (i.e,~for all $c \in [0, \infty)$, the level set $\lenergy^{-1}[0,c]$ is a closed subset of $\space X$). 
We note that in Theorems \ref{thm:Bessel_LDP} and \ref{thm:radial_LDP_finite_time}, the space $\spaceX$ is a metric space ($\BesselSpace{\bigT]}{\btheta_0}$, or $\cC$, respectively), so it is sufficient to check the lower semicontinuity property on sequences. 
Recall also that a~\emph{good rate function} is a rate function for which all level sets are compact subsets of $\space X$, which implies lower semicontinuity.

\begin{lem}\label{lemma:good_rate_function}
The Dyson-Dirichlet energy $\nBessel_\bigT$ in Definition~\ref{def:multiradial_Dirichlet_energy}  is a good rate function.
\end{lem}

\begin{proof}
First, observe that the Dirichlet energy of $\btheta \in C([0, T], \chamber)$ can be bounded from above in terms of the Dyson-Dirichlet energy as
\begin{align*}
\nDenergy_{\bigT}(\btheta) 
\leq \; & \nDenergy_{\bigT}(\btheta) - \Phi^{0}_\bigT(\btheta) + \DPotential(\btheta_0) 
&& \textnormal{[by~\eqref{eqn: Phi0_upper_bound}]}
\\
= \; & \nBessel^0_\bigT(\btheta) + \DPotential(\btheta_0) .
&& \textnormal{[by~\eqref{eq: J sum of independent and interaction} in \Cref{lem: J as sum of independent and interaction terms}]}
\end{align*}
Next, fix $c \ge 0$, and let $(\btheta_m)_{m \in \bN}$ be a sequence in
$\nBessel_\bigT^{-1}[0,c]  \subset \nDenergy_\bigT^{-1}[c+ \DPotential(\btheta_0)]$. 
Since $\nDenergy_{\bigT}$ is a good rate function (by Schilder's theorem), 
we can pass to a subsequence, also denoted by $(\btheta_m)_{m \in \bN}$, 
which converges to some element
\begin{align*}
\btheta \in \nDenergy_\bigT^{-1} \big[ 0, c + \DPotential(\btheta_0) \big] .
\end{align*}
By continuity of $\Phi^{0}_\bigT$ from \Cref{lem: limit of Phi kappa} 
and lower semicontinuity of $\nDenergy_{\bigT}$ (by Schilder's theorem),
we obtain using Lemma~\ref{lem: J as sum of independent and interaction terms} the estimate
\begin{align*}
\nBessel_\bigT(\btheta) 
\; = \; \nDenergy_{\bigT}(\btheta) - \Phi^{0}_\bigT(\btheta) 
\; \le \; \liminf_{m \to \infty} \big( \nDenergy_{\bigT}(\btheta_m) - \Phi^{0}_\bigT(\btheta_m) \big) 
\; = \; \liminf_{m \to \infty} \nBessel_\bigT(\btheta_m) 
\; \le \; c,
\end{align*}
yielding  
$\btheta \in \nBessel_\bigT^{-1}[0,c]$. 
This shows that $\nBessel_\bigT^{-1}[0,c]$ is compact, so $\nBessel_\bigT$ is a good rate function. 
\end{proof}

We now conclude with the proof of the first main result of the present work. 

\BesselLDPgeneral*

\begin{proof}
We already know that $\nBessel_\bigT$ is a good rate function by Lemma~\ref{lemma:good_rate_function}. 
Hence, it remains to show that for any closed subset $\closed$ and open subset $\open$ of $\BesselSpace{\bigT]}{\btheta_0}$, we have
\begin{align} 
\label{eq: limsup claim}
\limsup_{\kappa \to 0+} \kappa \log \PrDyson \big[ \bDyson \in \closed \big] 
\; & \leq - \inf_{\btheta \in \closed} \nBessel_\bigT(\btheta) ,
\\
\label{eq: liminf claim}
\liminf_{\kappa \to 0+} \kappa \log \PrDyson \big[ \bDyson \in \open \big] 
\; & \geq - \inf_{\btheta \in \open} \nBessel_\bigT(\btheta) .
\end{align}
We will use Schilder's theorem (Theorem~\ref{thm:Schilder}) combined with Varadhan's lemma (Lemma~\ref{lemma:Varadhan}), 
to prove the upper~\&~lower bounds~(\ref{eq: limsup claim},~\ref{eq: liminf claim}).
The former is the harder one. 

Using the Radon-Nikodym derivative~\eqref{eq: desired RN},
\begin{align} \label{eqn:Radon-Nikodym} 
\frac{\ud \Pr^{\kappa}_\bigT}{ \ud \bP_\bigT} = \frac{\mgle_\bigT}{\mgle_0} 
= \exp \bigg( \frac{1}{\kappa} \Phi^{\kappa}_\bigT (\bDyson) \bigg) ,
\end{align} 
for any Borel set $\Borel \subseteq \BesselSpace{\bigT]}{\btheta_0}$, we have
\begin{align} \label{eq: RN_terms}
\kappa \log \Pr^{\kappa}_\bigT \Big[ \bDyson_{[0,\bigT]} \in \Borel \Big] 
= \; & \kappa \log \bE_\bigT
\bigg[ \exp \bigg( \frac{1}{\kappa} \Phi^{\kappa}_\bigT(\bDyson) \bigg) 
\one \{\bDyson_{[0,\bigT]} \in \Borel\} \bigg] .
\end{align}

\noindent {\bf Lower bound.} 
Fix an open set $\open \subset \BesselSpace{\bigT]}{\btheta_0}$. 
Without loss of generality, we may assume that 
\begin{align*}
M_\open := \inf_{\btheta \in \open} (\nDenergy_{\bigT}(\btheta) - \Phi^{0}_\bigT(\btheta) )  < \infty .
\end{align*} 
Fix $\varepsilon > 0$ and $\btheta{(\varepsilon)} \in \open$ such that $\nDenergy_{\bigT}(\btheta{(\varepsilon)}) - \Phi^{0}_\bigT (\btheta{(\varepsilon)}) \le M_\open + \varepsilon$, which in particular implies that $\Phi^{0}_\bigT(\btheta{(\varepsilon)}) > -\infty$. 
As $\Phi^{0}_\bigT$ is continuous by \Cref{lem: limit of Phi kappa}, we can pick an open neighborhood $\open{(\varepsilon)} \subset \open$ of $\btheta{(\varepsilon)} \in \BesselSpace{\bigT]}{\btheta_0}$
such that $\Phi^{0}_\bigT \ge \Phi^{0}_\bigT(\btheta{(\varepsilon)}) - 1$ on $\open{(\varepsilon)}$.

Now, let $\Phi$ be the lower semicontinuous function equaling $\Phi^{0}_\bigT$ on $\open{(\varepsilon)}$ and $\Phi^{0}_\bigT(\btheta{(\varepsilon)}) - 2$ otherwise. 
Since  $\Phi^{\kappa}_\bigT \ge \Phi^{0}_\bigT - \kappa  \, \frac{\constaddnew T}{2\constmult}$ by \Cref{lem: limit of Phi kappa}, 
by applying Item~\ref{item2_Var} of Varadhan's lemma (Lemma~\ref{lemma:Varadhan}) 
to the set $\open{(\varepsilon)}$ and the function $\Phi$, 
combined with Schilder's theorem (Theorem~\ref{thm:Schilder}), we obtain
\begin{align*}
& \liminf_{\kappa \to 0+} \kappa \log \bE_\bigT
\bigg[ \exp \bigg( \frac{1}{\kappa} \Phi^{\kappa}_\bigT(\bDyson) \bigg)\one{\{\bDyson_{[0,\bigT]} \in \open\}} \bigg] \\
\ge \; & 
\liminf_{\kappa \to 0+} \kappa \log \bE_\bigT
\bigg[ \exp \bigg( \frac{1}{\kappa} \Phi^{0}_\bigT(\bDyson) - \frac{\constaddnew T}{2\constmult} \bigg)\one{\{\bDyson_{[0,\bigT]} \in \open{(\varepsilon)} \}} \bigg] \\
\ge \; & 
\liminf_{\kappa \to 0+} \kappa \bigg(-\frac{\constaddnew T}{2\constmult} \;+ \; \log \bE_\bigT
\bigg[ \exp \bigg( \frac{1}{\kappa} \Phi^{0}_\bigT(\bDyson)  \bigg)\one{\{\bDyson_{[0,\bigT]} \in \open{(\varepsilon)} \}} \bigg] \bigg) \\
\geq \; & 
- \inf_{\btheta \in \open{(\varepsilon)}} (\nDenergy_{\bigT}(\btheta)
- \Phi (\btheta) ) 
\quad \geq \quad 
- (\nDenergy_{\bigT}(\btheta{(\varepsilon)}) - \Phi^{0}_\bigT (\btheta{(\varepsilon)}) ) 
\\
\ge \; &  - M_\open - \varepsilon \qquad \xrightarrow[]{\varepsilon \to 0} 
\qquad - M_\open.
\end{align*}

\noindent {\bf Upper bound.} 
Fix a closed set $\closed \subset \BesselSpace{\bigT]}{\btheta_0}$. By \Cref{lem: limit of Phi kappa}, we have
\begin{align*}
\Phi^{\kappa}_\bigT(\btheta) 
\; \leq \; \Phi^{0}_\bigT(\btheta) + \kappa \Psi(\closed),
\qquad \textnormal{for all }  \btheta \in \closed,
\end{align*} 
and the lower bound in~\eqref{eq:differential_inequalities} implies that
\begin{align} \label{eq:Psi_bdd}
\Psi(\closed) := \sup_{\btheta \in \closed} \bigg[\frac{\constaddnew T}{2 \constmult } \; - \;\sum_{j=1}^n \frac{1}{2} \int_0^\bigT \Big((\partial_j \newphi^j)(\btheta_s)  \,  \Big)\ud s  \bigg]  \;  \ge 0 .
\end{align} 
We now separate the proof into two parts. First, assume that $\Psi(\closed) < \infty$. 
Fix $\varepsilon, M > 0$. Note that $\Phi (\btheta; M, \varepsilon) 
:= \max \{ \Phi^{0}_\bigT(\btheta) + \varepsilon \Psi(\closed) , -M \}$ 
is a continuous function by \Cref{lem: limit of Phi kappa}, 
and $\Phi^{\kappa}_\bigT(\btheta) \leq \Phi (\btheta; M, \varepsilon)$ 
for all $\btheta \in \closed$ and for all $\kappa \in [0,\varepsilon]$. 
Applying Item~\ref{item1_Var} of Varadhan's lemma (Lemma~\ref{lemma:Varadhan}) 
to the set $\closed$ and the function $\Phi$, combined with Schilder's theorem (Theorem~\ref{thm:Schilder}),
we obtain\footnote{Here, the order of limits $M \to \infty$ and $\varepsilon \to 0$ does not matter.} 
\begin{align*}
& \limsup_{\kappa \to 0+} \kappa \log \bE_\bigT
\bigg[ \exp \bigg( \frac{1}{\kappa} \Phi^{\kappa}_\bigT(\bDyson) \bigg)\one{\{\bDyson_{[0,\bigT]} \in \closed\}} \bigg] \\
\leq \; & 
\limsup_{\kappa \to 0+} \kappa \log \bE_\bigT
\bigg[ \exp \bigg( \frac{1}{\kappa} \Phi (\bDyson; M, \varepsilon) \bigg)\one{\{\bDyson_{[0,\bigT]} \in \closed\}} \bigg] \\
\leq \; & - \inf_{\btheta \in \closed} (\nDenergy_{\bigT}(\btheta)
- \Phi (\btheta; M, \varepsilon)) 
\qquad \xrightarrow[\varepsilon \to 0]{M \to \infty} \qquad 
- \inf_{\btheta \in \closed} 
(\nDenergy_{\bigT}(\btheta)
- \Phi^{0}_\bigT(\btheta) ) .
\end{align*}
Next, if $\Psi(\closed) = +\infty$, then we fix $\epsilon > 0$ and consider the stopping time $\tau_\epsilon$~\eqref{eq:stopping time} from \Cref{lem:bad event bound}. 
Note that the set $\closed(\epsilon) := \closed \cap \{\tau_\epsilon \geq \bigT \}$ is closed, and $\Psi(\closed(\epsilon))<\infty$ because 
the upper bound in~\eqref{eq:differential_inequalities} implies that 
\begin{align*}
\Psi(\closed(\epsilon)) 
\; \leq \; \sup_{\btheta \in \closed(\epsilon)} 
\bigg[\frac{\constaddnew T}{ \constmult}\;+\;\sum_{j=1}^n \int_0^\bigT  \frac{(\newphi^j(\btheta_s))^2 }{2\constmult} \ud s \bigg],
\end{align*}
which is finite since any $\btheta \in \closed(\epsilon)$ in particular satisfies $\btheta_t\in\chamberCl{\epsilon/2}$ for all $t\in [0,T]$, so each $\newphi^j(\btheta_s)$ in the integral is uniformly bounded as $\newphi^j$ is continuous. 
Therefore,  we see that  
\begin{align*}
\; & 
\limsup_{\kappa \to 0+} \kappa \log \bE_\bigT
\bigg[ \exp \bigg( \frac{1}{\kappa} \Phi^{\kappa}_\bigT(\bDyson) \bigg)\one{\{\bDyson_{[0,\bigT]} \in \closed\}} \bigg] 
\\ 
\leq \; & 
\limsup_{\kappa \to 0+}\kappa \log 
\Big(\Pr^{\kappa} \big[\bDyson \in \closed(\epsilon)\big] + \Pr^{\kappa}\big[ \tau_\epsilon \le \bigT\big] \Big) 
\end{align*}
is bounded from above by the maximum of the two terms
\begin{align*}
& \; \limsup_{\kappa \to 0+}\kappa \log \Pr^{\kappa} \big[\bDyson \in \closed(\epsilon)\big] 
\; \leq \; - \inf_{\btheta \in \closed(\epsilon)} 
(\nDenergy_{\bigT}(\btheta)
- \Phi^{0}_\bigT(\btheta) ) 
\; \leq \; - \inf_{\btheta \in \closed} 
(\nDenergy_{\bigT}(\btheta)
- \Phi^{0}_\bigT(\btheta) ) , \\
& \; \limsup_{\kappa \to 0+}\kappa \log \Pr^{\kappa}\big[ \tau_\epsilon \le \bigT\big] 
\; \leq \; -\upbd(\epsilon, \btheta_0, \partfn) 
\qquad \xrightarrow[]{\epsilon \to 0} 
\qquad - \infty ,
\end{align*}
where to bound the first term we use the first part of the proof and the fact that $\closed(\epsilon) \subseteq \closed$, 
and $R = R(\epsilon, \btheta_0, \partfn)$ from~\eqref{eqn:bad event bound} in~\Cref{lem:bad event bound} bounds the second term.

\noindent {\bf Conclusion.} 
By Lemma~\ref{lem: J as sum of independent and interaction terms}, 
the asserted inequalities~(\ref{eq: limsup claim},~\ref{eq: liminf claim})
follow from the above bounds together with~\eqref{eq: RN_terms}. 
\end{proof}

\begin{remark}\label{rem:minimizers-exist}
It follows from the goodness of the multiradial Dirichlet energy (\Cref{lemma:good_rate_function}) that it attains its minimum on $\BesselSpace{T]}{\btheta_0}$. 
Moreover, from \Cref{thm:Bessel_LDP_general} we see that the minimum equals zero: 
taking $\closed = \open = \BesselSpace{\bigT]}{\btheta_0}$, 
(\ref{eq: limsup claim},~\ref{eq: liminf claim}) together imply that 
\begin{align*}
\min_{\btheta \in \BesselSpaceSmall{\bigT]}{\btheta_0}} \nBessel_\bigT(\btheta) 
= \inf_{\btheta \in \BesselSpaceSmall{\bigT]}{\btheta_0}} \nBessel_\bigT(\btheta) 
= \lim_{\kappa \to 0+}\kappa \log
\underbrace{\PrDyson \big[ \bDyson \in \BesselSpace{\bigT]}{\btheta_0} \big]}_{= \; 1} = 0.
\end{align*}
\end{remark}

\section{LDP for multiradial $\SLE_{0+}$}\label{sec:SLE_LDP}

The goal of this section is to prove our second main result, \Cref{thm:radial_LDP_finite_time}, which is
a finite-time LDP for the $n$-radial $\SLE_\kappa$ process as $\kappa \to 0+$.
(Readers only interested in the results concerning Dyson-type diffusions may proceed directly to \Cref{sec:zero-energy}.)

Recall that $n$-radial $\SLE$ is defined (\Cref{def:n-radial_SLE_driving_functions}) 
as the Loewner chain whose ($n$-dimensional) driving process is 
obtained from the $n$-radial Bessel process (cf.~\Cref{cor:SLE_kappa_measure}). 
As \Cref{thm:Bessel_LDP} gives an LDP for the $n$-radial Bessel process, it would be convenient to just apply the Loewner transform and use the contraction principle (recalled in \Cref{thm:contraction-principle}) 
to deduce an LDP for multiradial $\SLE$. 
Unfortunately, the standard contraction principle cannot be applied directly, since the Loewner transform~\eqref{eq: Loewner transform} is not continuous for the Hausdorff metric, but only in the  Carath\'eodory sense. 
As the latter topology is not very useful for addressing geometric properties of hulls,
we need to address the discontinuities of the Loewner transform under the Hausdorff metric~\eqref{eq: Hausdorff metric}.
From the chordal case~\cite{Peltola-Wang:LDP_of_multichordal_SLE_real_rational_functions_and_determinants_of_Laplacians}, 
we know that discontinuities of the Loewner transform (for the Hausdorff metric) occur at hulls with non-empty interiors. 
In contrast, we will show that finite-energy hulls are simple radial multichords; 
in particular, they avoid the boundary and each other 
(Sections~\ref{subsec:Derivative estimate}~\&~\ref{subsec:finite-energy-gives-simple-curves} --- see in particular \Cref{thm:finite-energy-gives-simple-curves}).
This allows us to sidestep the discontinuities: we can apply the contraction principle on a smaller space where the Loewner transform is continuous, 
and then extend the LDP to the full space using \Cref{prop:restricted-LDP}~\&~\Cref{lem:caratheodory-vs-hausdorff-convergence} in \Cref{subsec: LDP for SLE}.

Proving that finite-energy hulls are simple radial multichords (\Cref{thm:finite-energy-gives-simple-curves}) is the main work of this section. 
In the chordal case with $n=1$, an analogous result has been verified by two methods. 
On the one hand, 
following the methodology of 
Lind, Marshall, and Rohde~\cite{LMR:Collisions_and_spirals_of_Loewner_traces}, 
Wang used quasiconformal maps to argue that each finite-energy ($n=1$) hull is a quasi-arc~\cite{Wang:Energy_of_deterministic_Loewner_chain}, 
which was later generalized to the case of $n$-multichords in~\cite{Peltola-Wang:LDP_of_multichordal_SLE_real_rational_functions_and_determinants_of_Laplacians}. 
(See~\cite{Abuzaid-Peltola:In_prep} for an elaboration of this approach in the radial case.)

On the other hand, motivated by rough path theory, in~\cite{Friz-Shekhar:Finite_energy_drivers} Friz and Shekhar  
derived a strong derivative estimate for the Loewner uniformizing map near the tip
for drivers with finite Dirichlet  energy~\eqref{eqn:Dirichlet_energy_def}.
This can be used via standard arguments to imply that the Loewner hulls thus obtained are in fact simple curves. 
In the present work, we employ the strategy used 
in~\cite{Friz-Shekhar:Finite_energy_drivers} combined with an analogue of the conformal restriction property~\cite{LSW:Conformal_restriction_the_chordal_case} 
(see Proposition~\ref{prop:removal-of-hulls-loewner}). 
As a by-product, we obtain a radial version of the main theorem of~\cite{Friz-Shekhar:Finite_energy_drivers}, 
but generalized to allow weight functions $\lambda$ --- see \Cref{thm: FS_estimate}.

\subsection{Multiradial Loewner equation and multiradial $\SLE_\kappa$}
\label{subsec: Multiradial Loewner equation background}

We will now consider a more general (weighted) version of the Loewner equation (\Cref{eqn:multiradial_Loewner_general} below, which generalizes~\eqref{eqn:multiradial_Loewner_1common}). 
This allows us to consider Loewner flow with a more general time-dependent parameterization and enables us to reparameterize radial multichords when necessary.  
We restrict our attention to ``nice'' weights, as follows.

\begin{df} A \emph{weight function} is a c\`adl\`ag (i.e., right-continuous with left limits) and locally integrable function
$\lambda \colon [0,\infty) \to (0,\infty)$.
\end{df}
Notice that if $\lambda$ is a weight function, then the map $t \mapsto \int_0^t \lambda_s \ud s$ is strictly increasing 
(hence, it can be used to define a time change), 
and $\lambda$ is bounded on compact time intervals.

\medskip

\noindent \textbf{Multiradial Loewner equation.}
For any weight function $\lambda$, we define the \emph{multiradial Loewner equation with weight $\lambda$} as the boundary value problem 
\begin{align}\label{eqn:multiradial_Loewner_general}
\nradial{\lambda_t}{\tildeg}{t}{z}{w^j}, 
\qquad \tildeg_0(z)=z ,
\qquad z \in \overline{\bD} , \quad t \geq 0 ,
\end{align}
where the driving functions $w^1_t, \ldots, w^n_t \in \partial \bD$ are non-intersecting and continuous in time. 
The solution $\tildeg_t = \tildeg_t^{\lambda}$ to~\eqref{eqn:multiradial_Loewner_general} is called the \emph{Loewner chain with $\lambda$-common parameterization}. 
Abusing terminology, we also refer to the corresponding hulls $(\tildeK_t)_{t \geq 0}$ as a ``Loewner chain.''
Then, $\tildeg_t \colon \bD\smallsetminus \tildeK_t \to \bD$ 
is the uniformizing map normalized at the origin, and in the parameterization in~\eqref{eqn:multiradial_Loewner_general}, we have 
\begin{align*}
\log \tildeg_t'(0) = n \int_0^t \lambda_s \ud s .
\end{align*} 
Note also that the map $\tildeh_t$ related to $\tildeg_t$ via
$\tildeg_t(e^{\ii u}) = \exp( \ii \tildeh_t(u) )$, and with $w_t^j = \exp(\ii \theta_t^j)$ for $1 \leq j \leq n$, satisfies the ODE
\begin{align} \label{eqn:multiradial_Loewner_general_cylinder}
\partial_t \tildeh_t(u) = \lambda_t  \, \sum_{j=1}^n \cot \bigg(\frac{\tildeh_t(u) - \theta_t^j}{2}\bigg) .
\end{align}
We say that the generated hulls $\tildeK_t$ have the \emph{$\lambda$-common parameterization}.

An even more general version of~\eqref{eqn:multiradial_Loewner_general} could be obtained by weighting each term in the sum by a different weight $\lambda^j_t$, which would allow the components of the generated hull to be parameterized at different rates, but this is not needed for the present work. 

\begin{df} \label{def: generate multichord}
Fix $\bigT \in (0,\infty)$. 
Let $\btheta \in \BesselSpace{\bigT]}{}$, let $\lambda$ be a weight function, and let $\tildeg$ and $\tildeK_t$ be as in~\eqref{eqn:multiradial_Loewner_general}.
We say that $\btheta$ \emph{generates a radial multichord} $\bgamma_{[0,\bigT]}$ in $\bD$ 
with the $\lambda$-common parameterization 
if $t \mapsto \gamma^j_t$ is a continuous map from $[0,\bigT]$ to $\overline{\bD}$ such that $\gamma_0^j \in \partial \bD$ for each $1 \leq j \leq n$,
the image $\bgamma_{[0,t]}$ generates $\tildeK_t$ for all $t \in [0,\bigT]$,
and the concatenations of $\gamma^j$ with any simple curves from $\gamma^j_{\bigT}$ to the origin form a radial multichord (as in Definition~\ref{def:radial_multichord}). 
We call the radial multichord $\bgamma_{[0,\bigT]}$  \emph{simple} if its each component $\gamma^j$ is injective, 
$\bgamma_{(0,\bigT]} \subset \bD$, and furthermore 
$\gamma^i_{[0,\bigT]} \cap \gamma^j_{[0,\bigT]}$ for all $1 \leq i \neq j \leq n$. 
\end{df}

When $n=1$, we call $\gamma_{[0,\bigT]}$ a (simple) radial chord in $\bD$ with the $\lambda$-parameterization.
In this case, $\gamma$ is also often referred to as the \emph{Loewner trace} in the literature.

\begin{remark}\label{rmk:weight_function}
In general, the geometry of Loewner hulls depends on both the weight function and the driving function; reparameterizing a hull allows one to focus on whichever is more convenient.
For example, the well-known phase transition for ($n=1$) chordal $\SLE_\kappa$ \cite{Rohde-Schramm:Basic_properties_of_SLE} 
from almost surely simple ($\kappa\leq 4$) to self-touching and to space-filling ($\kappa \geq 8$) can be understood by performing a time change 
so that the driving function is standard Brownian motion $B_t$, and analyzing the resulting weight function.
Indeed, the curves generated by the weighted chordal Loewner equation
\begin{align*}
\partial_t g_t(z) = \frac{ \alpha }{g_t(z)-B_t}, \qquad g_0(z)=z, \qquad z \in \overline{\bH} , \quad t \geq 0 ,
\end{align*}
are almost surely simple if $\alpha\geq 1/2$ and space-filling if $\alpha \leq 1/4$. In this case, the time change allows for comparison between the weight function $\alpha=2/\kappa$ and the parameter of the usual Bessel process on the real line.
In this context, Item~\ref{item: FS_estimate_simple} in \Cref{thm: FS_estimate} is rather surprising: 
there, we show that driving functions with finite energy generate simple radial multichords for \emph{any} weight function that is uniformly bounded away from zero.
\end{remark}

\subsubsection{Multiradial $\SLE_\kappa$, for $\kappa \in (0, 4]$}

The next corollary clarifies the relationship between the measures $\PrDyson$ discussed in \Cref{sec:Bessel_LDP} and $n$-radial $\SLE_\kappa$ processes (\Cref{def:n-radial_SLE_driving_functions}).

\begin{cor}\label{cor:SLE_kappa_measure}
Suppose that $\bB_t = (B^1_t, \ldots, B^n_t)$ is an $n$-dimensional standard Brownian motion in $\bR^n$ 
defined on the filtered probability space $(\Omega, \mathcal F_t, \bP)$, 
where $\mathcal F_\bullet$ is its natural right-continuous completed filtration. 
Fix $\btheta_0 \in \chamber$ and define $\bDyson_t = (\Dyson^{1}_t,\ldots, \Dyson^n_t)  \in \chamber$ by
\begin{align*}
\bDyson_t = \btheta_0 + \sqrt \kappa \bB_t , 
\qquad
\bDyson_0 = \btheta_0 ,
\qquad \textnormal{for $\kappa \in (0,4]$.}
\end{align*}
Let $z^j_t := e^{\ii \Dyson^j_t}$ for $1 \leq j \leq n$. 
Then, in the measure $\PrDyson$ appearing in \Cref{prop:Bessel_measure}, the process $(z^1_t, \ldots, z^n_t)$ 
comprises the driving functions for $n$-radial $\SLE_\kappa$ started from $\bz_0$.
\end{cor}

\begin{proof}
This is an immediate consequence of \Cref{prop:Bessel_measure} and \Cref{def:n-radial_SLE_driving_functions}.
\end{proof}

\begin{remark} \label{rmk:def_justification2} 
In~\cite{Healey-Lawler:N_sided_radial_SLE} the authors consider (for $0<\kappa\leq 4$) a sequence of measures $\mu^\kappa_{t,T}$ which are absolutely continuous with respect to $\bP_t$, 
with Radon-Nikodym derivative which is a large-time $T$ truncation of the chordal Radon-Nikodym derivative~\eqref{eqn:chordal_RN_deriv}.  
For each fixed $t$, as $\bigT \to \infty$, the measures $\mu^\kappa_{t,T}$ converge in the total variation distance to $\PrDyson_t$.  
Combining this convergence result with Corollary \ref{cor:SLE_kappa_measure} justifies the definition of multiradial $\SLE_\kappa$ 
that we use in the present work (\Cref{def:n-radial_SLE_driving_functions}). 
This definition is consistent with the well-known multiradial partition function~\eqref{eq:multiradial_partition_function}, as described in \Cref{rmk:def_justification,rmk:HL_Bessel}.
\end{remark}

Since the present work relies on the construction of $n$-radial $\SLE_\kappa$ in \cite{Healey-Lawler:N_sided_radial_SLE}, it is worthwhile to briefly address our differing choice of parameterization.
In \cite{Healey-Lawler:N_sided_radial_SLE}, for each $\kappa \in (0, 4]$, the authors describe $n$-radial $\SLE_\kappa$ as the Loewner chain generated by 
the multiradial Loewner equation~\eqref{eqn:multiradial_Loewner_general} with weight $\lambda \equiv 4/\kappa$ and driving functions $w^j_t = e^{2 \ii \Bessel^j_t}$, $1 \leq j \leq n$, where 
$\bBessel_t^\alpha = (\Bessel^1_t, \ldots, \Bessel^n_t)$ is the
$n$-radial Bessel process with parameter $\alpha=4/\kappa$ from Definition~\ref{def:radial_Bessel_DBM}\footnote{See~\cite[Theorem~3.12 and the discussion following Corollary~3.13]{Healey-Lawler:N_sided_radial_SLE}.}. 
However, the dependence of the weight function on $\kappa$ (i.e., using the $\lambda \equiv 4/\kappa$-common parameterization) poses complications as $\kappa \to 0+$, so this setup is not amenable to large deviations analysis.
Thus, it will be more convenient for us to consider the corresponding process up to the time change $t \mapsto \frac{\kappa}{4} \, t$, so that the curves have the $1$-common parameterization. 
This ensures that the capacity of the $n$-radial $\SLE_\kappa$ at time $t$ does not depend on $\kappa$. 
In this setup, we consider the uniformizing conformal maps $g_t \colon \bD\smallsetminus K_t \to \bD$ normalized at the origin and satisfying the multiradial Loewner equation~\eqref{eqn:multiradial_Loewner_general} with $\lambda \equiv 1$ (i.e., Equation~\eqref{eqn:multiradial_Loewner_1common}). 
In particular, we note that the Loewner hulls generated by~\eqref{eqn:multiradial_Loewner_1common} with the $1$-common parameterization 
are the same as those generated by~\eqref{eqn:multiradial_Loewner_general} with the $\lambda=4/\kappa$-common parameterization, if 
\begin{align*}
g_{t}(z)=\tildeg_{\kappa t/4}(z) 
\qquad \textnormal{and} \qquad 
z^j_t=w^j_{\kappa t/4}.
\end{align*}
We shall address more general time changes in the next Section~\ref{subsub:Time change}.

\begin{remark} \label{rem:spiral_SLE}
For each $\kappa \leq 4$ and an additional parameter $\spr \in \bR$, 
one can similarly define multiradial $\SLE_\kappa^\spr$ with spiraling rate $\spr$ (and with the common parameterization) as the random radial multichord $\bgamma$ 
for which the uniformizing conformal maps $g_t \colon \bD \smallsetminus \bgamma_{[0,t]} \to \bD$ satisfy~\eqref{eqn:multiradial_Loewner_1common}
with driving functions $z^j_t = e^{\ii \Dyson^j_t}$ for $1 \leq j \leq n$, 
where $\bDyson_t = (\Dyson^1_t, \ldots, \Dyson^n_t)$ is the strong solution in $\BesselSpace{\infty)}{\btheta_0}$ to the SDEs~\eqref{eqn:diffusion_in_new_measure_spiral}
\cite{Miller-Sheffield:Imaginary_geometry4, KWW:Commutation_relations_for_two-sided_radial_SLE, HPW:Multiradial_SLE_with_spiral}. 
Our results apply directly to derive an LDP for this process as well (i.e., a version of \Cref{thm:radial_LDP_finite_time}), 
with good rate function obtained from \Cref{cor:Bessel_LDP_spiral} similarly as in~\eqref{eqn:nradial-energy}. 
Indeed, note that 
\begin{align*}
\nBessel^{\spr}_\bigT(\btheta) 
=\;& \tfrac{1}{2} \int_0^\bigT \sum_{j=1}^n \big| \tfrac{\ud}{\ud s} \theta^j_s - 2 \big( \phi^j(\btheta_s) + \spr \big)\big|^2 \ud s \\
\le\;& \tfrac{1}{2} \int_0^\bigT \sum_{j=1}^n \bigg(2 \, \big| \tfrac{\ud}{\ud s} \theta^j_s - 2 \big( \phi^j(\btheta_s) + \spr' \big)\big|^2  + 8 |\spr'-\spr|^2\bigg)\ud s\\
=\;& 2 \nBessel^{\spr'}_\bigT(\btheta) + 4 n \bigT |\spr'-\spr|^2 , \qquad \spr, \spr' \in \bR .
\end{align*}
Taking $\spr'=0$ we may conclude that the multiradial Dirichlet energy $\nBessel^{\spr}_\bigT$ with spiral is finite if and only if 
the multiradial Dirichlet energy $\nBessel_{\bigT}$ without spiral is finite. 
Thus, \Cref{thm:finite-energy-gives-simple-curves} (proven in Section~\ref{subsec:Derivative estimate}) also holds with the assumption 
$\nBessel_\bigT(\btheta) < \infty$ replaced by the assumption $\nBessel^{\spr}_\bigT(\btheta) < \infty$. 
Using this fact, one can check that also the proof of \Cref{thm:radial_LDP_finite_time} applies verbatim to the spiraling case.
\end{remark}

\begin{remark} 
Multiradial $\SLE_\kappa^\spr$ curves are expected to satisfy the so-called \emph{re-sampling property}:
for each curve $\gamma^j$ in the tuple $\bgamma = (\gamma^1, \ldots, \gamma^n)$, 
conditionally on the other curves $\{\gamma^k \; , k \neq j\}$, the law of $\gamma^j$ is that of the chordal $\SLE_\kappa$ in its natural connected component. 
To prove this property, one should first show that the $n$-radial $\SLE_\kappa$ is supported on radial multichords (as in \Cref{def:radial_multichord}), 
continuous at the origin, 
and elsewhere pairwise disjoint (cf.~\cite{Lawler:Continuity_of_radial_and_two-sided_radial_SLE_at_the_terminal_point}). 
This follows from~\cite{Miller-Sheffield:Imaginary_geometry4} by using a coupling of $\SLE_\kappa$ curves as flow lines of the Gaussian free field,
and is proven in~\cite{HPW:Multiradial_SLE_with_spiral} using SLE techniques. 
We will not need these properties in the present work.
\end{remark}

\subsubsection{Time changes}
\label{subsub:Time change}

\Cref{rmk:weight_function} describes the application of a particular time change to a Loewner chain with $n=1$. More generally, we see that~\eqref{eqn:multiradial_Loewner_general} is related to~\eqref{eqn:multiradial_Loewner_1common} by the following time change. Let $\lambda_t$ and $\tildeg_t$ as in~\eqref{eqn:multiradial_Loewner_general}, and define
\begin{align}\label{eqn:time change}
\sigma(t) : = \int_0^t \lambda_s\ud s, \qquad  \tau(t) : = \sigma^{-1}(t), \qquad  \textnormal{and}\qquad \hatg_t : = \tildeg_{\tau(t)}.
\end{align}
Then, we have $\tfrac{\ud}{\ud t} \tau(t) = 1/\lambda_{\tau(t)}$, so the chain rule shows that $g_t$ satisfies the multiradial Loewner equation~\eqref{eqn:multiradial_Loewner_1common} with $\hatw^j_t = w^j_{\tau(t)}$. 
Consequently, this time change allows us to conveniently move between the $1$-common parameterization and the $\lambda$-common parameterization as needed. 
Most importantly, this allows us to reparameterize radial multichords: the property of having finite truncated energy is preserved under a large class of time changes, as the next lemma states.

\begin{lem}\label{lem:time change} 
Fix $\bigT \in (0,\infty)$. 
Let $\timechange \colon [0,\bigT]\to [0,\sigma(\bigT)]$ be strictly increasing and differentiable, with $\timechange(0)=0$, and suppose that 
$\dot\sigma(t) := \tfrac{\ud}{\ud t} \sigma(t)$ is uniformly bounded away from zero and infinity, i.e., 
\begin{align*} 
\supnorm{\dot\sigma} 
:= \sup_{t \in [0,\bigT]} | \dot\sigma(t) |
\; \in \; (0, \infty)
\qquad \textnormal{and} \qquad 
\supnorm{\tfrac{1}{\dot\sigma}} 
:= \sup_{t \in [0,\bigT]} \tfrac{1}{| \dot\sigma(t) |}
\; \in \; (0, \infty) .
\end{align*}
For $\btheta \in \BesselSpace{\bigT]}{}$, set $\hat \btheta_t := \btheta_{\sigma(t)}$, and $\hat T := \timechange^{-1}(T)$.
Then, we have 
\begin{align*} 
\nBessel_\bigT (\btheta)<\infty
\qquad \textnormal{if and only if} \qquad 
\nBessel_{\hat \bigT} (\hat \btheta)<\infty .
\end{align*}
\end{lem}

In particular, if the time change $\timechange$ is defined by~\eqref{eqn:time change} for $\lambda$ bounded away from zero and infinity, then
the conclusion of \Cref{lem:time change} is that on finite time intervals, finite-energy drivers for 1-common and $\lambda$-common parameterizations coincide up to time change. 

\begin{proof}
Since $\sigma$ is strictly increasing and differentiable, 
we have $\dot\sigma(t) > 0$ for all $t$, so we
can estimate the Dirichlet energy of each $\hat \theta^j$ by 
\begin{align*}
\supnorm{\tfrac{1}{\dot\sigma}} \, \nDenergy_{\bigT}(\theta^j) 
\; \leq \; \nDenergy_{\hat T} (\hat \theta^j) 
\; \leq \; 
\supnorm{\dot\sigma} \, \nDenergy_{\bigT}(\theta^j) .
\end{align*} 
Moreover, we have $\detblowuptime( \btheta)> T$ if and only if $\detblowuptime(\hat \btheta)>\hat T$, and
\Cref{cor:finite-energy-characterization} thus implies that $\nBessel_{\bigT}(\btheta) < \infty$ is equivalent to $\nBessel_{\hat \bigT}(\hat{\btheta}) < \infty$. 
\end{proof}

\subsection{Derivative estimate for finite-energy Loewner chains for $n=1$}
\label{subsec:Derivative estimate}

In this section, we consider solutions to the 
(single) radial Loewner equation~\eqref{eqn:multiradial_Loewner_general_cylinder} (with $n=1$) with some weight function $\lambda \colon [0,\bigT] \to (0,\infty)$. 
A well-known condition for 
the property that the driving function 
$\theta \in \CameronMartinSpace{\bigT]}{}$ generates a radial chord $\gamma_{[0,\bigT]}$ in $\bD$ 
is an estimate for the derivative of the inverse map 
$\tildef_t := \tildeh_t^{-1}$ near the driving function $\theta_t$ (locally) uniformly in time. 
More precisely, to verify the existence of the Loewner trace $\gamma$, it suffices to show the existence of the radial limit at its tip (see, e.g.,~\cite[Theorem~4.1]{Rohde-Schramm:Basic_properties_of_SLE} or~\cite[Theorem~6.4]{Kemppainen:SLE_book}):
\begin{align} \label{eq: limit curve}
\gamma_t := \lim_{y \to 0+} \exp \big( \ii \,  \tildef_t(\theta_t + \ii y) \big) , 
\qquad \textnormal{uniformly for all } t \in [0, T] .
\end{align}
It is not hard to check (see, e.g.,~\cite[Appendix]{Friz-Shekhar:Finite_energy_drivers}  or~\cite[Theorem~3.6]{Rohde-Schramm:Basic_properties_of_SLE}) 
that the limit~\eqref{eq: limit curve} exists uniformly in time if 
there exists a constant $b \in (0, 1)$ such that
\begin{align} \label{eq: uniform bound for f'}
|\tildef_t'(\theta_t + \ii y)| \lesssim y^{b-1} , \qquad 
\textnormal{for all } y > 0 \textnormal{ and } t \in [0,\bigT] .
\end{align}
When $\theta$ has finite energy, 
the derivative estimate~\eqref{eq: uniform bound for f'} holds (in a very strong form), 
and hence, the Loewner trace~\eqref{eq: limit curve} exists and is continuous in time.

A chordal version of the next result appeared in~\cite[Theorem~2(i)]{Friz-Shekhar:Finite_energy_drivers} without any weight function. 
\Cref{thm: FS_estimate} includes a general weight function and thanks to its radial setup should be useful in applications to various planar growth processes.

\begin{thm} \label{thm: FS_estimate}
Fix $n=1$ and $\bigT \in (0,\infty)$. 
Let $\lambda \colon [0,\bigT] \to (0,\infty)$ be a weight function that is uniformly bounded away from zero, i.e., 
\begin{align*}
\supnorm{\tfrac{1}{\lambda}} 
:= \sup_{t \in [0,\bigT]} \tfrac{1}{| \lambda_t |}
< \infty . 
\end{align*}
Let $\theta \in \HoneSpace{\bigT]}{0}$ 
\textnormal{(}i.e., absolutely continuous such that $\theta_0 = 0$ and $\nDenergy_{\bigT}(\theta) < \infty$\textnormal{)}. 
\begin{enumerate}
\item \label{item: FS_estimate_derivative}
Then, we have
\begin{align} \label{eq: strong der estimate}
\big| \tildef_t'(\theta_t + \ii y) \big| 
\leq \; & 
\exp \bigg( \frac{1}{2} \, \supnorm{\tfrac{1}{\lambda}}  \, \nDenergy_{\bigT}(\theta) \bigg) ,
\qquad \textnormal{for all } y > 0 ,
\end{align}
where $\tildef_t := \tildeh_t^{-1}$ is the inverse of the Loewner map $\tildeh_t$ satisfying~\eqref{eqn:multiradial_Loewner_general_cylinder} with $n=1$.

\item \label{item: FS_estimate_simple}
Moreover, $\theta$ generates a simple radial chord $\gamma_{[0,\bigT]}$ in $\bD$ with the $\lambda$-parameterization.
\end{enumerate}
\end{thm}

The proof of Item~\ref{item: FS_estimate_derivative} 
uses a computation similar to that in~\cite[Proof of Theorem~4]{Friz-Shekhar:Finite_energy_drivers}.
(Though such computations have been used in earlier works, including~\cite{Lind:Sharp_condition_for_Loewner_equation_to_generate_slits, Marshall-Rohde:The_loewner_differential_equation_and_slit_mappings, 
Lawler:Multifractal_analysis_of_the_reverse_flow_for_Schramm-Loewner_evolution, 
LMR:Collisions_and_spirals_of_Loewner_traces,
Lawler-Viklund:Optimal_Holder_exponent_for_the_SLE_path}.)
The proof of Item~\ref{item: FS_estimate_simple} relies on the bound~\eqref{eq: uniform bound for f'} implied by Item~\ref{item: FS_estimate_derivative} together with 
an argument that the resulting curve is indeed \emph{simple}, which differs from prior arguments used in the chordal case 
(that in the literature rely on the specific form of the chordal Loewner equation, or scale-invariance which is absent in the radial case).
Alternatively, one could estimate the quasiconformal distortion to show that radial finite-energy hulls are quasislits as in~\cite{Marshall-Rohde:The_loewner_differential_equation_and_slit_mappings,
LMR:Collisions_and_spirals_of_Loewner_traces, Abuzaid-Peltola:In_prep}.

\begin{proof}
Fix $t \ge 0$ and write $\integrand_s := \theta_t-\theta_{t-s}$. 
Then, the (mirror) backward Loewner flow
\begin{align*}
\tildep_s(z) := \tildeh_{t-s}(\tildef_t(z + \theta_t)) - \theta_t , \qquad 0 \le s \le t ,
\end{align*}
satisfies $\tildep_t(z) = \tildef_t(z+\theta_t) - \theta_t$ and the backward Loewner equation
\begin{align*} 
\partial_s \tildep_s(z) =\;& - \ell_s \cot \bigg(\frac{\tildep_s(z) + \integrand_s}{2}\bigg) , \qquad \tildep_0(z) = z , 
\qquad \ell_s := \lambda_{t-s} , \qquad 0 \le s \le t .
\end{align*}
Writing
\begin{align*}
\tildep_s(z) + \integrand_s = X_s + \ii Y_s 
\qquad \textnormal{and} \qquad 
N_s := \cosh (Y_s) - \cos (X_s) ,
\end{align*}
we find for the inverse Loewner map the equation
\begin{align*}
\partial_s \log|\tildef_s'(z + \theta_t)| 
= \partial_s\log|\tildep'_s(z)|  
= \ell_s\frac{(1 - \cos (X_s) \cosh (Y_s))}{N_s^2} , \qquad 0 \le s \le t .
\end{align*}
Writing also $G_s := \integrand_s - X_s$, we obtain
\begin{align*}
\partial_s X_s = \partial_s \integrand_s - \partial_s G_s , \qquad 
\partial_s Y_s = \ell_s\frac{\sinh (Y_s)}{N_s} , \qquad 
\partial_s G_s = \ell_s\frac{\sin (X_s)}{N_s} .
\end{align*}
Now, a straightforward computation shows that 
\begin{align*}
\partial_s \log|\tildef_s'(z + \theta_t)|
=\;& -\ell_s\frac{\sinh^2 (Y_s)}{N_s^2} 
+ \ell_s\frac{\cosh (Y_s)}{N_s}
= -\frac{\sinh (Y_s) }{N_s} (\partial_s Y_s) + \frac{\partial_s \sinh (Y_s)}{\sinh (Y_s)} .
\end{align*}
To write this in a more useful form, let us compute
\begin{align*}
\frac{\partial_s N_s}{N_s} =\;&  \frac{\sinh (Y_s)}{N_s} (\partial_s Y_s) + \frac{\sin (X_s) }{N_s} (\partial_s X_s)  
\\
=\;& \frac{\sinh (Y_s)}{N_s} (\partial_s Y_s) 
+ \frac{1}{\ell_s} (\partial_s G_s)( \partial_s \integrand_s - \partial_s G_s) .
\end{align*}
Putting the above computations together, we see that
\begin{align}\label{eqn:log-derivative}
\log|\tildef_s'(z + \theta_t)| = \log \bigg( \frac{\sinh (Y_t)}{\sinh (Y_0)} \bigg) 
- \log \bigg(\frac{N_t}{N_0}\bigg) 
+ \int_0^t \Big( (\partial_s G_s) (\partial_s \integrand_s) - (\partial_s G_s)^2 \Big) \frac{\ud s}{\ell_s} .
\end{align}
To evaluate~\eqref{eq: strong der estimate}, 
take $z = \ii y$, with $y>0$, so that $X_0 = 0$ and $Y_0 = y$. Then, since $Y_s$ and $\partial_s Y_s$ are positive, we see that
\begin{align*}
\log \bigg( \frac{\sinh (Y_t)}{\sinh (Y_0)} \bigg) - \log \bigg(\frac{N_t}{N_0}\bigg) 
\leq \;& \log \bigg( \frac{\sinh (Y_t)}{\sinh (Y_0)} \bigg) - \log \bigg(\frac{\cosh (Y_t) - 1}{\cosh (Y_0) - 1} \bigg) \\
=\;&\int_0^t \frac{\ud}{\ud t} \log\bigg(\frac{\sinh(Y_t)}{\cosh(Y_t)-1}\bigg) \ud t\\
=\;&\int_0^t \bigg(\frac{\cosh (Y_s)}{\sinh (Y_s)} - \frac{\sinh (Y_s)}{\cosh (Y_s) - 1}\bigg) (\partial_s Y_s) \ud s\\
=\;& -\int_0^t \frac{\partial_s Y_s}{\sinh (Y_s)}\ud s \le 0 .
\end{align*}
Finally, noting that $\frac{1}{4}(\partial_s \integrand_s)^2 \ge (\partial_s G_s) (\partial_s \integrand_s) - (\partial_s G_s)^2$, we obtain from~\eqref{eqn:log-derivative} the sought estimate~\eqref{eq: strong der estimate}:
\begin{align*}
\log|\tildef_t'(iy + \theta_t)| 
\; \le \; \frac{1}{4}\int_0^t (\partial_s \integrand_s)^2 \frac{\ud s}{\ell_s} 
\; = \; \frac{1}{4}\int_0^t (\partial_s \theta_s)^2 \frac{\ud s}{\lambda_s} 
\; \le \; \frac{1}{2} \, \supnorm{\tfrac{1}{\lambda}}  \, \nDenergy_{\bigT}(\theta) .
\end{align*}
This proves Item~\ref{item: FS_estimate_derivative}. 
To prove Item~\ref{item: FS_estimate_simple}, note first that 
the estimate~\eqref{eq: strong der estimate} 
already implies that, for every $t \in [0,\bigT]$ and $0 < y < y' \le y_0$, we have
\begin{align*}
\big| \tildef_t(\theta_t + \ii y) -  \tildef_t(\theta_t + \ii y') \big|
\le  \; & \int_y^{y'} |\tildef_t'(\theta_t + \ii u)| \ud u \\
\le \; & y_0 \exp \bigg( \frac{1}{2} \, \supnorm{\tfrac{1}{\lambda}}  \, \nDenergy_{\bigT}(\theta) \bigg) 
\qquad \xrightarrow{y_0 \to 0} \qquad 0.
\end{align*}
This shows that the radial limit~\eqref{eq: limit curve} exists uniformly in time and in particular is continuous in time. 
By arguments similar to~\cite[Theorem~4.1]{Rohde-Schramm:Basic_properties_of_SLE}, 
this then implies that $\theta$ generates a radial chord $\gamma_{[0,\bigT]}$ in $\bD$. 
It remains to show that $\gamma$ is simple. 
Observe that if $\gamma$ is not simple,
then there exists a time $\tau \in [0,\bigT]$ such that one of the following holds:
\begin{enumerate}[{(i)}]
\item $\gamma$ intersects the boundary at some point $\gamma_\tau = x \in \partial \bD \smallsetminus \{1\}$ at time $\tau = \swallowtime{x}$; or

\item at time $\tau$, the curve $\gamma$ intersects its own past, so $\gamma_\tau = \gamma_{\tau'}$ for some $0\leq \tau' <\tau$.
\end{enumerate} 
If scenario (ii) occurs, then  for any intermediate time $s \in (\tau', \tau)$, the part $t \mapsto \tildeg_s(\gamma_{s+t}) =: \tilde\gamma_t$ of the curve after time $s$ hits $\partial \bD \smallsetminus \{e^{\ii\theta_s}\}$ at time $t = \tau-s$. 
By additivity of the Dirichlet energy~\eqref{eqn:Dirichlet_energy_def}, 
the energy of the driving function $\tilde\theta$ of $\tilde\gamma$ 
satisfies $\nDenergy_{\bigT-s}(\tilde\theta) \leq \nDenergy_{\bigT}(\theta)$, so scenario (ii) reduces to scenario (i). 
It thus remains to show that scenario (i) cannot occur. 
Thanks to \Cref{lem:time change}, by making a time change we may assume without loss of generality that $\lambda_t \equiv 1$. 
Suppose, towards a contradiction, that scenario (i) occurs for some $\gamma_\tau = x \in \partial \bD \smallsetminus \{1\}$.
Consider the time evolution~\eqref{eqn:multiradial_Loewner_general_cylinder} (with $n=1$) of 
$\xi_t := \tildeh_t(-\ii \log x) \in (0,2\pi)$:
\begin{align*}
\frac{\ud}{\ud t} \xi_t 
= \cot \bigg( \frac{\xi_t - \theta_t}{2} \bigg)
= \cot \bigg( \frac{\omega_t}{2} \bigg) , 
\qquad \textnormal{where} \qquad 
\omega_t := \xi_t - \theta_t 
\end{align*}
satisfies $\omega_0 = \xi_0 \in (0,2\pi)$. 
At the hitting time $\swallowtime{x}$ to $x$, 
we have $\omega_{\swallowtime{x}} \in \{0,2\pi\}$, and
\begin{align*}
+\infty \; > \;
2 \nDenergy_\bigT(\theta)
\; \geq \;
2 \nDenergy_{\swallowtime{x}}(\theta)
\; \geq \; \; & 2 \limsup_{t \to \swallowtime{x}-} \nDenergy_t(\theta)  \\
= \; & \limsup_{t \to \swallowtime{x}-}
\int_0^t \Big( \tfrac{\ud}{\ud s} \omega_s - \tfrac{\ud}{\ud s} \xi_s \Big)^2 \ud s  \\
= \; & \limsup_{t \to \swallowtime{x}-}
\int_0^{t} \bigg( 
\big(\tfrac{\ud}{\ud s} \omega_s\big)^2 
- 2 \big( \tfrac{\ud}{\ud s} \, \omega_s \big) \cot\big(\tfrac{\omega_s}{2} \big) 
+ \cot^2\big(\tfrac{\omega_s}{2} \bigg)  \big) \ud s  \\
\geq \; & - 2 \liminf_{t \to \swallowtime{x}-}
\int_0^{t} \big( \tfrac{\ud}{\ud s} \omega_s \big) \,  \cot\big(\tfrac{\omega_s}{2} \big) \ud s  \\
= \; & - 4 \liminf_{t \to \swallowtime{x}-}
\int_0^{t} \frac{\ud}{\ud s} \bigg( \log \sin\big(\tfrac{\omega_s}{2} \big) \bigg) \ud s  \\
= \; & - 4 \,  \liminf_{t \to \swallowtime{x}-}
\log \Bigg( \frac{\sin \big( \tfrac{\omega_t}{2} \big)}{\sin \big( \tfrac{\omega_0}{2} \big)} \Bigg)
\; = \; +\infty .
\end{align*}
This contradiction shows that scenario (i) cannot occur, and finishes the proof.
\end{proof}

\begin{remark} 
As a consequence of the proof of Theorem~\ref{thm: FS_estimate}, we get the following form for the derivative of the inverse Loewner chain (compare with~\cite[Proposition~1]{Friz-Shekhar:Finite_energy_drivers}):  
\begin{align*}
\log|\tildef_s'(z + \theta_t)| 
= \log \bigg( \frac{\sinh (Y_t)}{\sinh (Y_0)} \bigg) - \log \bigg(\frac{N_t}{N_0} \bigg)
+ \int_0^t \Big( (\partial_s G_s) (\partial_s \integrand_s) 
- (\partial_s G_s)^2 \Big) \frac{1}{\ell_s} \ud s .
\end{align*}
\end{remark}

By a closer investigation of the above computation, it should also be possible to extend other results in~\cite{Friz-Shekhar:Finite_energy_drivers} 
(for example,~\cite[Theorem~4]{Friz-Shekhar:Finite_energy_drivers} in the context of 
It\^o-F\"ollmer type integrals). 
Such generalizations would be, however, beyond the applications that we have in mind in the present work, 
so we shall not attempt to do this.

\subsection{Finite-energy hulls are simple radial multichords}
\label{subsec:finite-energy-gives-simple-curves}

The purpose of this section is to prove~\Cref{thm:finite-energy-gives-simple-curves}. 
The proof comprises a few steps. 
We first show that for each $t \in [0,\bigT]$, the hull $K_t$ consists of $n$ disjoint sets which only touch the boundary $\partial \bD$ at the starting points $z^1_0, \ldots, z^n_0$ 
(\Cref{prop:hulls_avoid_boundary_and_each_other}). 
We then derive an analogue of the conformal restriction property (\Cref{prop:removal-of-hulls-loewner}, cf.~\cite{LSW:Conformal_restriction_the_chordal_case}), 
which enables us to pass from the case of one radial curve to the case of several curves. 
We combine these results with \Cref{thm: FS_estimate} to finish 
the proof of \Cref{thm:finite-energy-gives-simple-curves} in the end.

\begin{prop} \label{prop:hulls_avoid_boundary_and_each_other}
Consider a multiradial Loewner chain with the 
$1$-common parameterization for which the uniformizing conformal maps $g_t \colon \bD \smallsetminus K_t \to \bD$ satisfy  Equation~\eqref{eqn:multiradial_Loewner_1common}
with driving functions $z^j_t = e^{\ii \theta^j_t}$ for $1 \leq j \leq n$, 
where $\btheta = (\theta^1, \ldots, \theta^n) 
\in C_{\btheta_0}([0,\bigT], \chamber)$.
If the multiradial Dirichlet energy of $\btheta$ is finite, i.e., 
$\nBessel_\bigT(\btheta) < \infty$, then we have 
\begin{align*}
K_t = \bigsqcup_{j=1}^n K^j_t , \qquad \textnormal{for each $t \in [0,\bigT]$,}
\end{align*} 
where $K^j_t$ are pairwise disjoint connected hulls such that $K^j_t \cap \partial \bD = \{e^{\ii \theta^j_0}\}$, for all $j$. 
\end{prop}

\begin{proof}
We will first prove that $K_t \cap \partial \bD = \{e^{\ii \theta^1_0}, \ldots, e^{\ii \theta^n_0}\}$ for all $t \in [0,\bigT]$. 
As the first step, we show that 
none of the boundary points $x \in \partial \bD \smallsetminus \{e^{\ii \theta^1_0}, \ldots, e^{\ii \theta^n_0}\}$ can be swallowed when the energy is finite. Consider the swallowing times
\begin{align*}
\swallowtime{x} := \min_{1 \le j \le n} \swallowtime{x}^j
 \qquad  \textnormal{where}  \qquad  
\swallowtime{x}^j := \; & \sup \Big\{t \geq 0 \; \colon \; \inf_{s \in [0,t]} \, \big| g_s(x) - e^{\ii \theta_s^j} \big| > 0 \Big\} .
\end{align*}
Towards a contradiction, suppose that $\swallowtime{x} \leq T$.
On the one hand, 
\Cref{cor:finite-energy-characterization} shows that $T < \detblowuptime$. 
On the other hand, 
if $\swallowtime{x}^j = \swallowtime{x}^i$ for some $i \ne j$, then
\begin{align*}
|e^{\ii\theta_s^j}-e^{\ii\theta_s^i}| 
\; \leq \; |e^{\ii \theta_s^j}-g_s(x)| + | g_s(x) - e^{\ii \theta_s^j} \big| 
\quad
\xrightarrow{s \to  \swallowtime{x}^j-} \quad 0,
\end{align*}
which shows that $\detblowuptime \le \swallowtime{x}^j$. 
Hence, we may without loss of generality assume that $\{j_0\} := \argmin_j \swallowtime{x}^j = \{1\}$, so that
$\swallowtime{x} = \swallowtime{x}^{j_0} = \swallowtime{x}^{1} \leq T$.
Consider the time evolution 
\begin{align*}
g_t(e^{\ii u}) = \exp( \ii h_t(u) ) 
\qquad \textnormal{and} \qquad 
\xi_t := h_t(-\ii \log x) \in (0,2\pi) , \qquad  t < \swallowtime{x} ,
\end{align*}
and denote
\begin{align*}
\omega^j_t := \xi_t - \theta^j_t , \quad 1 \leq j \leq n ,
\quad \textnormal{so that} \quad
\omega^1_{\swallowtime{x}} \in \{0,2\pi\} .
\end{align*}
From~\eqref{eqn:multiradial_Loewner_general_cylinder} (with $\lambda_t \equiv 1$) we see that
\begin{align*}
\frac{\ud}{\ud t} \xi_t = \sum_{j=1}^n \cot \bigg( \frac{\omega^j_t}{2} \bigg) , \qquad t < \swallowtime{x} .
\end{align*}
We will now estimate the multiradial Dirichlet energy of $\btheta$ under the assumption that $\swallowtime{x} = \swallowtime{x}^{1} \leq T$, which will lead to a contradiction with the finiteness of the energy: 
\begin{align} 
\nonumber
+\infty \, > \, 
2 \nBessel_{\bigT}(\btheta) 
\, \geq 
2 \nBessel_{\swallowtime{x}}(\btheta) 
= \; & \int_0^{\swallowtime{x}} \sum_{j=1}^n \big| \tfrac{\ud}{\ud s} \theta^j_s - 2 \phi^j(\btheta_s) \big|^2 \ud s 
\; \geq \; \int_0^{\swallowtime{x}} \big| \tfrac{\ud}{\ud s} \theta^1_s - 2 \phi^1(\btheta_s)\big|^2 \ud s \\
= \; & \int_0^{\swallowtime{x}} |V_s-Z_s|^2 \ud s ,
\label{eqn:inf_energy}
\\
\textnormal{where} \qquad 
V_s := \; & \frac{\ud}{\ud s} \omega^1_s -\cot\bigg(\frac{\omega^1_s}{2} \bigg) , 
\nonumber 
\\
Z_s := \; & \sum_{j=2}^n \bigg( \cot \bigg(\frac{\omega^j_s}{2} \bigg) 
+ 2 \cot \bigg(\frac{\omega^1_s-\omega^j_s}{2} \bigg)\bigg) .
\nonumber 
\end{align}
We will show that the righthand side of~\eqref{eqn:inf_energy} is infinite, which gives a contradiction. 
First, as $\detblowuptime > \swallowtime{x}$, 
there exists a constant $R \in (0,\infty)$ such that $|Z_s| \leq R$ for all $s \leq \swallowtime{x}$, so 
\begin{align*}
\int_0^{\swallowtime{x}} |Z_s|^2 \ud s 
\; \leq \; R \, \swallowtime{x} .
\end{align*}
Second, the same computation as in the end of the proof of \Cref{thm: FS_estimate} shows that
\begin{align*}
\liminf_{t \to \swallowtime{x}-} \int_0^{t} |V_s|^2  \ud s \;
 \ge
 \; & - 4 \,  \liminf_{t \to \swallowtime{x}-}
\log \Bigg( \frac{\sin \big( \frac{\omega_t^1}{2} \big)}{\sin \big( \frac{\omega_0^1}{2} \big)} \Bigg)
\; = \; +\infty ,
\end{align*}
since $\omega^1_{\swallowtime{x}} \in \{0,2\pi\}$.
We conclude that 
\begin{align*}
+\infty \, > \, 
\textnormal{\eqref{eqn:inf_energy}}
\; = \; & \int_0^{\swallowtime{x}} |V_s-Z_s|^2 \ud s  \geq \underbrace{\liminf_{t \to \swallowtime{x}-} \int_0^{t} \tfrac{1}{2} |V_s|^2 \ud s}_{= \; + \infty}
\; - \; \underbrace{\int_0^{\swallowtime{x}} |Z_s|^2 \ud s}_{\leq \; R \, \swallowtime{x} \; \in \; [0,\infty)} = +\infty .
\end{align*}
This gives the sought contradiction --- so in conclusion, for all $x \in \partial \bD \smallsetminus \{e^{\ii \theta^1_0}, \ldots, e^{\ii \theta^n_0}\}$, 
we have $\swallowtime{x} > T$.
We have thus shown that $K_t \cap \partial \bD = \{e^{\ii \theta^1_0}, \ldots, e^{\ii \theta^n_0}\}$ for all $t \in [0,\bigT]$.

To finish, notice that any hull partitions into pairwise disjoint connected hulls, each of which has a connected intersection with the boundary $\partial \bD$. Thus, the hull $K_t$ is the union of $n$ pairwise disjoint connected hulls $K^j_t$, $j \in \{1, \ldots, n\}$, satisfying $K^j_t \cap \partial \bD = \{e^{\ii \theta_0^j}\}$.
\end{proof}

We will now prove an analogue of the conformal restriction property (\Cref{prop:removal-of-hulls-loewner}). 
In the $1$-common parameterization, the uniformizing Loewner maps $g_t \colon \bD \smallsetminus K_t \to \bD$ satisfy~\eqref{eqn:multiradial_Loewner_1common}
with $z^j_t=e^{\ii \theta^j_t}$, and the map $h_t$ related to $g_t$ via $g_t(e^{\ii u}) = \exp( \ii h_t(u) )$ 
satisfies
\begin{align} \label{eq: LE for h}
\partial_t h_t(u) = \sum_{j=1}^n \cot \bigg(\frac{h_t(u)-\theta^j_t}{2}\bigg) .
\end{align}
See Figure~\ref{fig:conf_rest} for an illustration of the setup of~\Cref{prop:removal-of-hulls-loewner}.

\begin{figure}
\centering
\includegraphics[width=0.6\textwidth]{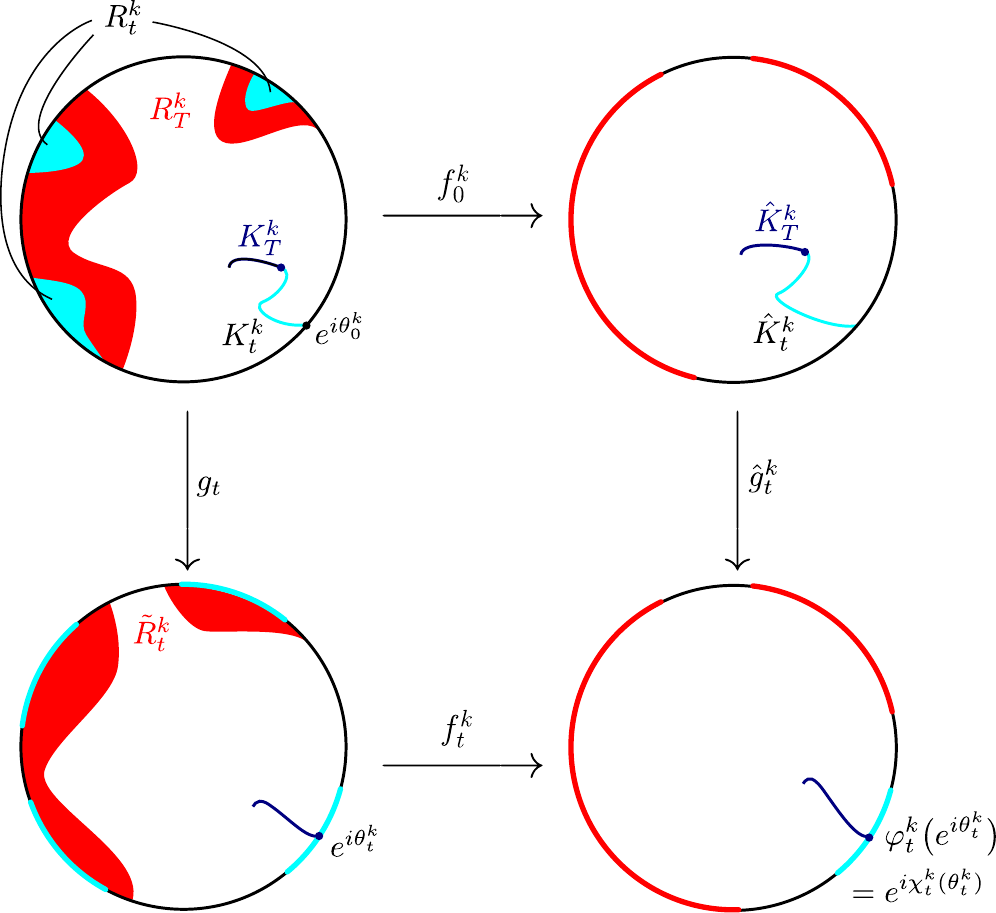}
\caption{\label{fig:conf_rest} 
Illustration of the setup of~\Cref{prop:removal-of-hulls-loewner}. }
\end{figure}

\begin{prop} \label{prop:removal-of-hulls-loewner}
Consider a multiradial Loewner chain with the 
$1$-common parameterization for which the uniformizing conformal maps $g_t \colon \bD \smallsetminus K_t \to \bD$ satisfy  Equation~\eqref{eqn:multiradial_Loewner_1common}
with driving functions $z^j_t = e^{\ii \theta^j_t}$ for $1 \leq j \leq n$, 
where $\btheta = (\theta^1, \ldots, \theta^n) 
\in C_{\btheta_0}([0,\bigT], \chamber)$.
Suppose that the multiradial Dirichlet energy of $\btheta$ is finite, i.e.,  
$\nBessel_\bigT(\btheta) < \infty$.

Fix $k \in \{1, \ldots, n\}$ and consider the partition $K_t = R_t^k \sqcup K^k_t$, where $K^k_t$ is the connected component of $K_t$ containing $e^{\ii \theta^k_0}$ as in~\Cref{prop:hulls_avoid_boundary_and_each_other},
and $R_t^k = K_t \smallsetminus K^k_t$ its complement. 
Define the following quantities, as shown in Figure~\ref{fig:conf_rest}: 
\begin{itemize}
\item 
Write $\tilde{R}^k_t := g_t(R^k_\bigT \smallsetminus R^k_t)$ for $t \in [0,\bigT]$. 

\item 
Let $\Mob^k_t \colon \bD \smallsetminus \tilde{R}^k_t \to \bD$ denote the 
uniformizing map normalized at the origin.

\item 
Write $\hat{K}^k_t = \Mob^k_0(K^k_t)$. 

\item 
Let $\hat{g}^k_t \colon \bD \smallsetminus \hat{K}^k_t \to \bD$ denote the 
uniformizing map normalized at the origin.

\item 
Define $\hat h ^k_t \in [0,2\pi)$ by
$\hat{g}^k_t(e^{\ii u}) = \exp( \ii \hat{h}^k_t(u) )$ for $u\in \bH$ such that $e^{\ii u} \in  \bD \smallsetminus \hat{K}^k_t$.

\item 
Define $\chi^k_t \in [0,2\pi)$ by
$\Mob^k_t(e^{\ii u}) = \exp( \ii \chi^k_t(u) )$  for $u\in \bH$ such that $e^{\ii u} \in  \bD \smallsetminus \tilde{R}^k_t $. 
\end{itemize} 
Then, we have
\begin{align} 
\label{eq: h hat Loewner equation}
\partial_t \hat{h}^k_t(u) 
= \; & \lambda_t^k \, 
\cot \bigg(\frac{\hat{h}^k_t(u) - \hat{\theta}_t^k}{2}\bigg) , 
\qquad t \in [0,\swallowtimehat{u})  , 
\end{align}
where $t \mapsto \lambda_t^k := \big( (\chi^k_t)'(\theta^k_t) \big)^2$ is a continuous weight function, 
$t \mapsto \hat{\theta}_t^k := \chi^k_t(\theta^k_t)$ is the driving function, and 
\begin{align*}
\swallowtimehat{u} 
:= \; & \sup \Big\{t \geq 0 \; \colon \; \inf_{s \in [0,t]} \, \big| \hat{g}^k_s(e^{\ii u}) - e^{\ii \hat{\theta}_s} \big| > 0 \Big\} .
\end{align*}
Furthermore, we have $\nDenergy_{\bigT}(\hat\theta^k) < \infty$. 
\end{prop}

\begin{proof}
See Figure~\ref{fig:conf_rest} for the setup. 
Note that $\hat{g}^k_t = \Mob^k_t \circ g_t \circ (\Mob^k_0)^{-1}$, 
since all uniformizing maps are normalized at the origin. 
Similarly, $\hat{h}^k_t = \chi^k_t \circ h_t \circ (\chi^k_0)^{-1}$.

We see that the hulls $(\hat K^k_{t})_{t \ge 0}$ are locally growing, since the hulls $(K^k_{t})_{t \ge 0}$ are locally growing and contained in the domain of $\Mob^k_0$, which is a homeomorphism that extends continuously to the boundary.
This implies that $\hat g^k_t$ satisfies the weighted single-curve radial Loewner equation for some weight $\lambda_t^k$, with driving function $e^{\ii {\hat \theta^k_t}}=\Mob^k_t(e^{\ii \theta^k_t})$.  Consequently,
$\hat h^k_t$ satisfies an equation of the form~\eqref{eq: h hat Loewner equation} with $\hat\theta^k_t = \chi^k_t(\theta_t)$, again with some weight $\lambda_t^k$.

In order to find $\lambda^k_t$, we compute the time derivative of $\hat h_t (u)=  (\chi^k_t \circ h_t \circ (\chi^k_0)^{-1}) (u)$ using the chain rule, 
substituting into~\eqref{eq: LE for h}, 
and then setting the result equal to the righthand side of~\eqref{eq: h hat Loewner equation},
which shows that 
\begin{align}\label{eqn:hat-comparison}
\lambda_t^k \, \cot \bigg(\frac{\chi^k_t(v) - \chi^k_t(\theta^k_t)}{2}\bigg)
= \; & (\partial_t \chi^k_t)(v)
+ (\chi^k_t)'(v) \; 
\sum_{j=1}^n \cot \bigg(\frac{v - \theta^j_t}{2}\bigg),
\end{align}
where $v  = (h_t \circ (\chi^k_0)^{-1}) (u)$ for $e^{\ii u} \in \bD\smallsetminus\hat K^k_t$. 
The above equation~\eqref{eqn:hat-comparison} holds whenever $e^{\ii v} \in \bD\smallsetminus\tilde R^k_t$.
We will solve for $\lambda^k_t$ and evaluate the limit as $v\to\startheta:=\theta^k_t$. 
For this purpose, we define the notation $\startheta$ to be clear that this value is fixed, even as we consider $\chi^k_s$ for $s\in [t-\epsilon, t+\epsilon]$. 
However, due to the singularity at $\startheta$, we first take care to check that all relevant maps are jointly continuous in a neighborhood of $(t, \startheta)$.

The conformal mapping $\chi^k_t$ is well-defined and extends continuously to the boundary 
in an $\bH$-neighborhood around $\startheta$, so 
by Schwarz reflection, $\chi^k_t$ extends conformally to a neighborhood of $\startheta$ in the complex plane. 
Moreover, the conformal maps $\chi^k_s$ are continuously differentiable in $s$, and we can find $\epsilon>0$ and $\bH$-neighborhood $\nbhd$ around $ \startheta$ 
such that the map $(v, s) \mapsto \partial_s \chi^k_s(v)$ exists and is jointly continuous on $\overline \nbhd \times [t-\epsilon, t+\epsilon]$.
Again extending by Schwarz reflection, there exists a $\bC$-neighborhood $\bignbhd \ni \startheta$ such that for $s\in [t-\epsilon, t+\epsilon]$, each $\partial_s \chi^k_s$ is conformal on $\bignbhd$, the map $(v,s)\mapsto  \partial_s \chi^k_s(v)$ is jointly continuous on $\overline \bignbhd \times [t-\epsilon, t+\epsilon]$. 
Without loss of generality, we may assume that $\bignbhd$ is simply connected with rectifiable boundary. 
Finally, for each $m \in \bN$, the map\footnote{Here, $\partial_z^m$ denotes the $m$:th complex derivative, while $\chi^k_t$ is the $k$:th component of the vector $(\chi^1_t, \ldots, \chi^n_t)$.}
$(z, t) \mapsto \partial_z^m\partial_t \chi^k_t$ is jointly continuous on $\overline \bignbhd \times [t-\epsilon, t+\epsilon]$, 
which we can see by applying the Cauchy differentiation formula, 
\begin{align*}
\frac{\partial_z^m \chi^k_{t+s}(z)
- \partial_z^m\chi_t^k(z)}{s} =\;& \frac{m!}{2\pi \ii}\int_{\partial U} \frac{\chi^k_{t+s}(w) - \chi_t^k(w)}{s}\frac{\ud w}{(w-z)^m}, \qquad \textnormal{for all } z 
\in \bignbhd,
\end{align*}
and the dominated convergence theorem, which yield 
\begin{align*}
\partial_t \partial_z^m\chi_t^k(z) = \frac{m!}{2\pi \ii}\int_{\partial U}\frac{\partial_t \chi_t^k(w)}{(w-z)^m}\ud w =  \partial_z^m \partial_t \chi^k_t(z).
\end{align*}

Next, we Laurent expand both sides of~\eqref{eqn:hat-comparison} 
around the singularity $\startheta$, to obtain
\begin{align*}
\frac{2\lambda_t^k}{\chi^k_t(v) - \chi^k_t(\theta^k_t)} + O(1) 
= \; & (\partial_t \chi^k_t)(v)
+ (\chi^k_t)'(v) \bigg(\frac{2}{v - \theta^k_t} + O(1) \bigg) , \qquad v \to \startheta .
\end{align*}
We can solve for $\lambda_t^k$ by multiplying both sides by $\tfrac{1}{2}(\chi^k_t(v) - \chi^k_t(\theta^k_t))$ and taking the limit as $v \to \startheta$ (which is justified by the continuity checks above):
\begin{align*}
\lambda_t^k 
= \lim_{v \to \theta^k_t} \bigg( (\chi^k_t)'(v) \; \frac{\chi^k_t(v) - \chi^k_t(\theta^k_t)}{v - \theta^k_t} \bigg)
= \big( (\chi^k_t)'(\theta^k_t) \big)^2 \; \in \; (0,\infty) .
\end{align*}
Let us also note that $(\chi^k_t(\startheta))'\neq 0$, 
since (by Schwarz reflection, as above) $\chi^k_t$ is conformal in a neighborhood of $\startheta$. 
Finally, the joint continuity of $(s,v)\mapsto(\chi^k_s)'(v)$ implies that $\lambda^k_t$ is continuous and therefore locally integrable. 
We thus conclude that $\lambda^k_t=\big( (\chi^k_t)'(\theta^k_t) \big)^2$ is a continuous weight function.
This proves the asserted equality~\eqref{eq: h hat Loewner equation}.

Finally, we check that $\nDenergy_{\bigT}(\hat\theta^k_t) < \infty$: differentiating, we have 
\begin{align*}
\tfrac{\ud}{\ud t}\hat\theta^k_t 
= \; & \tfrac{\ud}{\ud t}\chi^k_t(\theta^k_t) = \partial_t \chi^k_t(\theta^k_t) + \big((\chi^k_t)'(\theta^k_t)\big)\tfrac{\ud}{\ud t}\theta^k_t , 
\end{align*}
By joint continuity, all of 
$\partial_t\chi^k_t(\theta^k_t)$, 
and $(\chi^k_t)'(\theta^k_t)$, 
and $\partial_t(\chi^k_t)'(\theta^k_t)$, 
are uniformly bounded on the compact interval $[0,T]$, 
while by \Cref{cor:finite-energy-characterization}, 
the function $\theta^k_t$ has finite Dirichlet energy, so $t \mapsto \tfrac{\ud}{\ud t}\theta^k_t$ is integrable on $[0,\bigT]$ (as $\theta^k$ is absolutely continuous). 
Thus, we see that 
$\nDenergy_{\bigT}(\hat\theta^k_t) < \infty$, proving the last claim.
\end{proof}

\begin{proof}[Proof of \Cref{thm:finite-energy-gives-simple-curves}]
The case of $n = 1$ is covered by \Cref{thm: FS_estimate}, 
so we will consider the case where $n > 1$. 
By \Cref{prop:hulls_avoid_boundary_and_each_other}, the hull $K_t = \bigsqcup_{j=1}^{n} K^j_t$ is a disjoint union of $n$ connected components $K^j_t$ containing the starting points $\theta^j_0$, with $1 \leq j \leq n$. 
It thus suffices to show that each such connected component is generated by a simple curve.

Fix $k \in \{1, \ldots, n\}$.
By conjugating by a suitable rotation, we may assume without loss of generality that $\theta^k_0 = 0$. 
With notation from~\Cref{prop:removal-of-hulls-loewner}, the map 
$\hat{h}^k_t$ satisfies the 
(single) radial Loewner  equation~\eqref{eqn:multiradial_Loewner_general_cylinder} (with $n=1$) 
parameterized by the continuous weight function 
$\lambda^k$ and with driving function $\hat \theta^k$, 
which has finite Dirichlet energy $\nDenergy_{\bigT}(\hat \theta^k) < \infty$. 
Thus, by \Cref{thm: FS_estimate} we know that $\hat K^k$ is a simple radial chord, so $K^k_\bigT = (\Mob^k_0)^{-1}(\hat{K}^k_\bigT)$ is also a simple 
curve (as a conformal image of such).  
As the choice of the index $k \in \{1, \ldots, n\}$ was arbitrary, 
we conclude that every connected component of $K_\bigT$ is a simple curve.
\end{proof}

\subsection{Proof of the LDP for multiradial $\SLE_{0+}$}
\label{subsec: LDP for SLE}

In this section, we prove the main result, \Cref{thm:radial_LDP_finite_time}.
Let us begin by recalling that the Loewner transform $\cL_t$, defined in~\eqref{eq: Loewner transform}, sends driving functions to the hulls generated by the multiradial Loewner equation~\eqref{eqn:multiradial_Loewner_1common} with $1$-common parameterization.
Hence, it would be natural to apply the contraction principle (\Cref{thm:contraction-principle}) 
to deduce the LDP for multiradial $\SLE_\kappa$ from the LDP for Dyson Brownian motion (\Cref{thm:Bessel_LDP}).
However, as the Loewner transform is not continuous, we cannot do this directly. 
Instead, we first restrict $\cL_t$ into a subset with full measure where it is continuous, 
and use \Cref{prop:restricted-LDP} below, which will allow us to derive the large deviations result for multiradial $\SLE_{0+}$.

\begin{thmA}
[Contraction principle,~{\cite[Theorem~4.2.1]{Dembo-Zeitouni:Large_deviations_techniques_and_applications}}]
\label{thm:contraction-principle}
Let $\spaceX$ and $\spaceY$ be Hausdorff topological spaces, and let $f \colon \spaceX \to \spaceY$ be a continuous map. 
Suppose that the family $(\PrDyson)_{\kappa>0}$ of probability measures satisfies an LDP in $\spaceX$
with good rate function $\lenergy \colon \spaceX \to [0,+\infty]$, that is,
for any closed subset $\closed_0$ and open subset $\open_0$ of $\spaceX$, we have
\begin{align*} 
\limsup_{\kappa \to 0+} \kappa \log \PrDyson \big[ \closed_0 \big] 
\, \leq \, - \inf_{x \in \closed_0} \lenergy(x) 
\qquad \textnormal{and} \qquad
\liminf_{\kappa \to 0+} \kappa \log \PrDyson \big[ \open_0 \big] 
\, \geq \, - \inf_{x \in \open_0} \lenergy(x) ,
\end{align*}
and the level set $\lenergy^{-1}[0,c]$ is a compact subset of $\spaceX$, for all $c \in [0, \infty)$.
Define
\begin{align*}
\nBessel(y) := \inf_{x \in f^{-1}\{y\}} \lenergy(x) , \qquad y \in \spaceY .
\end{align*}
Then, the family $
(\PrPush^\kappa)_{\kappa>0}
:= (\PrDyson \circ f^{-1})_{\kappa>0}$ of pushforward probability measures satisfies an LDP in $\spaceY$ with good rate function $\nBessel$\textnormal{:}
for any closed subset $\closed$ and open subset $\open$ of $\spaceY$, 
\begin{align*}
\limsup_{\kappa \to 0+} \kappa \log \PrPush^\kappa [ \closed ] 
\, \leq \, - \inf_{y \in \closed} \nBessel(y) 
\qquad \textnormal{and} \qquad
\liminf_{\kappa \to 0+} \kappa \log \PrPush^\kappa [ \open ] 
\, \geq \, - \inf_{y \in \open} \nBessel(y) ,
\end{align*}
and the level set $\nBessel^{-1}[0,c]$ is a compact subset of $\spaceY$, for all $c \in [0, \infty)$.
\end{thmA}

\begin{prop}[{Restricted LDP; see also~\cite[Lemma~4.1.5(b)]{Dembo-Zeitouni:Large_deviations_techniques_and_applications}}]
\label{prop:restricted-LDP} 
Let $\spaceX$ be a Hausdorff topological space, $(\prob^\kappa)_{\kappa >  0}$ a family of probability measures on $\spaceX$, and $\lenergy \colon \spaceX \to [0,+\infty]$ a rate function. 
Suppose $A \subset \spaceX$ is a measurable subset such that $I^{-1}[0,\infty) \subset A$ and $\prob^\kappa[A] = 1$ for every $\kappa > 0$. 
Then, the family $(\prob|_A^\kappa)_{\kappa > 0}$ of restricted measures satisfies an LDP in $A$ with rate function $\lenergy|_A$ if and only if the family $(\prob^\kappa)_{\kappa > 0}$ satisfies an LDP in $\spaceX$ with rate function $\lenergy$.
Specifically, for every closed $\closed \subset \spaceX$ and open $\open \subset \spaceX$, 
the following equivalences hold:
\begin{align*}
\limsup_{\kappa \to 0}\kappa\log\prob^\kappa[\closed]  \, \geq \,  -\inf_{x \in \closed}\lenergy(x) 
\qquad \Longleftrightarrow \;& \qquad 
\limsup_{\kappa \to 0}\kappa\log\prob|_A^\kappa[\closed \cap A]  \, \geq \,  -\inf_{x \in \closed \cap A}\lenergy|_A(x) , \\
\liminf_{\kappa \to 0}\kappa\log\prob^\kappa[\open]  \, \leq \,  -\inf_{x \in \open}\lenergy(x) \qquad \Longleftrightarrow \;& \qquad 
\liminf_{\kappa \to 0}\kappa\log\prob|_A^\kappa[\open \cap A]  \, \le \,  -\inf_{x \in \open \cap A}\lenergy|_A(x) .
\end{align*}
Furthermore, $I$ is a good rate function if and only if $I|_A$ is a good rate function.
\end{prop}

\begin{proof}
Since $\prob^\kappa[A] = 1$ and $\lenergy^{-1}[0,\infty) \subset A$, for any measurable set $B \subset \spaceX$, we have
\begin{align*}
\prob^\kappa[B] = \prob^\kappa[B \cap A] = \prob|_A^\kappa[B\cap A] 
\qquad \textnormal{and} \qquad 
\inf_{x \in B}\lenergy(x) = \inf_{x \in B\cap A}\lenergy(x) = \inf_{x \in B\cap A}\lenergy|_A(x).
\end{align*}
The equivalences follow. 
The last claim follows from the assumption $I^{-1}[0,\infty) \subset A$.
\end{proof}

\begin{remark}\label{rem:measurability}
In \Cref{prop:restricted-LDP} it is enough to assume that $A \subset X$ is some (not necessarily measurable) subset such that $I^{-1}[0,\infty) \subset A$ and for every $\kappa > 0$ there is some measurable set $A^\kappa \subset A$ such that $\prob^\kappa[A^\kappa] = 1$. We equip $A$ with the subspace topology induced from $X$, and define the restricted measures by $\prob|^\kappa_A[E] := \prob^\kappa[E\cap A^\kappa]$.
\end{remark}

Recall that a \emph{hull} is a compact set $K \subset \overline{\bD}$ such that $\bD\smallsetminus K$ is simply connected, $0 \in \bD\smallsetminus K$, and the closure $\overline{K \cap \bD} = K$ in $\bC$. 
For each hull $K$, we denote by $g_K \colon \bD\smallsetminus K \to \bD$ the uniformizing map normalized at the origin, i.e., satisfying $g_K(0) = 0$ and $g_K'(0)>0$. 
We call $\log g_K'(0)$ the \emph{capacity} of $K$, so that the complement of $K$ has conformal radius $1/g_K'(0) = e^{-\log g_K'(0)}$.
For each fixed $\bigT \in (0,\infty)$, we denote 
\begin{align}\label{eqn:def_of_K_T}
\hullspace{\bigT} = \big\{ \textnormal{hulls } K \subset \overline{\bD} \textnormal{ of capacity }n \bigT \big\} 
\qquad \textnormal{and} \qquad 
\hullspace{} := \bigcup_{T \ge 0} \hullspace{\bigT} .
\end{align}
We endow the space $\hullspace{}$ of hulls with the coarsest (Carath\'eodory) topology 
for which a sequence $(K_m)_{m \in \bN}$ in $\hullspace{}$ converges to $K \in \hullspace{} $ if and only if the associated functions $g_{K_m}^{-1}$ converge to $g_K^{-1}$ uniformly on compact subsets of $\bD$. 
By~\cite[Theorem~3.1]{Duren:Univalent_functions}, 
this is equivalent to the \emph{Carath\'eodory kernel convergence} of the complementary domains $D_m := \bD\smallsetminus K_m$ to $D := \bD\smallsetminus K$ with respect to the origin:  
for any subsequence $(D_{m_j})_{j \in \bN}$ we have $D = \bigcup_{j\ge 1} \big(\bigcap_{i \ge j} D_{m_j}\big)_0$, denoting by $V_0$ the connected component of a set $V \subset \overline{\bD}$ containing the origin. 
Because we require that $\overline{K \cap \bD} = K$ for any hull $K$,
we see that for two hulls $K \neq \tilde{K}$, we have $\bD \smallsetminus K \neq \bD \smallsetminus \tilde{K}$, which shows that the Carath\'eodory topology on the set $\hullspace{}$ has the Hausdorff (T2) property. (This is required in the contraction principle.)

Although $\hullspace{} \subset \compsets$ is contained in the set of compact subsets of $\overline{\bD}$, the Carath\'eodory and Hausdorff~\eqref{eq: Hausdorff metric} 
topologies on $\hullspace{}$ are not comparable. 
However, we can characterize their difference in the following useful manner (via a radial analogue of~\cite[Lemma~2.3]{Peltola-Wang:LDP_of_multichordal_SLE_real_rational_functions_and_determinants_of_Laplacians}).

\begin{lem}\label{lem:caratheodory-vs-hausdorff-convergence}
Suppose that a sequence $(K_m)_{m \in \bN}$ in $\hullspace{}$ converges to $K \in \mathcal{K}$ in the Carath\'eodory sense and to $\tilde K \in \compsets$ in the Hausdorff metric. 
Then $\bD\smallsetminus K = (\bD\smallsetminus \tilde K)_0$. 
In particular, we have $\bD\cap K = \bD\cap \tilde K$ if and only if $\bD\smallsetminus \tilde K$ is connected. 
\end{lem}

\begin{proof}
This follows by the same proof as~\cite[Lemma~2.3]{Peltola-Wang:LDP_of_multichordal_SLE_real_rational_functions_and_determinants_of_Laplacians}. 
\end{proof}

The Loewner transform $\cL_\bigT \colon \BesselSpace{\bigT]}{} \to \compsets$ sends driving functions to hulls, 
\begin{align*}
\cL_\bigT(\btheta) 
:= \; & \{z \in \overline{\bD} \; | \; \swallowtime{z} \le \bigT\} 
\; \subset \; \hullspace{\bigT} 
\; \subset \; \compsets .
\end{align*}
It is well-known that $\cL_\bigT$ is
continuous in the Carath\'eodory sense 
(see~\cite[Proposition~6.1]{Miller-Sheffield:QLE} for a proof for general Loewner chains).
While $\cL_\bigT$ is not continuous in the Hausdorff metric, its discontinuities occur outside of the set of simple curves (cf.~\cite[Lemma~2.4]{Peltola-Wang:LDP_of_multichordal_SLE_real_rational_functions_and_determinants_of_Laplacians})
--- in particular, outside of the support of the $\SLE_\kappa$ measures when $\kappa \le 4$. 
 
\begin{proof}[Proof of Theorem~\ref{thm:radial_LDP_finite_time}]
We can write the Loewner transform as a composition $\cL_\bigT = \iota\circ\hat\cL_\bigT$, where $\hat\cL_\bigT \colon \BesselSpace{\bigT]}{} \to \hullspace{\bigT}$ is the Loewner transform to the set~\eqref{eqn:def_of_K_T} of hulls of capacity $n \bigT$, 
and $\iota \colon \hullspace{} \hookrightarrow \compsets$ is the inclusion of the hulls to the compact subsets of $\overline{\bD}$.
Now, the map $\hat\cL_\bigT$ is continuous in the Carath\'eodory sense (e.g., by~\cite[Proposition~6.1]{Miller-Sheffield:QLE}), 
so \Cref{thm:Bessel_LDP} and the contraction principle (\Cref{thm:contraction-principle}) 
together imply that the initial segments $\bgamma^\kappa_{[0,\bigT]}$ of multiradial $\SLE_\kappa$ curves with laws $(\probSLE)_{\kappa > 0}$ satisfy an LDP in $\hullspace{\bigT}$ (in the Carath\'eodory sense)
with good rate function $\smash{\hat \lenergy_\bigT} \colon \hullspace{\bigT}\to[0,+\infty]$ defined similarly to~\eqref{eqn:nradial-energy},
\begin{align*}
\hat \lenergy_\bigT(K) := \inf_{\btheta\in\hat\cL_\bigT^{-1}(K)} \nBessel_\bigT(\btheta).
\end{align*}
Next, denote by $\rmultichordspace \subset \hullspace{\bigT}$ the set of simple radial multichords with total capacity $n \bigT$ which are generated by a driving function in the $1$-common parameterization (as in \Cref{def: generate multichord}). 
For $\kappa \leq 4$, 
we have $\probSLE[\rmultichordspace] = 1$, while by \Cref{thm:finite-energy-gives-simple-curves}, the set $\rmultichordspace$ contains all finite-energy hulls. 
Thus, we deduce from \Cref{prop:restricted-LDP}\footnote{Note that $\rmultichordspace$ contains a $\probSLE$-measurable subset of full measure, because the solution of the Loewner equation is measurable with respect to the driving process (whose law is $\probSLE$). See also \Cref{rem:measurability}.} 
that the family $(\probSLE)_{\kappa > 0}$ satisfies an LDP in $\rmultichordspace$ in the Carath\'eodory sense
and with good rate function $\hat \lenergy_\bigT|_{\rmultichordspace}$.

Now, we claim that the restricted map $\iota|_{\rmultichordspace} \colon \rmultichordspace \hookrightarrow \compsets$ is continuous, when the former space carries the Carath\'eodory topology and the latter the Hausdorff metric.  Indeed, suppose that 
a sequence $\big(\boldeta_m \big)_{m\in \bN}$ of simple radial multichords in $\rmultichordspace$ converges to $\boldeta \in \rmultichordspace$ in the Carath\'eodory sense. 
By compactness of $\compsets$, passing to a subsequence, 
$\boldeta_m$ also converge in the Hausdorff metric to some $\tilde\boldeta \in \compsets$.
Then \Cref{lem:caratheodory-vs-hausdorff-convergence} implies that $\bD\cap \boldeta = \bD\cap \tilde \boldeta$ (since otherwise, $\boldeta$ would have non-empty interior).
Furthermore, since $\boldeta$ is a hull and $\tilde \boldeta$ is compact, 
this implies that $\boldeta = \overline{\boldeta \cap \bD} = \overline{\bD\cap \tilde \boldeta} \,\subseteq \, \tilde \boldeta$.
Now, if $x \in (\partial\bD) \smallsetminus \boldeta$, 
then the sets $\boldeta_m$ avoid $x$ for large enough $m$, so $x \notin \tilde \boldeta$.
It follows that 
$\boldeta$ and $\tilde \boldeta$ agree on the interior of the disk and also on the boundary, so 
$\boldeta = \tilde \boldeta$. 
This shows that $\iota|_{\rmultichordspace} \colon \rmultichordspace \hookrightarrow \compsets$ is continuous. 
Therefore, we can apply
the contraction principle (\Cref{thm:contraction-principle}) 
again to deduce that the pushforward measures $(\probSLE \circ (\iota|_{\rmultichordspace}))_{\kappa > 0}$ satisfy an LDP in $\iota(\rmultichordspace) \subset \compsets$ 
in the Hausdorff metric with good rate function $(\hat \lenergy_\bigT\circ\iota^{-1})|_{\iota(\rmultichordspace)} = (\lenergy_\bigT)|_{\iota(\rmultichordspace)}$. 
From this, we conclude again using \Cref{prop:restricted-LDP} 
that the initial segments $\bgamma^\kappa_{[0,\bigT]}$ of multiradial $\SLE_\kappa$ curves with laws $(\probSLE)_{\kappa > 0}$ 
indeed satisfy the LDP~(\ref{eq: limsup claim curves},~\ref{eq: liminf claim curves}) in $\compsets$ with good rate function $\lenergy_\bigT$.
\end{proof}

\section{Large-time behavior of finite-energy systems}\label{sec:zero-energy}

For any function with finite Dyson-Dirichlet energy, we intuitively expect that as $t\to \infty$, 
the interacting particle system that describes these functions approaches an equilibrium configuration.
This section makes these ideas precise.

The key result, \Cref{thm:U_equally_spaced}, follows by collecting the results of this section:
\begin{proof}[Proof of Theorem~\ref{thm:U_equally_spaced}]
This is the content of Propositions~\ref{prop:limit_equally_spaced_zero}~\&~\ref{prop:finite_energy} below.
\end{proof}

\subsection{Zero-energy flow: existence and uniqueness}

Clearly, the $n$-dimensional Dirichlet energy $\nDenergy_{\bigT}$ appearing in~\eqref{eqn:nDirichlet_energy_def} is non-negative and attains the minimum 
$\nDenergy_{\bigT}(\btheta_0) = 0$ at the constant function $\btheta \equiv \btheta_0$.
Although the sign of the functional $\Phi^0_\bigT$ is not clear from its formula~\eqref{eqn:Phi_kappa}, 
Proposition~\ref{prop:ODE_existence_uniqueness} below shows that the Dyson-Dirichlet energy $\nBessel_\bigT$ also attains the minimum zero. 
From \Cref{def:multiradial_Dirichlet_energy}, we see that its minimizers satisfy an ODE system, which in fact has a unique global solution. 
(In fact, the autonomous ODE system~\eqref{eqn:ODE_system} is a gradient flow with Lyapunov function $\DPotential$.)

\begin{prop} \label{prop:ODE_existence_uniqueness}
Assume that $\DPotential$ is a Dyson-type potential as in \Cref{def:Dyson_type_potential}. 
The system of differential equations on $\BesselSpace{\infty)}{\btheta_0}$ given by
\begin{align} \label{eqn:ODE_system} 
\tfrac{\ud}{\ud t} \theta^j_t =  \newphi^j (\btheta_t) , 
\qquad \textnormal{for all } t \geq 0 \textnormal{ and } j \in \{1,\ldots,n\} ,
\end{align}
has a unique solution for each initial configuration $\btheta_0 \in \chamber$. 
\end{prop}

The derivatives $\partial_k\newphi^j = \partial_k\partial_j\log \partfn$ in~\eqref{eqn:def_of_drift} are continuous, so 
for each $\epsilon > 0$, the function $\btheta \mapsto \newphi^j(\btheta)$ is Lipschitz on the compact set 
$\chamberCl{\epsilon} = \{\btheta\in \chamber \; | \; \DeltaMin_{\btheta} \ge \epsilon \}$, where 
$\DeltaMin_{\btheta} := \min_j \, \big| \theta^{j+1} - \theta^j \big| \in \big[0,  \tfrac{2\pi}{n} \big]$. 
Consequently, for any initial configuration $\btheta_0 \in \chamber^\epsilon$, a unique solution to the ODE system~\eqref{eqn:ODE_system} exists in $\chamber^\epsilon$ up until the time when the boundary $\partial \chamber^\epsilon$ is hit. 
(Indeed, the Picard-Lindel\"of theorem implies that the unique solution exists on a time interval whose length depends only on the Lipschitz constant, which in turn only depends on $\epsilon$.) 
It thus suffices to show that any solution to~\eqref{eqn:ODE_system} with initial condition in $\chamber^\epsilon$ stays in $\chamber^\epsilon$ without hitting $\partial \chamber^\epsilon$. 
Instead of invoking a direct argument, though, we provide a proof that uses properties of the Dyson-Diriclet energy.
(We shall provide a direct argument for separately convex potentials later in \Cref{eqn:Delta_derivative} in the proof of \Cref{prop:limit_equally_spaced_zero}, when we characterize the long-time asymptotics of the solutions.) 

\begin{proof}[Proof of \Cref{prop:ODE_existence_uniqueness}]
By \Cref{rem:minimizers-exist}, for every $\bigT \ge 0$ there exists at least one minimizer $\btheta \in \BesselSpace{\bigT]}{\btheta_0}$ with $\nBessel_{\bigT}(\btheta) = 0$
--- thus, a solution to the ODEs~\eqref{eqn:ODE_system}. 
Since finite-energy drivers are continuous, from~\eqref{eqn: deterministic collision time} in \Cref{cor:finite-energy-characterization} we see that there exists $\epsilon_{\bigT} > 0$ such that $\btheta_t \in \chamber^{\epsilon_{\bigT}}$ for all $t \in [0, \bigT]$.
Now, as each $\newphi^j$ is Lipschitz on $\chamber^{\epsilon_{\bigT}}$, the uniqueness of the minimizer follows by applying the Picard-Lindel\"of theorem to two possible solutions to~\eqref{eqn:ODE_system} 
with the same initial configuration $\btheta_0 \in \chamber^{\epsilon_{\bigT}}$.
Since $T \ge 0$ was arbitrary, we can extend the solution for all times, to obtain the sought $\btheta \in \BesselSpace{\infty)}{\btheta_0}$.
\end{proof}

\begin{remark}\label{rem:spiral_existence_soln}
The above proof in fact also works for Dyson-Dirichlet energy with spiral (recalling \Cref{rem:finite-energy-characterization_spiral} and \Cref{rem:spiral_SLE}). 
We will return to this in \Cref{rem:spiral_minimizers}.
\end{remark}

\subsection{Zero-energy systems for symmetric separately convex potentials}
\label{subsec:zero-energy-existence}

The solution to the ODE system in \Cref{prop:ODE_existence_uniqueness} converges in the long run to some critical point of the potential $\DPotential$ 
(i.e., where $\newphi^j = -\partial_j \, \DPotential= 0$ for all $j$).
To investigate the possible long-time limits, we impose further conditions on the Dyson-type potential $\DPotential$, motivated by the application to Dyson Brownian motion.
To simplify the analysis, we will thus assume that $\DPotential$ is symmetric and separately (strictly) convex (\Cref{def:symmetric_Dyson_type_potential}):
\begin{align}\label{eqn:deriv_bound_def}\tag{{\textsc{sym,cvx}}}
\newphi^j= \sum_{\substack{1 \leq i \leq n \\[.1em] i\neq j}} \pairderiv(\theta^j - \theta^i) 
\quad \textnormal{for each $j$, and}
\qquad 
\constphi := - \sup_{\theta\in(0, 2\pi)} \pairderiv'(\theta) > 0,
\end{align}
where $\pairderiv \in C^1 (\chamberSingle\setminus\{0\}, \bR)$ is an odd function\footnote{In particular, we have $\pairderiv(\theta)=\pairderiv(2\pi-\theta)$, and $\theta \mapsto \pairderiv'(\theta)$ is an even function.} with $\pairderiv(\pi)=0$ and 
$\smash{\underset{\theta \to 0+}{\lim} \, \pairderiv(\theta) \in (0,+\infty]}$.

The constant $\constphi>0$ plays an important role in our analysis and appears in the exponential rate of convergence in \Cref{prop:limit_equally_spaced_zero} below,
as well as in the main \Cref{thm:U_equally_spaced}.

In fact, the results of this section hold just under the assumption that $\newphi^j$ satisfy \eqref{eqn:deriv_bound_def}; 
in other words, neither the existence of the potential itself, nor the asymptotics~\eqref{eqn:growth_conditions} or differential inequalities~\eqref{eq:differential_inequalities}, are needed to establish the results concerning the zero-energy and finite-energy systems themselves.
In particular, the potential, if exists, is not required to be bounded from below (which allows slightly more general potentials).

\begin{lem}\label{lem:phi-difference-estimate2}
Fix $\btheta \in \chamber$. Suppose $\newphi^j$ satisfy \eqref{eqn:deriv_bound_def}. 
Then, for any choice of an index $\jminzero \in \smash{\underset{1 \leq i \leq n}{\argmin} \, \big| \theta^{i+1}-\theta^i \big|}$, 
we have
\begin{align}
\label{eqn:phi-difference-estimate_min2}
\newphi^{\jminzero+1}(\btheta) - \newphi^\jminzero(\btheta) 
\ge \; & \constphi \Big(2\pi -n (\theta^{\jminzero+1}-\theta^\jminzero)\Big)\geq 0 ,
\end{align}
using the convention that $\theta^{n+1} = \theta^1+2\pi$ as in~\eqref{eq: torus ordering}, and $\theta^{0} = \theta^n-2\pi$.
\end{lem}

\begin{proof}
Rewriting the lefthand side of Equation~\eqref{eqn:phi-difference-estimate_min2} using~\eqref{eqn:sym_def_intro} yields
\begin{align} \label{eqn:phi-difference}
\newphi^{\jminzero+1}(\btheta) - \newphi^\jminzero(\btheta) 
= \; & \sum_{1\leq k \neq \jminzero+1 \leq n} \big( \pairderiv(\theta^{\jminzero+1} - \theta^{k}) - \pairderiv(\theta^{\jminzero} - \theta^{k-1}) \big) .
\end{align}
A key observation is that the sum is telescoping when $\theta^1, \ldots, \theta^n$ are equally-spaced on $[0, 2\pi)$; 
an idea that will be used again in the proof of \Cref{prop:limit_equally_spaced_zero}.  
Note first that 
\begin{align}\label{eqn:difference_est}
\pairderiv(u)-\pairderiv(v) \geq \constphi \, (v-u) > 0, \qquad 0 < u \leq v < 2\pi ,
\end{align}
where $\constphi> 0$ since $\pairderiv'<0$ by~\eqref{eqn:deriv_bound_def_intro}. 
The definition of $\jminzero$ as the minimizing index guarantees that $\theta^{\jminzero+1}-\theta^k\leq \theta^\jminzero-\theta^{k-1}$ for every $k$, allowing us to apply~\eqref{eqn:difference_est} with $u=\theta^{\jminzero+1}-\theta^k$ and $v= \theta^\jminzero-\theta^{k-1}$.
We can use this to bound the terms in~\eqref{eqn:phi-difference} as
\begin{align*}
\pairderiv \big(\theta^{\jminzero+1} - \theta^k\big) - \pairderiv \big( \theta^\jminzero - \theta^{k-1} \big) 
\geq \; & \constphi \, ( \theta^\jminzero - \theta^{k-1} ) - \constphi( \theta^{\jminzero+1} - \theta^k ) .
\end{align*}
Substituting these bounds into~\eqref{eqn:phi-difference}, then adding and subtracting $\constphi(\theta^{\jminzero+1}-\theta^\jminzero)$, 
we obtain
\begin{align*} 
\newphi^{\jminzero+1}(\btheta) - \newphi^\jminzero(\btheta) 
\geq \; &
\constphi  \, \underbrace{\sum_{k=1}^n  (\theta^k - \theta^{k-1} )}_{= \,2\pi} 
- \, n \, \constphi \, ( \theta^{\jminzero+1} - \theta^\jminzero ) .
\end{align*}
This gives the asserted inequality~\eqref{eqn:phi-difference-estimate_min2}. 
\end{proof}

The next result shows that, for any initial configuration, the zero-energy particle system eventually approaches a static equally-spaced configuration, moreover exponentially fast.

\begin{prop}\label{prop:limit_equally_spaced_zero}
Suppose $\newphi^j$ satisfy \eqref{eqn:deriv_bound_def}. 
If $\nBessel(\btheta) = 0$, then there exists $\zeta \in \bR$ such that 
\begin{align} \label{eq:equally-spaced}
\lim_{t\to \infty} \btheta_t = \big( \zeta , \, \zeta + \tfrac{2\pi}{n} , \, \ldots , \, \zeta+ \tfrac{(n-1)2\pi}{n} \big) ,
\end{align}
and the convergence is exponentially fast with exponential rate $\constphi n$,
with $\constphi$ from~\eqref{eqn:deriv_bound_def_intro}.
\end{prop}

In particular, we have $\constphi=\ONE$ for the multiradial partition function as in~\eqref{eq:DBM_choices}, so in this case the exponential convergence occurs with rate $n$.

From the assumption $\nBessel(\btheta) = 0$ and \Cref{def:multiradial_Dirichlet_energy}, 
we know that $\btheta$ satisfies the ODEs~\eqref{eqn:ODE_system} in \Cref{prop:ODE_existence_uniqueness}. 
In the below proof of \Cref{prop:limit_equally_spaced_zero}, we will as a byproduct also establish that for any $\epsilon > 0$, 
any solution to the ODEs~\eqref{eqn:ODE_system} with initial condition in $\chamber^\epsilon$ stays in $\chamber^\epsilon$ without hitting $\partial \chamber^\epsilon$. 
This gives another proof for the existence and uniqueness of solutions to the ODEs~\eqref{eqn:ODE_system}, under the assumptions \eqref{eqn:deriv_bound_def} on $\bnewphi$.

\begin{proof}
{\bf Step 1.} 
We will first show that all gaps between adjacent particles approach $2\pi/n$: 
\begin{align} \label{eq:little_d_conv}
d(\btheta_t) := \max_{1\leq j\leq n} \bigg| \frac{2\pi}{n} - ( \theta^{j+1}_t-\theta^j_t )  \bigg| 
\quad \overset{t\to \infty}{\longrightarrow} \quad 0, 
\end{align}
and this convergence happens exponentially fast at rate $\constphi n$.
In fact, for this it is actually sufficient to show that the \emph{smallest} gap approaches $2\pi/n$ as $t\to \infty$. 
Indeed, consider
\begin{align} 
\label{eq: zmin}
\zprocess_t := & \; \frac{2\pi}{n} - 
\min_{1\leq j\leq n} \big( \theta^{j+1}_t - \theta^j_t \big)
= \frac{2\pi}{n} - \DeltaMin(t) \geq 0 , \\
\label{eq: Zmax}
\Zprocess_t := & \; \max_{1\leq j\leq n} \big( \theta^{j+1}_t - \theta^j_t \big)
- \frac{2\pi}{n} \geq 0 ,
\end{align}
where $\DeltaMin(t) := \DeltaMin_{\btheta_t} = \smash{\underset{1\leq j\leq n}{\min} \, \big| \theta_t^{j+1} - \theta_t^j \big|}$. 
Using the ODEs~\eqref{eqn:ODE_system} and \Cref{lem:phi-difference-estimate2}, we infer that 
\begin{align} \label{eqn:der_pos}
\tfrac{\ud}{\ud t}\big(\theta^{\jmin+1}_t - \theta^{\jmin}_t \big)
= \; & \newphi^{\jmin+1}(\btheta_t) - \newphi^\jmin(\btheta_t)  \geq \constphi  \,(2\pi - n \DeltaMin(t) ) \geq 0 , 
\end{align}
for any index $\jmin \in \minindexset_t:= \underset{1 \leq k \leq n}{\argmin} \, \big| \theta^{k+1}_t-\theta^k_t \big|$. 
From this, we deduce that 
\begin{align} \label{eqn:Delta_derivative}
\tfrac{\ud}{\ud t} \DeltaMin(t) 
= \min_{\jmin\in \minindexset_t} \tfrac{\ud}{\ud t} \big( \theta^{\jmin+1}_t - \theta^{\jmin}_t \big) 
\geq 0 ,
\end{align}
so $t \mapsto \DeltaMin(t)$ is non-decreasing\footnote{Since it is continuous for all $t$ and differentiable almost everywhere, to argue that $t \mapsto \DeltaMin(t)$ is non-decreasing, it suffices to know that $\tfrac{\ud}{\ud t} \DeltaMin(t) \geq 0$ for each $t > 0$ where it exists.}.
(This also implies that any solution to~\eqref{eqn:ODE_system} with initial condition in $\chamber^\epsilon$ stays in $\chamber^\epsilon$ without hitting $\partial \chamber^\epsilon$, 
yielding another proof for \Cref{prop:ODE_existence_uniqueness}.)

Since the gaps $(\theta^{j+1}_t-\theta^j_t)$ sum up to $2\pi$, we see that $\Zprocess_t \leq (n-1)\zprocess_t$, and thus,
\begin{align} \label{eq:d_maxZz}
d(\btheta_t) = \max \{ \zprocess_t, \Zprocess_t\}
\leq (n-1) \zprocess_t.
\end{align} 

From~\eqref{eqn:Delta_derivative}, we deduce that 
for almost every $t \ge 0$, there exists an index $\jmin$ such that 
\begin{align*}
\tfrac{\ud}{\ud t} \zprocess_t 
= -  \big(\newphi^{\jmin+1}(\btheta_t) - \newphi^{\jmin}(\btheta_t)\big) \le -  \constphi \, n \zprocess_t
\qquad \Longrightarrow \qquad
\zprocess_t \le \zprocess_0 \,  e^{-  \constphi nt}.
\end{align*} 
Applying~\eqref{eq:d_maxZz}, we obtain a bound for every gap (not only the smallest one): 
\begin{align} \label{eq: exp rate}
\Big| \tfrac{2\pi}{n} - \big( \theta^{j+1}_t-\theta^{j}_t \big) \Big| 
\leq (n-1) \zprocess_t
\leq (n-1) \zprocess_0 \,  e^{-  \constphi nt} , \qquad j \in \{1, \ldots, n\} .
\end{align}
As $t \to \infty$, the righthand side approaches zero exponentially fast with exponential rate $\constphi n$,  yielding~\eqref{eq:little_d_conv} 
and concluding Step~1 of the proof. 

{\bf Step 2.} 
It remains to prove the convergence of $\btheta_t$ to the \emph{static} equally-spaced configuration~\eqref{eq:equally-spaced} as $t \to \infty$. 
Fix $\j \in \{1,\ldots, n\}$. We will show that $\smash{\underset{t\to \infty}{\lim} \, \theta^{\j}_t}$ exists.
From~\eqref{eq: exp rate}, 
\begin{align}\label{eqn:gapexpbound}
\tfrac{2\pi}{n} - (n-1) \zprocess_0 \, e^{- \constphi nt} \;\le\; \theta^{j+1}_t-\theta^{j}_t \;\le\; \tfrac{2\pi}{n} + (n-1)\zprocess_0\, e^{- \constphi nt} , \qquad j \in \{1, \ldots, n\} .
\end{align}
Fix $t_0$ such that $(n-1) \zprocess_0\, e^{-  \constphi nt_0} < \frac{2\pi}{n^2}$.
Then, applying~\eqref{eqn:gapexpbound} $k$ times, we finally obtain 
$(\theta^{\j}_t-\theta^{\j-k}_t) \in (0,\pi)$ and $(\theta^{\j+k}_t-\theta^{\j}_t) \in (0,\pi)$ for all $k \in \{1,2,\ldots,\lfloor (n-1)/2 \rfloor \}$ and $t \geq t_0$.
Next, using the property that $\pairderiv'<0$ from~\eqref{eqn:deriv_bound_def_intro},
 we see that for all $t\geq t_0$, 
\begin{align} \label{eq:cot_sandwich}
\pairderiv \big( \tfrac{2 k \pi}{n} + k(n-1)\zprocess_0\, e^{- \constphi n t} \big)
\;\leq\;\pairderiv (\theta^{\j}_t-\theta^{\j-k}_t) 
\;\leq\; \pairderiv \big( \tfrac{2 k \pi}{n} - k(n-1)\zprocess_0\, e^{- \constphi n t} \big) ,
\end{align}
and similarly for $\pairderiv \big(\theta^{\j+k}_t-\theta^{\j}_t\big)$. 
Since $\pairderiv$ is an odd function, if $n$ is odd, we thus obtain
\begin{align*}
\big| \tfrac{\ud}{\ud t}\theta^{\j}_t \big|  
\leq \; & 2\sum_{k=1}^{ (n - 1)/2 } \;\Big| 
\pairderiv (\theta^{\j}_t-\theta^{\j-k}_t) - \pairderiv (\theta^{\j+k}_t-\theta^{\j}_t) \Big|  \\
\leq \; & 2\sum_{k=1}^{(n - 1)/2 }
\Big| 
\pairderiv \big(\tfrac{2 k \pi}{n} - k (n-1)\zprocess_0 \, e^{- \constphi n t} \big)
- \pairderiv \big(\tfrac{2 k \pi}{n} + k(n-1)\zprocess_0 \, e^{- \constphi n t} \big) 
\Big|
&& \textnormal{[by~\eqref{eq:cot_sandwich}]}
\\
\leq \; &  4 \lipconst (n-1) \zprocess_0\, e^{- \constphi nt} \sum_{k=1}^{ (n - 1)/2 } k
\; = \; \frac{  \lipconst \, (n-1)^2(n-3)}{2} e^{- \constphi nt} 
\qquad \overset{t \to \infty}{\longrightarrow} \qquad 0.
\end{align*}
where $\lipconst := \underset{j,k \in \{1, \ldots, n\}}{\max} \, \underset{\btheta \in \chamberCl{\pi/n}}{\max} \, |\partial_k \newphi^j(\btheta)| < \infty$.  
Similarly, when $n$ is even, we obtain 
\begin{align*}
\big| \tfrac{\ud}{\ud t}\theta^{\j}_t \big|  
\leq \; & 2 \, \underbrace{\big|\pairderiv (\theta^{\j}_t-\theta^{\j - n/2}_t)\big|}_{\leq \; \lipconst \frac{n(n-1)}{2} e^{-  \constphi nt}}
\; + \; \underbrace{2\sum_{k=1}^{ (n/2) - 1} \big| \pairderiv (\theta^{\j}_t-\theta^{\j-k}_t)
- \pairderiv (\theta^{\j+k}_t-\theta^{\j}_t) \big|}_{\leq \; \lipconst \frac{(n-1)(n-2)(n-4)}{2} e^{- \constphi nt}}  \quad 
\overset{t \to \infty}{\longrightarrow} \quad 0 ,
\end{align*}
where the bound on $|\pairderiv (\theta^{\j}_t-\theta^{\j - n/2}_t)|$
in the last line comes from taking the absolute values in~\eqref{eq:cot_sandwich}.
(When $n=2$, there is no second term, since the upper index of summation is $0$.)
In particular, we see that there exists a universal constant $c \in (0,\infty)$ such that
\begin{align*}
\int_{t_0}^\infty|\tfrac{\ud}{\ud t}\theta^{\j}_t|\ud t \leq c n^3 e^{- \constphi n t_0} < \infty,
\end{align*}
which shows that, first of all, $\underset{t \to \infty}{\lim} \theta^{\j}_t$ exists and is given by~\eqref{eq:equally-spaced} for some $\zeta \in \bR$, 
and second of all, the convergence happens with exponential rate $\constphi n$.
This concludes the proof.
\end{proof}

\begin{remark}\label{rem:equilibrium}
Equations~(\ref{eqn:der_pos},~\ref{eqn:Delta_derivative}) in the proof imply that if $\DeltaMin(t) < 2\pi/n$, 
then $\tfrac{\ud}{\ud t} \DeltaMin(t) > 0$, while if $\DeltaMin(t) = 2\pi/n$, then $\DeltaMin(t)$ stays constant after time $t$, since $\tfrac{\ud}{\ud t} \DeltaMin(t) = 0$. 
\end{remark}

\begin{ex}[Multiradial energy with spiral]\label{rem:spiral_minimizers}
If $\nBessel^{\spr}_\bigT(\btheta) = 0$, then we get instead 
\begin{align*}
\tfrac{\ud}{\ud t}\zprocess_t = 
-\big(\big( \newphinrad ^{j+1}(\btheta_t) + 2 \spr\big) - ( \newphinrad^{j}(\btheta_t) + 2 \spr)\big) \;=\; -  \big(\newphinrad^{j+1}(\btheta_t) - \newphinrad^j(\btheta_t)\big) \;\le\; - n \zprocess_t,
\end{align*}
with $\newphinrad^{j}=2\phi^j$ as in~\eqref{eq:DBM_choices} and $\constphi = \ONE$ in this example.
Hence, the differential equation
\begin{align*}
\tfrac{\ud}{\ud t} \theta^j_t = 2 \phi^j(\btheta_t) + 2 \spr
\end{align*} 
has a unique solution for each initial configuration $\btheta_0 \in \chamber$, and it satisfies
\begin{align*}
\lim_{t \to \infty} (\btheta_t - 2 \spr t) = \big(\zeta, \zeta + \tfrac{2\pi}{n}, \dots, \zeta + \tfrac{(n-1)2\pi}{n}\big),
\end{align*}
for some $\zeta \in \bR$, where the convergence is exponentially fast with exponential rate $n$.
\end{ex}

\subsection{Finite-energy systems for symmetric separately convex potentials}
\label{subsec:Finite-energy-results}

We will now show that any function with finite Dyson-Dirichlet energy converges to an equally-spaced system in the long run. 
However, if the energy is non-zero, it is possible that the convergence rate is very slow and that the system continues slow rotation for all time. 
(Compare to \Cref{prop:limit_equally_spaced_zero} for zero-energy systems, and see \Cref{rmk:ncopies_of_finite_energy}.)

\begin{prop}\label{prop:finite_energy}
Consider a function $\btheta \in \BesselSpace{\infty)}{\btheta_0}$. 
Suppose $\newphi^j$ satisfy \eqref{eqn:deriv_bound_def}. 
If $\nBessel(\btheta) < \infty$, then we have
\begin{align} \label{eq: gaps vanish}
\lim_{t \to \infty}(\theta^{j+1}_t-\theta^j_t) = \frac{2\pi}{n} , \qquad \textnormal{for all } j \in \{1, \ldots, n\}.
\end{align}
\end{prop}

Thus, for finite-energy systems the points $e^{\ii  \theta^1_t}, \ldots, e^{\ii \theta^n_t}$ eventually approach equal spacing around the circle
--- but, in contrast to \Cref{prop:limit_equally_spaced_zero}, it is not true that a system with finite energy necessarily converges to a \emph{static} equally-spaced configuration:

\begin{ex}\label{rmk:ncopies_of_finite_energy} 
Consider the system defined by $n$ equally-spaced copies of a single driver~$\theta$, 
\begin{align*} 
\btheta_t = \; & \big( \theta_t, \, \theta_t + \tfrac{2\pi}{n}, \, \ldots , \, \theta_t + \tfrac{(n-1)2\pi}{n} \big) , \qquad t \geq 0 , 
\end{align*}
so that since $\hat\newphi$ is odd, the functions $\newphi^j$ vanish for all $j \in \{1, \ldots, n\}$:
\begin{align*} 
\newphi^j(\btheta_t) = \; & \sum_{\substack{1 \le i \le n \\ i \ne j}}\hat\newphi(\theta^j-\theta^i)=0 , \qquad  t \geq 0 .
\end{align*}
If $\theta \in \BesselSpaceSingle{\infty)}{\theta_0}$ has finite Dirichlet energy 
$\nDenergy(\theta)<\infty$, then $\btheta$ has finite multiradial Dirichlet energy:
\begin{align*} 
\nBessel(\btheta) 
= \frac{1}{2} \int_0^\infty \sum_{j=1}^n \big| \tfrac{\ud }{\ud s} \theta^j_s - \newphi^j(\btheta_s) \big|^2 \, \ud s
= \frac{1}{2} \int_0^\infty \sum_{j=1}^n \big| \tfrac{\ud }{\ud s} \theta^j_s \big|^2 \, \ud s
= n\, \nDenergy(\theta) <\infty.
\end{align*}
However, this system may slowly spiral, for example if $\theta_t = \log (t+1)$.
\end{ex}

\begin{proof}[Proof of \Cref{prop:finite_energy}]
As in the proof of \Cref{prop:limit_equally_spaced_zero}, now under the assumption that $\nBessel(\btheta) < \infty$, 
the asserted limit \eqref{eq: gaps vanish} will follow by proving the convergence~\eqref{eq:little_d_conv} (in this case with unspecified rate).
By the observation in Equation~\eqref{eq:d_maxZz}, it actually suffices to show that the quantity $\zprocess_t$ defined in~\eqref{eq: zmin} approaches zero as $t \to \infty$.
To this end, we will first show that $\nBessel(\btheta) < \infty$ implies that
\begin{align} \label{claim:finite_meas}
\Dset (r) := \big\{ t \in [0,\infty) \;|\; \zprocess_t \geq r \big\} \quad \textnormal{ has finite Lebesgue measure for any $r > 0$.}
\end{align}
Thereafter, we will show that if $\nBessel(\btheta) < \infty$, then $\Dset(r)$ is a bounded set for every $r > 0$. 
This is equivalent with $\underset{t \to \infty}{\limsup} \zprocess_t \le r$ for every $r > 0$. 
As $\zprocess$ is non-negative, by taking $r \to 0$ we may then conclude that $\underset{t \to \infty}{\lim} \zprocess_t = 0$, as desired.

Fix $r > 0$. On the one hand, 
\Cref{lem:phi-difference-estimate2} implies that if $\zprocess_t \geq r$, then $\smash{2 \, \underset{1\leq j\leq n}{\max}  | \newphi^j(\btheta_t) | \geq \constphi n r}$. 
On the other hand, the triangle inequality yields
\begin{align}\label{eqn:bd_phi}
2 \max_{1\leq j\leq n}  | \newphi^j(\btheta_t) | 
\; \leq \; 
\max_{1\leq j\leq n}  \big| \tfrac{\ud}{\ud t} \theta^j_t -  \newphi^j(\btheta_t) \big|  
\; + \; \max_{1\leq j\leq n}  \big| \tfrac{\ud}{\ud t} \theta^j_t \big| .
\end{align}
Hence, if $\zprocess_t \geq r$, then at least one term on the righthand side of~\eqref{eqn:bd_phi} is greater than or equal to $\constphi \frac{n}{2} r$.
This allows us to bound the Lebesgue measure $\meas{\Dset (r)}$ of the set $\Dset(r)$ as
\begin{align} \label{eqn:bd_meas}
\meas{\Dset (r)} \leq \; & \meas{S} + \meas{R} , \\[.5em]
\nonumber
\textnormal{where}
\qquad
S := 
\Big\{ t \in [0,\infty) \;\big|\; \; &  \max_{1\leq j\leq n} | \tfrac{\ud}{\ud t} \theta^j_t -  \newphi^j(\btheta_t) |  \geq \constphi \tfrac{n}{2} r \Big\} , \\
\nonumber
R := \Big\{ t \in [0,\infty) \;\big|\; \; &  \max_{1\leq j\leq n} |  \tfrac{\ud}{\ud t} \theta^j_t | \geq \constphi \tfrac{n}{2} r \Big\} .
\end{align}
To bound the righthand side of~\eqref{eqn:bd_meas}, we note that each term on the righthand side of~\eqref{eqn:bd_phi} is square-integrable (for all time):
\begin{align*} 
\int_0^\infty \Big( \max_{1\leq j\leq n} \big| \tfrac{\ud}{\ud s} \theta^j_s -  \newphi^j(\btheta_s) \big| \Big)^2 \, \ud s
\leq \; &
\int_0^\infty \sum_{j=1}^n \big| \tfrac{\ud}{\ud s} \theta^j_s -  \newphi^j(\btheta_s) \big|^2 \, \ud s 
= 2\nBessel(\btheta) , \\
\int_0^\infty \Big(\max_{1\leq j\leq n} \big| \tfrac{\ud}{\ud s} \theta^j_s \big| \Big)^2 \, \ud s
\leq \; &
\int_0^\infty \sum_{j=1}^n \big| \tfrac{\ud}{\ud s} \theta^j_s \big|^2 \, \ud s 
= 2 \nDenergy(\btheta) ,
\end{align*}
so that
\begin{align*}
\meas{\Dset (r)} \leq 
\frac{8}{\constphi^2 n^2r^2} \, \big( \nBessel(\btheta) + \nDenergy(\btheta)  \big) < \infty ,
\end{align*}
as $\nBessel(\btheta)<\infty$ by assumption and $\nDenergy(\btheta)<\infty$ by \Cref{cor:infinite-energy-characterization}.
This verifies~\eqref{claim:finite_meas}.

Next, suppose $\Dset(r)$ is unbounded. Then, there exists and a sequence $(t_{k})_{k \in \bN}$ such that $t_{k} \smash{\overset{k \to \infty}{\longrightarrow}} \infty$ 
and $\zprocess_{t_{k}} \ge r$ for all $k$.
Since $\meas{\Dset (\frac{r}{2})} < \infty$, we may assume (passing to a subsequence if necessary) that 
on each interval $(t_{k}, t_{k+1})$, the function $\zprocess_t$ exits $\Dset(\frac{r}{2})$. 
Set 
\begin{align*}
s_{k} := \max  \big\{ 0 \leq t \leq t_{k} \; \colon \; \zprocess_t = \tfrac{r}{2} \big\} , \qquad k \in \bN .
\end{align*}
Since the set $\Dset(\frac{r}{2})$ has finite Lebesgue measure, the length of the intervals $(s_{k}, t_{k}]$ approaches zero 
as $k\to \infty$, so for any $\epsilon>0$ we can find an index $k_\epsilon$ such that
\begin{align*}
| t_{k_\epsilon}-s_{k_\epsilon} | < \epsilon.
\end{align*}
By construction, for every $k \in \bN$ and 
$i_k \in \underset{1 \leq j \leq n}{\argmin} \, \big| \theta^{j+1}_{t_{k}}-\theta^j_{t_{k}} \big|$ 
we also have 
\begin{align*}
\tfrac{r}{2} \;\le\;  \zprocess_{t_{k}} - \zprocess_{s_{k}} 
= \; & \min_{1\leq j\leq n} \big( \theta^{j+1}_{s_{k}}  - \theta^j_{s_{k}}  \big)
- \min_{1\leq j\leq n} \big( \theta^{j+1}_{t_{k}} - \theta^j_{t_{k}} \big)
\\
\leq \; & \big| \theta^{i_k+1}_{s_{k}} - \theta^{i_k+1}_{t_{k}} \big|
+ \big| \theta^{i_k}_{t_{k}} - \theta^{i_k}_{s_{k}} \big| .
\end{align*}
Hence, we see that there exists an index $j$ such that 
$\big| \theta^j_{t_{k_\epsilon}} - \theta^j_{s_{k_\epsilon}} \big| \geq r/4$.
We thus obtain (using also the Cauchy–Schwarz inequality)
\begin{align*}
+\infty 
\; > \; \nDenergy(\btheta) 
\; \geq \;  \frac{1}{2} \int_{s_{k_\epsilon}}^{t_{k_\epsilon}} \big|\tfrac{\ud}{\ud u}\theta^j_u\big|^2 \ud u 
\; \geq \; \frac{\big| \theta^j_{t_{k_\epsilon}} - \theta^j_{s_{k_\epsilon}} \big|^2}{2 |t_{k_\epsilon}-s_{k_\epsilon}|} 
\; \geq \; \frac{r^2}{32\epsilon} 
\quad \overset{\epsilon \to 0}{\longrightarrow} \quad +\infty ,
\end{align*}
which is a contradiction. 
This shows that $\Dset(r)$ is bounded for every $r > 0$ 
and, in particular, that~\eqref{eq: gaps vanish} holds --- and the proof is complete. 
\end{proof}

\begin{remark}\label{rem:conv_optimal} 
In contrast with \Cref{prop:limit_equally_spaced_zero}, finite-energy systems do not necessarily enjoy an exponential rate of convergence to the equally-spaced configuration~\eqref{eq: gaps vanish}. 
In~fact, as the next \Cref{ex:poly_conv} shows, it is possible to construct systems of arbitrarily small energies with polynomial convergence rates. 
\end{remark}

\begin{ex}\label{ex:poly_conv}
Let us consider the case of two drivers, $n = 2$. Let $f \colon [0,\infty) \to \bR$ be a continuous $L^2$-function, and suppose $\theta^1, \theta^2$ satisfy the differential equations
\begin{align*}
\begin{cases}
\tfrac{\ud}{\ud t}\theta^1_t = 2\cot\Big(\frac{\theta^1_t-\theta^2_t}{2}\Big) - f(t),\\
\tfrac{\ud}{\ud t}\theta^2_t = 2\cot\Big(\frac{\theta^2_t-\theta^1_t}{2}\Big) ,
\end{cases}
 \qquad t \geq 0 ,
\end{align*}
with initial configuration $\btheta_0 = (\theta^1_0, \theta^2_0) = (0, \pi)$ and potential as in~\eqref{eq:DBM_choices}.
Then, the function $\btheta = (\theta^1, \theta^2) \in  \BesselSpace{\infty)}{\btheta_0}$ has multiradial Diriclet energy
\begin{align*}
\nBessel_{\bigT}(\btheta) = \frac{1}{2}\int_0^\bigT f(s)^2 \ud s,
\end{align*}
and $u_t := (\theta^2_t-\theta^1_t)-\pi$ satisfies the differential equation
\begin{align*}
\tfrac{\ud}{\ud t}u_t = 4\cot\bigg(\frac{u_t+\pi}{2}\bigg) + f(t) , 
\qquad 
\textnormal{with initial configuration } 
u_0 = 0.
\end{align*}
Note that $d(\btheta_t) \ge |u_t|$. 
Consider a function $v \colon [0,\infty) \to \bR$ satisfying the IVP
\begin{align*}
\tfrac{\ud}{\ud t}v_t = -4v_t + f(t), 
\qquad 
v_0 = 0.
\end{align*}
Since $x \mapsto \cot(\frac{x+\pi}{2})$ is $1$-Lipschitz on $[0,\frac{\pi}{2}]$,  we have $u_t \ge v_t$ 
for all times before $v_t$ exits the interval $[0,\frac{\pi}{2}]$. 
Using the integrating factor $e^{-4t}$, we find that the solution is
\begin{align*}
v_t =  e^{-4t} \bigg( 1 + \int_0^t e^{4s} f(s)\, \ud s 
\bigg).
\end{align*}
Choosing $f(t) = \frac{\varepsilon}{t+1}$ for $\varepsilon \in (0,\frac{\pi}{2})$ gives rise to the function 
\begin{align*}
v_t =  e^{-4t} \bigg(1 + \varepsilon \int_0^t \frac{e^{4s}}{s+1} \, \ud s \bigg) ,
\end{align*}
which never exits $[0,\frac{\pi}{2}]$, since it satisfies 
\begin{align*}
e^{-4t} +\, \frac{\varepsilon(1-e^{-4t})}{4(t+1)} \leq v_t \leq e^{-4t} + \frac{\varepsilon}{4} (1-e^{-4t}).
\end{align*}
Thus, we find that 
\begin{align*}
d(\btheta_t) \, \ge \, u_t \, \ge \, v_t \, \ge \, \varepsilon \, \frac{1-e^{-4t}}{4(t+1)} \, = \, O(t^{-1}) .
\end{align*}
This gives a polynomial lower bound for the convergence rate 
to the equally-spaced configuration~\eqref{eq: gaps vanish}
for the system $\btheta = (\theta^1, \theta^2)$ having energy
\begin{align*}
\nBessel(\btheta) =\;& \int_0^\infty \frac{\varepsilon^2}{2(t+1)^2} \ud t = \frac{\varepsilon^2}{2},
\end{align*}
which can be made arbitrarily small by taking $\varepsilon \to 0$.
\end{ex}

Finally, we treat the convergence rate for systems with (locally) finite energy. 
For zero-energy systems, we recover the exponential rate of convergence of~\eqref{eq:little_d_conv} from the proof of \Cref{prop:limit_equally_spaced_zero}.

\begin{prop}\label{prop:finite_time_finite_energy_energy_slowconv}
Consider $\btheta \in \BesselSpace{\infty)}{\btheta_0}$. 
Suppose $\newphi^j$ satisfy \eqref{eqn:deriv_bound_def}. 
If $\nBessel_\bigT(\btheta) < \infty$ for every $\bigT \ge 0$, then 
\begin{align}\label{eq:equally-spaced-lower-estimate_d}
d(\btheta_t) \leq \; & (n-1) e^{-  \constphi nt}\bigg(2\sqrt{2}\int_0^t e^{  \constphi ns}\sqrt{\partial_s \nBessel_s(\btheta)}\ud s + d(\btheta_0)\bigg) , \qquad t \geq 0 ,
\\
\nonumber
\textnormal{where}
\qquad 
d(\btheta_t) := \; & \max_{1\leq j\leq n} \bigg| \frac{2\pi}{n} - ( \theta^{j+1}_t-\theta^j_t )  \bigg| .
\end{align}
\end{prop}

\begin{proof}
We will show the slightly stronger claim for $\zprocess_t$~\eqref{eq: zmin} that 
\begin{align}\label{eq:equally-spaced-lower-estimate}
\zprocess_t \le e^{-  \constphi nt}\bigg(2\sqrt{2}\int_0^t e^{  \constphi ns}\sqrt{\partial_s \nBessel_s(\btheta)}\ud s + \zprocess_0\bigg) , \qquad t \geq 0 .
\end{align}
The asserted bound~\eqref{eq:equally-spaced-lower-estimate_d} then follows from~\eqref{eq:d_maxZz}.

Similarly as in~\eqref{eqn:Delta_derivative}, we deduce that for almost all times $t$, we have
\begin{align} \label{eq:zt-differential-equation}
\tfrac{\ud}{\ud t} \zprocess_t 
= \; & 
\big( \tfrac{\ud}{\ud t} \theta^{\jmin}_t - \newphi^{\jmin}(\btheta_t) \big) 
- 
\big( \tfrac{\ud}{\ud t} \theta^{\jmin+1}_t - \newphi^{\jmin+1}(\btheta_t) \big) 
-  \big( \newphi^{\jmin+1}(\btheta_t) - \newphi^{\jmin}(\btheta_t) \big) 
\end{align}
for some $\jmin \in \underset{1 \leq i \leq n}{\argmin} \, \big| \theta^{i+1}_t-\theta^i_t \big|$. 
It follows from \Cref{lem:phi-difference-estimate2} that
\begin{align*}
 \big( \newphi^{\jmin+1}(\btheta_t) - \newphi^{\jmin}(\btheta_t) \big) \geq    \constphi n \zprocess_t ,
\end{align*}
and 
\begin{align*}
|\tfrac{\ud}{\ud t}\theta^{k}_t - \newphi^{k}(\btheta_t)| \le \Big( \sum_{i=1}^n \big|\tfrac{\ud}{\ud t} \theta^i_t - \newphi^i(\btheta_t)\big|^2 \Big)^{1/2} 
= \sqrt{2}\sqrt{\partial_t \nBessel_t(\btheta)} , \qquad k \in \{1,\ldots,n\} .
\end{align*}
Plugging these back to~\eqref{eq:zt-differential-equation} yields
\begin{align*}
\tfrac{\ud}{\ud t} \zprocess_t \le -  \constphi n \zprocess_t + 2\sqrt{2}\sqrt{\partial_t \nBessel_t(\btheta)},
\end{align*}
which implies~\eqref{eq:equally-spaced-lower-estimate} and concludes the proof.
\end{proof}

\bibliographystyle{alpha}

\end{document}